\documentclass[a4paper,reqno]{amsart}
\usepackage{verbatim} 
\usepackage{tikz}
\usepackage{tikz-cd}

\usepackage{rotating}

\usepackage{pstricks,pst-coil}
\usepackage{amsthm} 
\usepackage{graphics,graphicx}
\usepackage[ps,dvips,color,poly,necula,all]{xy}
\usepackage{stmaryrd}
\usepackage{latexsym}
\usepackage{amssymb}
\usepackage{amsmath}
\usepackage{extarrows}
\usepackage{upgreek}
\usepackage{mathrsfs} 
\usepackage{enumerate}
\usepackage{etoolbox}

\numberwithin{equation}{section}

\AtBeginEnvironment{figure}{\stepcounter{equation}}
\newtheorem{theorem}[equation]{Theorem}

\newtheorem{proposition}[equation]{Proposition}
\newtheorem{lemma}[equation]{Lemma}

\newtheorem{assumption}[equation]{Assumption}

\theoremstyle{definition} 
\newtheorem{definition}[equation]{Definition}
\newtheorem{remark}[equation]{Remark}
\theoremstyle{remark} 
\newtheorem{example}[equation]{Example}

\renewcommand{\and}{\quad\text{and}\quad}

\newcommand{\resp}{\text{resp. }}
\renewcommand{\for}{\text{ for }}
\newcommand{\ZZ}{\mathbb{Z}}

\newcommand{\NN}{\mathbb{N}}

\newcommand{\path}[2]{\xymatrix{#1 \ar @{~>}[r] & #2}}

\newcommand{\ca}{{\mathcal A}}

\newcommand{\Dr}{{\operatorname{D}}_1}
\newcommand{\Dk}{{\operatorname{D}}_0}

\renewcommand{\mod}{\mathop{\operatorname{mod}}}
\DeclareMathOperator{\umod}{\underline{mod}}
\DeclareMathOperator{\proj}{proj}
\DeclareMathOperator{\simp}{sim}
\DeclareMathOperator{\add}{add}
\DeclareMathOperator{\ind}{ind}

\DeclareMathOperator{\im}{Im}
\DeclareMathOperator{\Tr}{Tr}
\DeclareMathOperator{\soc}{soc}

\DeclareMathOperator{\Id}{Id}
\DeclareMathOperator{\rad}{rad}
\DeclareMathOperator{\corad}{corad}

\DeclareMathOperator{\tp}{top}
\DeclareMathOperator{\cotop}{cotop}

\DeclareMathOperator{\Hom}{Hom}
\DeclareMathOperator{\uHom}{\underline{Hom}}
\DeclareMathOperator{\Ext}{Ext}
\DeclareMathOperator{\End}{End}
\DeclareMathOperator{\uEnd}{\underline{End}}
\DeclareMathOperator{\Aut}{Aut}

\DeclareMathOperator{\inj}{inj}
\DeclareMathOperator{\fl}{fl}

\DeclareMathOperator{\CM}{CM}
\DeclareMathOperator{\uCM}{\underline{CM}}

\DeclareMathOperator{\supp}{supp}
\DeclareMathOperator{\Irr}{Irr}
\DeclareMathOperator{\len}{length}

\newcommand{\shift}[2]{\ifstrequal{#2}{1}{\Sigma #1}{\Sigma^{#2}#1}}

\title[$0$-CY Configurations and Finite AR Quivers of Gorenstein Orders]{$0$-Calabi-Yau Configurations and Finite Auslander-Reiten Quivers of Gorenstein Orders}
\author[X. Luo]{Xueyu Luo}
\begin{document}
\maketitle
\begin{abstract}
We revisit Wiedemann's classification \cite{Wiedemannorder} of Auslander-Reiten quivers of representation-finite Gorenstein orders in terms of a Dynkin diagram, a configuration and an automorphism group. 
In this paper, we introduce the notion of $2$-Brauer relations and prove that Wiedemann's configurations are simply described in terms of $2$-Brauer relations.
We also give a simple self-contained proof of Wiedemann's classification.
\end{abstract}

\section{Introduction}
Representation theory of orders is classical and has been studied by a lot of authors, see e.g. \cite{FMO,RT,CR2,greenreiner,Reinermaximal,roggenkamp882, RoHulattices, RW05, POS}.
It deals with the categories of Cohen-Macaulay modules (lattices) on orders.
One of the traditional subjects is the classification of orders which are representation-finite in the sense that they have only finitely many indecomposable Cohen-Macaulay modules.
The classification was given for many important classes, e.g. commutative orders \cite{DR67, GK85, HJ67} (see also \cite[Chapter 9]{CM}), local orders \cite{DK2,HN3}, tiled orders \cite{POS,ZK} and B\"ackstr\"om orders \cite{RRorder}.
The important class of Gorenstein orders is an analogue of both finite dimensional self-injective algebras and commutative Gorenstein rings. 
The representation-finite Gorenstein orders were studied by Wiedemann \cite{Wiedemannorder} and Roggenkamp \cite{Ro86} (see also \cite{Wie86local, Wie86An}).
Wiedemann's work is an analogue of Riedtmann's work on representation-finite selfinjective algebras \cite{BLR81,Riedtmannstructure, riedtmannA, riedtmannD}.

One of the aims of this paper is to improve Wiedemann's classification of Auslander-Reiten quivers of representation-finite Gorenstein orders.
A key notion introduced by Wiedemann is configurations, which is similar but different from Riedtmann's configurations \cite{riedtmannA} appearing in classification of representation-finite self-injective algebras. He described the Auslander-Reiten quiver of a representation-finite Gorenstein order in terms of a Dynkin diagram $\Delta$, a configuration $C$ in $(\ZZ\Delta)_0$ and an automorphism group $G$ of $\ZZ\Delta$. 
In this paper, we give a simple proof of the theorem below of Wiedemann \cite{Wiedemannorder} (see Section  \ref{G-def:translation}, Definitions \ref{G-def:recover} and \ref{G-def:configuration} for $(\ZZ\Delta/G)_C$). 

\begin{theorem}[{Theorem \ref{G-thm:mainthm1}}]
\label{G-th:mainthintro}
Let $R$ be a complete discrete valuation ring and $\Lambda$ be a ring-indecomposable representation-finite Gorenstein $R$-order satisfying the following conditions
             \begin{itemize}
                    \item[(a)] $\rad P$ is indecomposable and non-projective for any indecomposable projective $\Lambda$-module $P$;
                    \item[(b)] $\rad P$ is not isomorphic to $\rad Q$ when the two indecomposable projective $\Lambda$-modules $P, Q$ are nonisomorphic. 
             \end{itemize}
Then the Auslander-Reiten quiver $\mathfrak{A}(\Lambda)$ of $\CM\Lambda$ is isomorphic to $(\ZZ\Delta/G)_ C,$ where $\Delta$ is a Dynkin diagram, $G\subset \Aut(\ZZ\Delta)$ is a weakly admissible group and $ C$ is a configuration of $\ZZ\Delta/G$.
\end{theorem}

There are very few cases where the conditions (a) and (b) are not satisfied. 
In fact, as is shown by Roggenkamp \cite[Remark 1]{Ro86}, $\Lambda$ is classified in this case as a special B\"ackstr\"om-order \cite{RRorder}.
Note that Wiedemann proved a partial converse of Theorem \ref{G-th:mainthintro}.

We point out that Wiedemann's configurations are {\textquotedblleft}{$0$-Calabi-Yau}{\textquotedblright} in the sense that they are preserved by the Serre functor $S$ in the stable category $\uCM\Lambda$, while Riedtmann's configurations are {\textquotedblleft}{$(-1)$-Calabi-Yau}{\textquotedblright} since they are preserved by $S \circ [1]$.
Note that a variant of Wiedemann's and Riedtmann's configurations appeared recently in cluster tilting theory, where $n$-cluster tilting objects in $n$-cluster categories correspond bijectively to $n$-Calabi-Yau configurations in the derived categories \cite{BMRRT} (see \cite[Section 4.2.1]{Iyama07} for details).
\begin{center}
\begin{tabular}{|c|c|c|}
\hline
Riedtmann's configuration & Wiedemann's configuration &$n$-cluster tilting object ($n\ge 2$)\\
\hline
$(-1)$-Calabi-Yau &$0$-Calabi-Yau &$n$-Calabi-Yau\\
\hline
\end{tabular}
\end{center}

For type $A$ case, Wiedemann described configurations in terms of Brauer relations with \emph{Stra{\ss}eneigenschaft} \cite[Satz, p. 47]{Wiedemannorder}.
In this paper, we simplify his description by introducing a much simpler notion of \emph{$2$-Brauer relations} and give nice correspondences with Wiedemann's configurations:
Let $n \ge 1$ be an integer.
We denote by $\mathbf{B}_n^2$ the set of all $2$-Brauer relations of rank $n$ (see Definition \ref{G-def:Brauer}) and by $\mathbf{C}(A_{n+1})$ the set of configurations of $\ZZ A_{n+1}$.
Then our first result is the following.
\begin{theorem}[{Theorem \ref{G-thm:correspondenceA}}]
\label{G-th:mainAintro}
We have the following one-to-one correspondence
$$ \mathbf{C}(A_{n+1})  \xlongleftrightarrow{1-1}  \mathbf{B}_n^2. $$
\end{theorem}

In order to describe all the configurations for other Dynkin diagrams, we also define \emph{symmetric $2$-Brauer relations} of rank $2n$ and \emph{crossing $2$-Brauer relations} of rank $2n$, and we denote by $\mathbf{B}_{2n}^{2,s}$ and $\mathbf{B}_{2n}^{2,c}$ the sets of these two kinds of Brauer relations respectively (see Definitions \ref{G-def:symmetricBrauer} and \ref{G-def:crossingBrauer}).
Consider  the set $\mathbf{C}(B_{n+1})$ of configurations of $\ZZ B_{n+1}$ and the set $\mathbf{C}(C_{n+1})$ of configurations of $\ZZ C_{n+1}$.
For $\ZZ D_{n+2}$, we consider a partition $\mathbf{C}( D_{n+2})=\mathbf{C}^1(D_{n+2}) \sqcup \mathbf{C}^2(D_{n+2})$ (see Section \ref{G-s:typeD} for details).
Then we give theorems to describe all the configurations.
\begin{theorem}[{Theorems \ref{G-thm:correspondenceB}, \ref{G-thm:correspondenceC}, \ref{G-thm:correspondenceD} and \ref{G-th:exceptional}}]
\label{G-th:mainBCDintro}
We have the following one-to-one correspondences
 \begin{align*}
   \mathbf{C}(B_{n+1})  &\xlongleftrightarrow{1-1}  \mathbf{B}_{2n}^{2,s}, \\
 \mathbf{C}(C_{n+1})  &\xlongleftrightarrow{1-1}  \mathbf{B}_{2n}^{2,s}, \\
  \mathbf{C}^1(D_{n+2})  &\xlongleftrightarrow{1-1}  \mathbf{B}_{2n}^{2,s}, 
 \end{align*}
 and a two-to-one correspondence 
  \begin{align*} 
  \mathbf{C}^2(D_{n+2})  \xlongleftrightarrow{2-1}  \mathbf{B}_{2n}^{2,c}.
   \end{align*} 
   For types $E$, $F$ and $G$, we have a complete list of configurations (see Section \ref{G-s:appendix}).
   \end{theorem}
   
   Using this and computer program, we show the following theorem.
   \begin{theorem}[{Propositions \ref{pro:cardinalA}, \ref{pro:cardinalBC}, \ref{pro:cardinalD} and Section \ref{G-s:appendix}}]
   The number of configurations is given by the following table,
    \begin{center}
\begin{tabular}{|c|c|c|c|c|c|c|c|c|}
\hline
$A_{n+1}$&$B_{n+1}$& $C_{n+1}$ &$D_{n+2}$& $E_6$ & $E_7$ &$E_8$ & $F_4$ &$G_2$\\
\hline
$M(n)$&$M^s(n)$& $M^s(n)$ &$M^s(n)+M^c(n)$&$77$&$346$ &$1892$ &$25$ &$4$\\
\hline
\end{tabular}
\end{center}
where $M(n)$ is given by the recursive formula 
$$M(0)=1, M(1)=1, M(n)=M(n-1)+ \sum_{i=0}^{n-2}M(i) M(n-2-i) \for n\ge2,$$ 
the number $M^s(n)$ is given by the recursive formula 
$$M^s(0)=1, M^s(1)=2, M^s(n)=M^s(n-1)+M(n-1)+ 2\sum_{i=0}^{n-2}M(i) M^s(n-2-i) \for n\ge2,$$ 
the number $M^c(n)$ is given by the recursive formula 
$$M^c(0)=1, M^c(1)=1, M^c(n)= \sum_{i=1}^{n-1} \sum_{j=i+1}^{n}M(j-i-1) M(n+i-j-1) \for n\ge2.$$
\end{theorem}
By calculation, we have the following list of values of $M(n)$, $M^s(n)$ and $M^c(n)$ for $0 \le n \le 10$.
\begin{center}
\begin{tabular}{|c|c|c|c|c|c|c|c|c|c|c|c|}
\hline
&$0$&$1$& $2$ &$3$& $4$ & $5$ &$6$ & $7$ &$8$& $9$ &$10$\\
\hline
$M(n)$&$1$& $1$ &$2$&$4$&$9$ &$21$ &$51$ &$127$&$323$&$835$&$2188$\\
\hline
$M^s(n)$&$1$& $2$ &$5$&$13$&$35$ &$96$ &$267$ &$750$&$2123$&$6046$&$17303$\\
\hline
$M^c(n)$&$1$& $1$ &$1$&$3$&$10$ &$30$ &$90$ &$266$&$784$&$2304$&$6765$\\
\hline
\end{tabular}
\end{center}
By an immediate induction, we have the following formulas (see \cite{Motzkin1, Motzkin}).
\begin{align*}
M(n) & =\sum_{k=0}^{\lfloor n/2 \rfloor} \frac{n!}{(n-2k)!(k+1)! k!}, \quad 
M^s(n)=\sum_{k=0}^{\lfloor(n+1)/2\rfloor} \frac{n!(n+1-k)!}{k!(n-k)!k!(n+1-2k)!}, \\
M^c(n) &= \sum_{k=0}^{\lfloor (n-2)/2 \rfloor} \frac{n!}{k! (k+2)!(n-2-2k)!}.
\end{align*}

Before explaining organization of this paper, we explain our strategy to prove Theorem \ref{G-th:mainthintro}. It consists of two parts.
\begin{itemize}
 \item[(A)] We show that $C_\Lambda =\{ \rad P \mid P \in \ind(\proj \Lambda)\}$ is a ``categorical configuration" in the sense that two categorical conditions (Theorem \ref{G-pro:c1} and Theorem \ref{G-pro:c2}) are satisfied.
 
 \item[(B)] By using general results to calculate hom-sets from Auslander-Reiten quivers given in \cite{tauone}, we show that $C_\Lambda$ is a ``combinatorial configuration'' in the sense that two combinatorial conditions (C1) and (C2) (see Definition \ref{G-def:configuration}) are satisfied.
\end{itemize}
Note that our argument is different from Wiedemann's argument since he defined configurations in terms of mesh categories of translation quivers and used his results on existence of covering functors \cite{Wiedemanncovering}.

We organize this paper as follows. In Section \ref{G-s:pre}, we recall some basic definitions and facts. 
In Section \ref{G-s:categorical}, we discuss (A) above.
In Section \ref{G-s:homspaces}, we give a systematic method to calculate the length of hom-sets of the stable category $\uCM\Lambda$ from the Auslander-Reiten quiver $\underline{\mathfrak{A}}(\Lambda)$. 
In Section \ref{G-s:combinatorial}, we discuss (B) above. In this section, we also give a simple proof of Theorem \ref{G-th:mainthintro} by using the conditions (C1), (C2) of combinatorial configurations.
Later in the paper, we prove our main Theorems \ref{G-th:mainAintro} and \ref{G-th:mainBCDintro} case by case: Section \ref{G-s:typeA} for type $A$, Section \ref{G-s:typeBC} for types $B$ and $C$, Section \ref{G-s:typeD} for type $D$.
In the last section, we deal with configurations of the exceptional cases: types $E$, $F$ and $G$.

\medskip\noindent{\bf Acknowledgments }
This paper is a part of the author's PhD thesis.
The author would like to show her deep gratitude to her advisor Osamu Iyama for his guidance and valuable discussions. 

\section{Preliminaries}
\label{G-s:pre}
In this section, we recall some definitions and facts that we need later.
We always denote by $f g$ the composition of $X \xrightarrow{f} Y \xrightarrow{g} Z$ when they are maps or functors in this paper.

\subsection{Valued translation quivers}
\label{G-def:translation}
Recall that we do not assume that the residue field of $R$ is algebraically closed. This is quite natural to cover important classes of orders appearing in number theory (e.g. \cite{HiNi92}). Then the Auslander-Reiten quivers of $R$-orders have the structure of \emph{valued} translation quivers.
First, let us recall the definition of valued translation quivers.

  A \emph{valued translation quiver} is a quadruple $Q=(Q_0, d, d', \tau)$, where $Q_0$ is a set, $d$ and $d'$ are maps $d,d': Q_0 \times Q_0 \to \NN_{\ge0}$ and $\tau$ is a bijection from $ Q_0 \smallsetminus Q^p $ to $ Q_0 \smallsetminus Q^i$ for subsets $Q^p$ and $Q^i$ of $Q_0$ such that for any vertices $y \in Q_0$ and $z \in Q_0 \smallsetminus Q^p$, we have $d'_{(\tau z) y}=d_{yz}$.
  Moreover, we assume that there exists a map $a: Q_0 \to \NN_{> 0}$ such that $a_x d_{xy} = d'_{xy} a_y$ for any $x,y \in Q_0$.
We draw $Q$ as a quiver with valued arrows $$x \xrightarrow{(d_{xy},d'_{xy})}  y$$
if $d_{xy} \neq 0$.
We denote a valued arrow $x\xrightarrow{(1,1)}  y$ by $x\xrightarrow{} y$.
  If $Q^p=Q^i = \varnothing$, then $Q$ is called \emph{stable}.
  
  An \emph{automorphism} $g:Q \to Q$ of a valued translation quiver $Q$ is a bijection $g: Q_0 \to Q_0$ such that $g(Q^p)= Q^p$, $g(Q^i)=Q^i$, $g \tau = \tau g$, $d_{gxgy}=d_{xy}$ and $d'_{gxgy}=d'_{xy}$ hold for any $x,y \in Q_0$.
We denote by $\Aut(Q)$ the automorphism group of $Q$. 
Following \cite{DE87}, we call a subgroup $G \subset \Aut(Q)$ \emph{weakly admissible}  if for each $g \in G \smallsetminus \{1\}$ and $x \in Q_0$, 
we have $x \neq gx$ and $x^+ \cap (gx)^+=\varnothing$, where $x^+$ is the set of direct successors of $x$. 
Although the condition $x \neq gx$ is not assumed in \cite{DE87}, it is more reasonable to assume it. Note that if $x^+ \neq \varnothing$ and $x^+ \cap (gx)^+=\varnothing$, then $x \neq gx$ holds automatically.

For a weakly admissible subgroup $G \subset \Aut(Q)$, the quotient valued translation quiver $Q/G$ is defined as follows.
The set $(Q/G)_0$ of vertices is $Q_0/G=\{Gx \mid x \in Q_0 \}$. 
Moreover, $(Q/G)^p = Q^p/G$, $(Q/G)^i = Q^i/G$ and $\tau$, $d$, $d'$ are defined by the following commutative diagrams:
\begin{center}
\begin{tikzcd}
Q_0 \smallsetminus Q^p \arrow{r}[sloped, above]{\tau} \arrow{d} & Q_0 \smallsetminus Q^i \arrow{d}\\
(Q/G)_0 \smallsetminus (Q/G)^p \arrow{r}[sloped, above]{\tau} & (Q/G)_0 \smallsetminus (Q/G)^i,
\end{tikzcd}
\begin{tikzcd}
Q_0 \times Q_0 \arrow{rr}[sloped, above]{d \ (\resp d')} \arrow{d} && \NN_{\ge 0}\arrow[equal]{d}\\
(Q/G)_0 \times (Q/G)_0 \arrow{rr}[sloped, above]{d \ (\resp d')} && \NN_{\ge 0},
\end{tikzcd}
\end{center}
where the vertical maps are natural surjections.

\vspace{0.5cm}
Assume that $\Delta$ is a Dynkin diagram (see Figure \ref{G-fig:dynkin}).
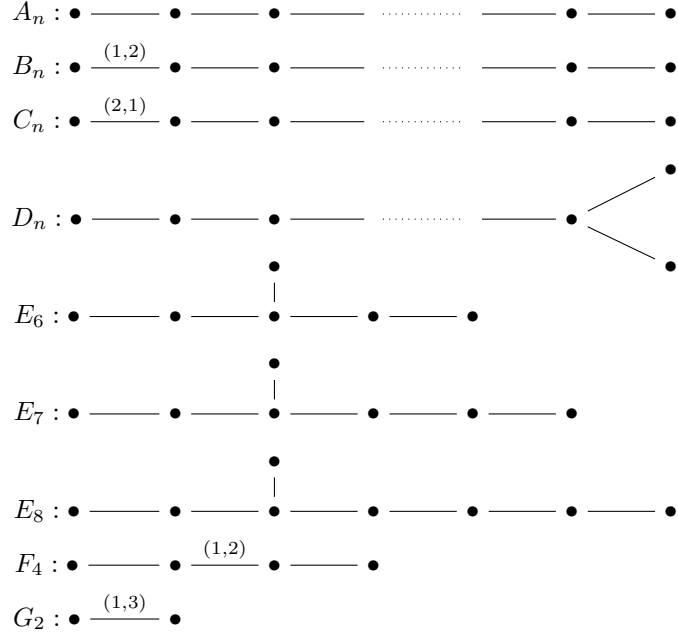
\begin{figure}
\begin{center}
\begin{tikzcd}[row sep=tiny]
A_n: \bullet \arrow[dash]{r} & \bullet \arrow[dash]{r} & \bullet \arrow[dash]{r}&\arrow[dotted,-]{r}&\arrow[dash]{r}&\bullet \arrow[dash]{r} &\bullet \\
B_n:  \bullet \arrow[dash]{r}[sloped, above]{(1,2)} & \bullet \arrow[dash]{r} & \bullet \arrow[dash]{r}&\arrow[dotted,-]{r}&\arrow[dash]{r}&\bullet \arrow[dash]{r} &\bullet \\
C_n:  \bullet \arrow[dash]{r}[sloped, above]{(2,1)} & \bullet \arrow[dash]{r} & \bullet \arrow[dash]{r}&\arrow[dotted,-]{r}&\arrow[dash]{r}&\bullet \arrow[dash]{r} &\bullet \\
&&&&&&\bullet \\
D_n:  \bullet \arrow[dash]{r} & \bullet \arrow[dash]{r} & \bullet \arrow[dash]{r}&\arrow[dotted,-]{r}&\arrow[dash]{r}&\bullet \arrow[dash]{ur} \arrow[dash]{dr}&\\
&&\bullet &&&&\bullet \\
E_6:   \bullet \arrow[dash]{r}&\bullet \arrow[dash]{r}&\bullet \arrow[dash]{r}\arrow[dash]{u}&\bullet \arrow[dash]{r}&\bullet &&\\
&&\bullet &&&&\\
E_7:   \bullet \arrow[dash]{r}&\bullet \arrow[dash]{r}&\bullet \arrow[dash]{r}\arrow[dash]{u}&\bullet \arrow[dash]{r}&\bullet \arrow[dash]{r}&\bullet &\\
&&\bullet &&&&\\
E_8:   \bullet \arrow[dash]{r}&\bullet \arrow[dash]{r}&\bullet \arrow[dash]{r}\arrow[dash]{u}&\bullet \arrow[dash]{r}&\bullet \arrow[dash]{r}&\bullet \arrow[dash]{r}&\bullet \\
F_4:  \bullet \arrow[dash]{r}&\bullet \arrow[dash]{r}[sloped, above]{(1,2)}&\bullet \arrow[dash]{r}&\bullet & &&\\
G_2: \bullet \arrow[dash]{r}[sloped, above]{(1,3)}&\bullet &&&&&
\end{tikzcd}
\end{center}
\caption{Dynkin diagrams with $n$ vertices}
\label{G-fig:dynkin}
\end{figure}
Let $Q$ be a valued quiver with underlying Dynkin diagram $\Delta$. 
We define the \emph{valued translation quiver} $\ZZ Q$ associated to $Q$ as follows. The vertices of $\ZZ Q$ are indexed by $(n,i)$ with $n \in \ZZ$ and $i \in Q_0$. There are arrows $$(n, \alpha): (n,i) \xrightarrow{(d_{ij}, d'_{ij})} (n,j) \and (n, \alpha^*): (n,j) \xrightarrow{(d'_{ij}, d_{ij})} (n+1, i)$$ in $\ZZ Q$ for any arrow $\alpha: i \xrightarrow{(d_{ij}, d'_{ij})} j$ in $Q$, and these are all the arrows in $\ZZ Q$. The translation $\tau$ is defined by $\tau((n,i))= (n-1,i)$. Then $\ZZ Q$ is a stable valued translation quiver.
According to \cite[Chapter I, 5.6]{Happel}, $\ZZ Q$ only depends on the underlying Dynkin diagram $\Delta$. Hence, we denote $\ZZ Q$ by $\ZZ\Delta$.

\vspace{0.5cm}
From a stable translation quiver and a subset of vertices, we can define a new translation quiver as follows.
 \begin{definition}
 \label{G-def:recover}
Let $Q$ be a stable valued translation quiver and $ C$ be a subset of $Q_0$. We define a valued translation quiver $Q_ C$ by adding to $Q$ a vertex $p_c$ and two arrows $c \to p_c \to \tau^{-1}(c)$ for each $c\in  C$. Moreover, the translation of $Q_ C$ coincides with the translation of $Q$ on $Q_0$ and is not defined on $Q_C^p=Q_C^i= \{p_c \mid c \in C\}$.
 \end{definition}
Let us see an example.
\begin{example}
The translation quiver $Q= \ZZ A_5$ is given by the left diagram below.
Moreover, let $C$ be a set of all the marked vertices. By adding new vertices and new arrows requested in Definition \ref{G-def:recover}, we get the new translation quiver $Q_C$.
\begin{center}
$Q=\ZZ A_5$: \begin{tikzcd}[column sep=tiny, row sep=tiny]
&&\bullet \drar & &\node[draw,circle,inner sep=1pt]{\bullet};\drar&&\node[draw,circle,inner sep=1pt]{\bullet}; \drar & &\\
\arrow[dotted,-]{r}&\node[draw,circle,inner sep=1pt]{\bullet};\urar\drar & &\bullet\urar\drar&&\bullet\urar\drar& &\bullet\arrow[dotted,-]{rr}&&\phantom{X}\\
&&\bullet\urar\drar&&\bullet\urar\drar&&\bullet\urar\drar&&\\ 
\arrow[dotted,-]{r}&\bullet\urar\drar & &\bullet\urar\drar &&\node[draw,circle,inner sep=1pt]{\bullet};\urar\drar& &\bullet\arrow[dotted,-]{rr}&& \phantom{X}\\
&&\node[draw,circle,inner sep=1pt]{\bullet}; \urar&&\bullet\urar &&\bullet \urar&&\\
\end{tikzcd}
$Q_C$: \begin{tikzcd}[column sep=tiny, row sep=tiny]
&&\bullet \drar & &\node[draw,circle,inner sep=1pt]{\bullet};\drar\arrow{r}&\bullet\arrow{r}&\node[draw,circle,inner sep=1pt]{\bullet}; \drar\arrow{r} & \bullet\arrow[dotted,-]{rr}&&\phantom{X}\\
\arrow[dotted,-]{r}&\node[draw,circle,inner sep=1pt]{\bullet};\urar\drar\arrow{r} &\bullet\arrow{r} &\bullet\urar\drar&&\bullet\urar\drar& &\bullet\arrow[dotted,-]{rr}&&\phantom{X}\\
&&\bullet\urar\drar&&\bullet\urar\drar&&\bullet\urar\drar&&\\ 
\arrow[dotted,-]{r}&\bullet\urar\drar & &\bullet\urar\drar &&\node[draw,circle,inner sep=1pt]{\bullet};\urar\drar\arrow{r}&\bullet\arrow{r} &\bullet\arrow[dotted,-]{rr}&& \phantom{X}\\
&&\node[draw,circle,inner sep=1pt]{\bullet}; \urar\arrow{r}&\bullet\arrow{r}&\bullet\urar &&\bullet \urar&&\\
\end{tikzcd}
\end{center}

\end{example}

\subsection{Cohen-Macaulay modules}
In this subsection, we introduce basic facts of orders and their Cohen-Macaulay modules.
Throughout this paper, we denote by $R$ a complete discrete valuation ring, for example, $R$ is a formal power series ring in one variable over a field.
 We denote the field of fractions of $R$ by $K$. We recall the definition of orders and Cohen-Macaulay modules first, see \cite{RT,CR2,roggenkamp882, RoHulattices, POS}.

\begin{definition} 
\label{G-def:order}
Let $\Lambda$ be an $R$-algebra and $M$ be a left $\Lambda$-module. We denote by $\mod \Lambda$ the category of finitely generated $\Lambda$-modules.
\begin{enumerate}
 \item $\Lambda$ is called an {\emph{$R$-order}} if it is finitely generated and free as an $R$-module. 

 \item If $\Lambda$ is an $R$-order, then a $\Lambda$-module $M$ is called a {\emph{Cohen-Macaulay $\Lambda$-module}} or \emph{$\Lambda$-lattice} if it is finitely generated and free as an $R$-module. 
      We denote by $\CM\Lambda$ the category of Cohen-Macaulay $\Lambda$-modules. It is a full subcategory of $\mod \Lambda$, which is closed under submodules and extensions. 
    
\item  We define the stable category $\uCM\Lambda$ of $\CM\Lambda$ as follows.
     \begin{itemize}
         \item The objects of $\uCM\Lambda$ are the same as the objects of $\CM\Lambda$. 
          \item For any $X,Y \in \uCM\Lambda$, the hom-set is defined by $$\uHom_\Lambda(X,Y):=\Hom_\Lambda(X,Y)/[\Lambda](X,Y),$$ where $[\Lambda](X,Y)$ is the set of morphisms factoring through a projective $\Lambda$-module.
          \end{itemize}
\end{enumerate}
\end{definition}  

We denote by $\fl\Lambda$ the category of finite length $\Lambda$-modules.
Then $(\fl\Lambda, \CM\Lambda)$ is a torsion pair in $\mod\Lambda$ in the sense that $\Hom_\Lambda(X,Y)=0$ holds for any $X \in \fl\Lambda$ and $Y \in \CM\Lambda$, and moreover, for any $X \in \mod \Lambda$, there exists an exact sequence $0 \to TX \to X \to FX \to 0$ with $TX \in \fl\Lambda$ and $FX \in \CM\Lambda$.
For any $R$-order $\Lambda$ (e.g. $\Lambda=R$), we have the following exact dualities 
\begin{align*}
\Dr & :=\Hom_R(-, R) : \CM\Lambda \leftrightarrow \CM\Lambda^{\operatorname{op}},\\
\Dk & :=\Ext^1_R(-,R): \fl\Lambda \leftrightarrow \fl\Lambda^{\operatorname{op}}.
\end{align*}
We denote by $\proj\Lambda$ the category of projective $\Lambda$-modules.
They are all projective objects in $\CM\Lambda$.
Now let us define the injective objects in $\CM\Lambda$.
\begin{definition}
\label{G-def:injective}
Let $\Lambda$ be an $R$-order. 
We call $I \in \CM\Lambda$ an \emph{injective Cohen-Macaulay} $\Lambda$-module if $I \in \add(\Dr(\Lambda^{\operatorname{op}}))$.
 We denote by $\inj \Lambda$ the category of injective Cohen-Macaulay $\Lambda$-modules. 
 \end{definition}

We have the following easy properties.
\begin{proposition}
\label{G-pro:injectivehull}
Let $\Lambda$ be an $R$-order.
     \begin{itemize}
          \item[(1)] For any $X \in \CM\Lambda$, there exists a short exact sequence 
               \begin{equation}
               \label{G-eq:injectivehull}
               0\to X \xrightarrow{f} Q \to Y \to 0
               \end{equation}
                with $Q\in \inj\Lambda$ and $Y \in \CM\Lambda$.
         \item[(2)] $\inj \Lambda =\{ I \in \CM\Lambda \mid \Ext^1_\Lambda(-, I)|_{\CM\Lambda} =0\}.$
     \end{itemize}
\end{proposition}
\begin{proof}
(1) We have a short exact sequence $0 \to Z \to P \to \Dr X \to 0$ over $\Lambda^{\operatorname{op}}$-modules with $P \in \proj \Lambda^{\operatorname{op}}$ and $Z \in \CM \Lambda^{\operatorname{op}}$. 
Applying the exact duality $\Dr: \CM \Lambda \to \CM\Lambda^{\operatorname{op}}$, we have $ 0 \to X \to \Dr P \to \Dr Z \to 0$ with $\Dr P \in \inj \Lambda$ and $\Dr Z \in \CM\Lambda$.

(2) We get the inclusion \textquotedblleft$\subset$\textquotedblright from the isomorphism $ \Ext^1_\Lambda(X,Y) \cong \Ext^1_{\Lambda^{\operatorname{op}}}(\Dr Y, \Dr X)$ for any $X,Y \in \CM\Lambda$.
The other inclusion \textquotedblleft$\supset$\textquotedblright follows from the exact sequence \eqref{G-eq:injectivehull}.
  \end{proof}

 From now on, we assume that $\Lambda$ is a ring-indecomposable $R$-order. 
 We denote by $\simp \Lambda$ the set of the isomorphism classes of simple $\Lambda$-modules.
For an additive category $\ca$, we denote by $\ind(\ca)$ the set of the isomorphism classes of indecomposable objects in $\ca$. 
 
 Consider the \emph{Nakayama functor} $$\nu:=(\Dr\Lambda)\otimes_\Lambda - \cong \Dr\Hom_\Lambda(-,\Lambda): \CM\Lambda \to \CM\Lambda.$$
 It induces an equivalence $$\nu: \proj\Lambda \xrightarrow{\sim} \inj\Lambda.$$
For any $X\in \CM\Lambda$, let $$ \corad X := \Dr (\rad (\Dr X)) \and  \cotop X := \Dk (\tp (\Dr X)).$$
We have the following proposition to give other bijections. 
\begin{proposition}
\label{G-lem:onetoone}
We have one-to-one correspondences between
\begin{itemize}   
\item[(a)] $\ind(\proj\Lambda)$ 
\item[(b)] $\ind(\inj \Lambda)$ 
\item[(c)] $\simp \Lambda$ 
\item[(a$'$)] $\ind(\proj\Lambda^{\operatorname{op}})$
\item[(b$'$)] $\ind(\inj \Lambda^{\operatorname{op}})$
\item[(c$'$)] $\simp \Lambda^{\operatorname{op}}$.
\end{itemize}
     The bijection $\ind (\proj \Lambda) \to \simp \Lambda$ is given by $P \mapsto \tp P$, the bijection $\ind (\proj \Lambda) \to \ind (\inj \Lambda)$ is given by $P \mapsto \nu P$ and the bijection $\ind (\inj \Lambda) \to \simp \Lambda$ is given by $I \mapsto \cotop I$. Moreover, these bijections commute.
\end{proposition}
\begin{proof}
Since $\Lambda$ is semiperfect (\emph{i.e.} every finitely generated $\Lambda$-module has a projective cover), the correspondence between (a) (\resp (a$'$)) and (c) (\resp (c$'$)) is classical. Since $\Dr$ is the duality, (a) is in bijection with (b$'$) (\resp (b) is in bijection with (a$'$)).
\end{proof}
 
The following sequence is basic.
\begin{lemma}
For any $X\in \CM\Lambda$, we have a short exact sequence $$ 0 \to X \to \corad X \to \cotop X \to 0.$$
\end{lemma}
\begin{proof}
For $\Dr X \in \CM\Lambda^{\operatorname{op}}$, we consider the exact sequence $0 \to \rad (\Dr X) \to \Dr X \to \tp (\Dr X) \to 0$.
By applying $\Hom_R (-, R)$, we have an exact sequence 
$$0 \to  \Dr(\tp (\Dr X)) \to \Dr(\Dr X) \to \Dr(\rad (\Dr X)) \to \Dk (\tp(\Dr X)) \to \Dk(\Dr X).$$
Since $\Dr(\Dr X)=X$, $\Dr(\tp(\Dr X))=0$ and $\Dk(\Dr X)=0$, we have the desired sequence.
\end{proof}

The following lemma is about the extension group from a simple module to an injective Cohen-Macaulay module.
\begin{lemma} 
\label{G-lem:onedim}
    \begin{itemize}
        \item[(1)] For any $P \in \ind(\proj \Lambda)$ and $S \in \simp\Lambda$, $\Hom_\Lambda(P,S) \neq 0$ if and only if $S=\tp P$. In this case, $\Hom_\Lambda(P,S)$ is a simple $\End_\Lambda(P)$-module and a simple $\End_\Lambda(S)^{\operatorname{op}}$-module.
         \item[(2)] For any $I\in \ind(\inj \Lambda)$ and $S\in \simp\Lambda$, $\Ext^1_\Lambda(S, I)\neq0$ if and only if $S= \cotop I$. In this case, $\Ext^1_\Lambda(S, I)$ is a simple $\End_\Lambda(I)^{\operatorname{op}}$-module and a simple $\End_\Lambda(S)$-module.     
          \item[(3)]  For any $P\in\ind(\proj\Lambda)$ and $I\in \ind(\inj \Lambda)$, $\Ext^1_\Lambda(\tp P,I)\neq 0$ if and only if $\cotop I\cong \tp P$ if and only if $I \cong \nu P$.
     \end{itemize}
\end{lemma}
\begin{proof}
(1) is immediate.

(2) The first statement is well-known, e.g. see \cite[1.2 (3)]{tauthree}.
We will prove that $\Ext^1_\Lambda(S, I)$ is a simple $\End_\Lambda(I)$-module.
Applying $\Hom_\Lambda(-, I)$ to $0 \to I \to \corad I \to \cotop I \to 0,$
we have an exact sequence 
\begin{equation}
\label{G-eq:injone}
0=\Hom_\Lambda(\cotop I, I) \to \Hom_\Lambda(\corad I, I) \to \Hom_\Lambda(I,I)\to \Ext^1_\Lambda(\cotop I, I) \to \Ext^1_\Lambda(\corad I, I)=0,
 \end{equation}
where the last term is $0$ since $\corad I \in \CM\Lambda$ and $I \in \inj \Lambda$.
Since $\Dr I \in \ind(\proj \Lambda^{\operatorname{op}}) $, there is an exact sequence $$0 \to \rad (\Dr I) \to \Dr I \to \tp(\Dr I) \to 0.$$
By applying $\Hom_{\Lambda^{\operatorname{op}}}(\Dr I, -)$, we get an exact sequence 
\begin{equation}
\label{G-eq:injtwo}
0\to \Hom_{\Lambda^{\operatorname{op}}}(\Dr I, \rad(\Dr I)) \to \Hom_{\Lambda^{\operatorname{op}}}(\Dr I, \Dr I) \to \Hom_{\Lambda^{\operatorname{op}}}(\Dr I, \tp(\Dr I)) \to 0.
 \end{equation}
 Since the left and the middle terms of \eqref{G-eq:injone} and \eqref{G-eq:injtwo} are isomorphic, we have $$\Ext^1_\Lambda(\cotop I, I)\cong \Hom_{\Lambda^{\operatorname{op}}}(\Dr I, \tp(\Dr I)).$$ This is a simple $\End_\Lambda(I)^{\operatorname{op}}$-module and a simple $\End_\Lambda(S)$-module by (1).
 
(3) According to (2) above, the statement $\Ext^1_\Lambda(\tp P,I)\neq 0$ is equivalent to $\cotop I\cong \tp P$. Then by Proposition \ref{G-lem:onetoone}, this is equivalent to $I \cong \nu P$.
\end{proof}

The class of Gorenstein orders is very important and the definition of Gorenstein orders is the following. 
 \begin{definition}
 \label{G-def:gorenstein}
An $R$-order $\Lambda$ is \emph{Gorenstein} if $\Hom_{{R}}(\Lambda_\Lambda, R)$ is projective as a left $\Lambda$-module, or equivalently, if $\Hom_R({}_\Lambda \Lambda, R)$ is projective as a right $\Lambda$-module. 
 \end{definition}
 If $\Lambda$ is Gorenstein, then $\CM\Lambda$ is a Frobenius category and therefore $\uCM\Lambda$ has a natural structure of a triangulated category \cite{Happel} such that the suspension functor $[1]$ is given by $X[1]=Y$, where $Y$ is given by the exact sequence \eqref{G-eq:injectivehull}.
The following lemma is about some basic property of Cohen-Macaulay modules over Gorenstein orders that we need to use later.
 \begin{lemma}
\label{G-lem:iso}
Let $\Lambda$ be a Gorenstein $R$-order.
For any $X, Z\in \CM\Lambda$, we have the following isomorphism $$\uHom_\Lambda(X, Z) \cong \Ext^1_\Lambda (Y, Z),$$
where $Y$ is given by \eqref{G-eq:injectivehull}.
 \end{lemma}
 \begin{proof}
  Applying $\Hom_\Lambda(-, Z)$ with $Z\in \CM\Lambda$ to \eqref{G-eq:injectivehull}, we get the following long exact sequence
\begin{equation*}
\label{G-long}
\begin{tikzcd}
0 \arrow{r} & \Hom_\Lambda(Y, Z) \arrow{r} & \Hom_\Lambda(Q, Z) \arrow{r}{f^*} & \Hom_\Lambda(X, Z) \arrow{r} &\Ext^1_\Lambda(Y, Z)\arrow{r} &0.
\end{tikzcd}
\end{equation*}
Since $\Lambda$ is Gorenstein, we have $\Ext^1_\Lambda(Y,\Lambda)=0.$
Therefore, any morphism from $X$ to a projective $\Lambda$-module factors through $f$.
In particular, we have $\im f^* \cong {[\Lambda]}(X, Z).$ 
By the above exact sequence, we have $\uHom_\Lambda(X, Z) \cong \Ext^1_\Lambda (Y, Z)$.
 \end{proof}

 Notice that when $\Lambda$ is Gorenstein, the Nakayama functor induces an autoequivalence $$\nu: \CM\Lambda \xrightarrow{\sim} \CM\Lambda,$$ which induces a triangle-equivalence $\nu: \uCM\Lambda \xrightarrow{\sim} \uCM\Lambda$.
The composition $\tau:=\nu [-1] : \uCM\Lambda \xrightarrow{\nu} \uCM\Lambda \xrightarrow{[-1]} \uCM\Lambda$ is called the \emph{Auslander-Reiten translation}.
 If $\Lambda \otimes_R K$ is a semisimple $K$-algebra, then Auslander-Reiten theory works in $\CM\Lambda$ as the following classical results show.

\begin{proposition}
[{\cite{FMO, almostsplitseqorder}}]
\label{G-pro:nakayama}
If $\Lambda$ is Gorenstein and $\Lambda \otimes_R K$ is a semisimple $K$-algebra, then 
\begin{itemize}
   \item[(1)] $\uCM\Lambda$ is a Hom-finite triangulated category;
   \item[(2)] we have a functorial isomorphism $\uHom_\Lambda(X,Y)\cong \Dk\uHom_\Lambda(Y,\nu X)$ for any $X,Y \in \CM\Lambda$.
       \end{itemize}
\end{proposition}
\begin{proof}
(1) Let us prove the Hom-finiteness. For any $X,Y \in \CM\Lambda$, since $\Lambda \otimes_R K$ is semisimple by our assumption, we have $$\Ext^1_\Lambda (X,Y) \otimes_R K = \Ext^1_{\Lambda \otimes_R K}(X\otimes_R K, Y\otimes_R K)=0.$$
Therefore, $\Ext^1_\Lambda(X,Y) $ is an $R$-module of finite length. We conclude by using Lemma \ref{G-lem:iso}.

 (2) Note that we have $\tau=\Tr  [-1]  \Dr: \uCM\Lambda \to \uCM\Lambda$, where $\Tr$ is the transpose: $\umod\Lambda \to \umod\Lambda^{\operatorname{op}}$ \cite{bluebook, AB69}. 
As $R$ has Krull dimension $1$, according to \cite[Chapter \uppercase\expandafter{\romannumeral1}, Proposition 8.7]{FMO}, we have a functorial isomorphism $\uHom_\Lambda(X,Y) \cong \Dk \Ext^1_\Lambda(Y, \tau X)$.
Therefore, $\uHom_\Lambda(X,Y)\cong \Dk\uHom_\Lambda(Y,\nu X)$ holds for any $X,Y \in \uCM\Lambda$.
\end{proof}

Given an $R$-order $\Lambda$, we denote the radical of the category $\CM\Lambda$ by $\rad_{\CM\Lambda} (-,-)$.
We can define a valued Auslander-Reiten quiver as follows.
\begin{definition}[e.g. {\cite[3.8]{roggenkamp882}}]
Let $X,Y \in \ind(\CM\Lambda)$.
Denote $$\Irr(X,Y):=\frac{\rad_{\CM\Lambda}(X,Y)}{\rad_{\CM\Lambda}^2(X,Y)} \and k_X:=\frac{\End_\Lambda(X)}{\rad_{\CM\Lambda}(X,X)}.$$
We define $$d_{XY}:= \dim_{k_X} (\Irr(X,Y)) \and d'_{XY}:= \dim (\Irr(X,Y))_{k_Y}.$$
The \emph{Auslander-Reiten quiver} $\mathfrak{A}(\Lambda)$ of $\Lambda$ is a valued translation quiver $(\ind(\CM\Lambda), d, d', \tau)$, where the translation $\tau$ is the Auslander-Reiten translation, $\mathfrak{A}(\Lambda)^p= \ind(\proj \Lambda)$ and $\mathfrak{A}(\Lambda)^i = \ind(\inj \Lambda)$.
 The \emph{stable Auslander-Reiten quiver} $\underline{\mathfrak{A}}(\Lambda)$ is the full subquiver of $\mathfrak{A}(\Lambda)$ whose vertices are the objects in $\ind(\CM\Lambda)\smallsetminus \ind(\proj\Lambda)$.
\end{definition}

\begin{remark}
\label{G-rem:almostsplit}
 (1) (e.g. {\cite[3.8]{roggenkamp882}}) For any $X\in \ind(\CM\Lambda)$, the right minimal almost split map is of the form $$\bigoplus_{Y \in \ind(\CM\Lambda)} Y^{d_{YX}} \to X,$$ and the left minimal almost split map is of the form $$X \to \bigoplus_{Y \in \ind(\CM\Lambda)} Y^{d'_{XY}} .$$
 
 (2) For any $P\in \ind(\proj\Lambda)$, $\rad P \to P$ is a minimal right almost split map.
 For any $I \in \ind(\inj\Lambda)$, $I \to \corad I$ is a minimal left almost split map.
\end{remark}

We have the following analogue of Riedtmann's structure theorem of stable Auslander-Reiten quivers \cite{Riedtmannstructure} for orders.
 \begin{theorem} 
 \label{G-th:structure}
 Let $R$ be a complete discrete valuation ring and $\Lambda$ be a representation-finite Gorenstein $R$-order.
 The stable Auslander-Reiten quiver $\underline{\mathfrak{A}}(\Lambda)$ of $\Lambda$ is isomorphic to a disjoint union of $\ZZ\Delta/G,$ where $\Delta$ is a Dynkin diagram and $G \subset \Aut(\ZZ\Delta)$ is a weakly admissible group.
\end{theorem}
\begin{proof}
By Proposition \ref{G-pro:nakayama}, $\uCM\Lambda$ is a Hom-finite triangulated category. Thus, the statement follows from a result by Xiao and Zhu \cite[Theorem 2.3.5]{Xiaozhutriangulated}, see also \cite[Theorem 4.1]{Amiottri}.
Note that in \cite{Xiaozhutriangulated}, triangulated categories over a field are discussed, but all arguments work in our settings.
\end{proof}

\section{Categorical configuration}
\label{G-s:categorical}
\begin{assumption}
\label{G-assumption}
In this section, we assume that the following constraints are satisfied.
\begin{itemize}
     \item $\Lambda$ is a ring-indecomposable representation-finite Gorenstein $R$-order; 
     \item $\rad P$ is indecomposable and non-projective for any $P\in \ind(\proj\Lambda)$;
     \item  $\rad P \ncong \rad Q$ holds for any $P, Q \in \ind(\proj\Lambda)$ satisfying $P \ncong Q$.
 \end{itemize}    
 \end{assumption}
Under Assumption \ref{G-assumption}, we will show in two main theorems that the set of isomorphism classes of the radicals of indecomposable projective $\Lambda$-modules satisfy two categorical conditions. The first one is the following.
\begin{theorem}
\label{G-pro:c1}
For any non-projective $X\in\ind(\CM\Lambda)$, $X$ admits an injective hull $P\in\ind(\proj\Lambda)$ which satisfies $$\uHom_\Lambda(X, \rad P) \neq 0.$$
\end{theorem}
\begin{proof}
For any non-projective $X\in\ind(\CM\Lambda)$, thanks to Proposition \ref{G-pro:injectivehull} (1), there exists a non-split exact sequence 
\begin{equation*}
0\to X \xrightarrow{h} P \to Y \to 0
\end{equation*}
 with $P \in \inj\Lambda$ and $Y \in \CM\Lambda$.
 Moreover, since $\Lambda$ is Gorenstein, we can choose the above exact sequence such that $P$ is the injective hull of $X$ in $\CM\Lambda$ which is indecomposable.
 
Since $h$ is a monomorphism and $X$ is non-projective, $h$ factors through $\rad P$. In fact, we have $$h= h' i_P :X \xrightarrow{h'} \rad P \xrightarrow{i_P} P.$$ 
We show that $h'$ does not factor through any projective modules. 
Assume $h'=st: X \xrightarrow{s} Q \xrightarrow{t} \rad P$ with $Q \in \ind(\proj\Lambda)$.
Since $Q$ is injective Cohen-Macaulay, we have the following commutative diagram
\begin{center}
\begin{tikzcd}
0 \arrow{r} &X \arrow{r}{h} \arrow{d}{s} &P \arrow{r}\arrow{dl}{g} &Y \arrow{r} & 0.\\
& Q &&&
\end{tikzcd}
\end{center}
Then $hgti_P=h$ holds and hence, $gt i_P \cong \Id_P$ and $i_P$ is a split epimorphism. This is a contradiction.
Thus, $ \uHom_\Lambda(X, \rad P)\neq 0$.
\end{proof}

The second categorical condition of configurations is described as follows.
\begin{theorem}
\label{G-pro:c2}
For any $P,Q\in\ind(\proj\Lambda)$, we have 
\begin{align*}
\len_{\End_\Lambda(\rad P)} \uHom_\Lambda (\rad P, \rad Q)
=\begin{cases}
2, \text{ if }P \cong Q \cong \nu^{-1} Q ;\\
1, \text{ if }P \cong  Q \ncong \nu^{-1} Q ;\\
1, \text{ if }P \cong \nu^{-1} Q \ncong Q ;\\
0,\text{ if } P \ncong Q, \nu^{-1} Q.
\end{cases}
\end{align*}
 \end{theorem}
 
In the rest, we consider the following diagram with $P,Q\in\ind(\proj\Lambda)$
 \begin{equation}
\label{G-basic}
\begin{tikzcd}
0 \arrow{r} &\rad P \arrow{r}{i_P}\arrow{d}{ f} & P \arrow{r}{\pi_P}\arrow{d}{ g} & \tp P \arrow{r}\arrow{d}{ h} &0\\
0 \arrow{r} &\rad Q \arrow{r}[swap]{i_Q} & Q \arrow{r}[swap]{\pi_Q} & \tp Q \arrow{r} &0.
\end{tikzcd}
\end{equation}

\begin{lemma}
\label{G-lem:commdiag}
Let $P,Q \in \ind(\proj\Lambda)$.
\begin{itemize}
\item[(1)] For any $g \in \Hom_\Lambda ( P, Q)$, there exist $f \in \Hom_\Lambda (\rad P, \rad Q)$ and $h\in \Hom_\Lambda (\tp P, \tp Q)$ such that \eqref{G-basic} is commutative.
\item[(2)] Assume $Q \ncong \nu P$.
For any $f \in \Hom_\Lambda (\rad P, \rad Q)$, there exist $g\in \Hom_\Lambda ( P, Q)$ and $h\in \Hom_\Lambda (\tp P, \tp Q)$ such that \eqref{G-basic} is commutative.
\item[(3)] For any $f \in [\Lambda](\rad P, \rad Q)$, there exist $g \in \rad_{\CM\Lambda}(P,Q)$ such that \eqref{G-basic} is commutative with $h=0$.
\item[(4)] For the commutative diagram \eqref{G-basic}, the following conditions are equivalent:
     \begin{itemize}
           \item[(a)] $h=0$, 
           \item[(b)] $f \in [\Lambda](\rad P, \rad Q)$,
           \item[(c)]  $g \in \rad_{\CM\Lambda}(P,Q)$.
      \end{itemize}
\end{itemize}
\end{lemma}
\begin{proof}
(1) Since any morphism $g: P \to Q$ induces a morphism $f: \rad P \to \rad Q$, then there exists $h\in \Hom_\Lambda (\tp P, \tp Q)$ such that  \eqref{G-basic} is commutative.

(2) It is sufficient to find $g\in \Hom_\Lambda(P,Q)$ such that $i_P  g=f  i_Q$.
Since $Q\ncong \nu P$, $\Ext^1_\Lambda(\tp P, Q)=0$ by Lemma \ref{G-lem:onedim} (3).
By applying $\Hom_\Lambda(-, Q)$ to the short exact sequence 
\begin{align*}
0 \to \rad P \xrightarrow{i_{ P}}  P \xrightarrow{\pi_{P}}  \tp{P} \to 0,
\end{align*}
we get the following long exact sequence
\begin{align*}
0 \to \Hom_\Lambda(\tp P, Q) \to \Hom_\Lambda( P, Q) \xrightarrow{i^*_P} \Hom_\Lambda(\rad P, Q) \to \Ext^1_\Lambda(\tp P, Q)=0.
\end{align*}
Therefore, there exists $g\in \Hom_\Lambda(P,Q)$ such that $i_P g=f  i_Q$.

(3) First, let us prove the existence of $g$. 
Assume $Q \ncong \nu P$. There exists $g : P \to Q$ such that $ f  i_{Q}= i_P  g$ by (2). 
Assume $Q \cong \nu P$. 
We have the following long exact sequence $$0 \to\Hom_\Lambda(\tp P, Q) \to  \Hom_\Lambda(P,Q) \xrightarrow{i_P^*} \Hom_\Lambda(\rad P,  Q)\xrightarrow{\rho}\Ext^1_\Lambda (\tp P,  Q) \to 0.$$
As $f \in {[\Lambda]}(\rad P, \rad Q)$, there exists $I\in \proj\Lambda$ and morphisms $s,t$ such that the following diagram commutes:
\begin{center}
\begin{tikzcd}[row sep=small]
0 \arrow{r} &\rad P \arrow{rr}{i_P}\arrow{dd}{\forall f} \drar{\exists s} && P \arrow{r}{\pi_P} & \tp P \arrow{r} &0\\
&& I\dlar{\exists t}\drar{t  i_{Q}}&&&\\
0 \arrow{r} &\rad Q \arrow{rr}[swap]{i_{Q}} && Q \arrow{r}[swap]{\pi_{Q}} & \tp Q \arrow{r} &0.
\end{tikzcd}
\end{center}
Let us prove that the morphism from $\Ext_\Lambda^1(\tp P, I)$ to $\Ext_\Lambda^1(\tp P, Q)$ induced by $t  i_{Q}$ is $0$.
Indeed, thanks to Lemma \ref{G-lem:onedim} (3), either $\Ext_\Lambda^1(\tp P, I)=0$ holds or $I \cong Q \cong \nu P$ holds. 
In the first case, the conclusion is obvious.
In the second case, as the morphism $t  i_{Q}: I \to Q$ is not an isomorphism, it induces a morphism $\chi$ in $\rad_{\CM\Lambda}(Q,Q) \subset \End_\Lambda(Q)$.
Moreover, by Lemma \ref{G-lem:onedim} (2), $\Ext_\Lambda^1(\tp P, Q)$ is a simple $\End_\Lambda(Q)^{\operatorname{op}}$-module. Thus, the morphism from $\Ext_\Lambda^1(\tp P, I)$ to $\Ext_\Lambda^1(\tp P, Q)$ induced by $t  i_{Q}$ has its image included in $ \Ext_\Lambda^1(\tp P, Q) \cdot \chi =0$. 
Then $\rho(  f  i_{Q})=\rho(s  t  i_{Q})=0$, thus, there exists $g : P \to Q$ such that $ f  i_{Q}= i_P  g $. 

Now we show that $g$ is in the radical.
Since $\rad P$ and $\rad Q$ are non-projective, it follows that $f$ is not an isomorphism which means that $g\in \rad_{\CM\Lambda}(P,  Q)$. Since $P$ and $Q$ are indecomposable projective, it follows that $h=0$.

(4) First, we prove ($a \Rightarrow b$). Since $h=0$, $g  \pi_Q=0$ holds. 
Thus, there exists $t : P \to \rad Q$ such that $g=t   i_Q$. It follows $f  i_Q = i_P  t  i_Q$, therefore, $f=i_P   t \in {[\Lambda]}(\rad P, \rad Q)$. 
($b \Rightarrow c$) follows from the proof of (3). ($c \Rightarrow a$) is obvious since $P$ and $Q$ are indecomposable projective.
\end{proof}

By using the above lemma, we can get the following results.
\begin{lemma}
\label{G-lem:Hiso}
Let $P,Q \in \ind(\proj\Lambda)$. We have the following two isomorphisms:
\begin{itemize}
\item[(1)]  $i_{Q*}: \Hom_\Lambda(\rad P, \rad Q) \to \Hom_\Lambda(\rad P, Q);$
\item[(2)]  $F: \rad_{\CM\Lambda}(P,Q) \to [\Lambda](\rad P, \rad Q)$ which sends $g$ to $f$ in \eqref{G-basic}.
\end{itemize}
\end{lemma}
\begin{proof}
(1) Since $i_Q$ is a monomorphism, so is $i_{Q*}:=\Hom_\Lambda(\rad P, i_Q)$.
We show that $i_{Q*}$ is surjective. Indeed, for any $h \in \Hom_\Lambda(\rad P, Q)$, 
since $i_{Q}$ is minimal right almost split and $\rad P \ncong Q$ by Assumption \ref{G-assumption}, there exists $f \in \Hom_\Lambda(\rad P, \rad Q)$ satisfying $h= f  i_{Q}$.

(2) For any $g \in \rad_{\CM\Lambda}(P,Q)$, by Lemma \ref{G-lem:commdiag} (1), we have $f \in \Hom_\Lambda(\rad P, \rad Q)$ making the diagram \eqref{G-basic} commutative.
By Lemma \ref{G-lem:commdiag} (4), $f$ factors through $\proj\Lambda$.
The correspondence $F: g \mapsto f$ is well-defined since the commutative diagram \eqref{G-basic} shows that $g=0$ implies $f=0$.
Moreover, $F$ is surjective by Lemma \ref{G-lem:commdiag} (3) and injective since $f=0$ implies that $g$ factors through $\tp P$ and has to be zero.
\end{proof}

To prove Theorem \ref{G-pro:c2}, we need the following lemma which contributes to calculate the length of hom-sets of radicals of projective modules.
\begin{lemma}
\label{G-lem:seqforlength}
Let $P,Q \in \ind(\proj \lambda)$ satisfying $P \cong   \nu^{-1}Q$.
We have the following short exact sequence:
\begin{equation}
\label{G-eq:seqforlength}
0\to \frac{\Hom_\Lambda(P,Q)}{\rad_{\CM\Lambda}(P,Q)} \to \uHom_\Lambda(\rad P, \rad Q) \to \Ext^1_\Lambda (\tp P,  Q) \to 0.
\end{equation}
\end{lemma}
\begin{proof}
First, we prove that there is the following commutative diagram:
\begin{equation}
\label{G-eq:main}
\begin{tikzcd}[ampersand replacement=\&, column sep=small]
0 \arrow{r} \& \Hom_\Lambda(P,Q) \arrow{r}{i_P^*}  \&\Hom_\Lambda(\rad P,  Q)\arrow{r}{\rho}  \&\Ext^1_\Lambda (\tp P,  Q) \arrow{r} \& 0 \\
\&\rad_{\CM\Lambda}(P,Q)\arrow[hookrightarrow]{u}[swap]{\iota'} \arrow{d}[above,rotate=90]{\sim}{F} \& \& \& \\
0\arrow{r} \& {[\Lambda]}(\rad P , \rad Q) \arrow{r}{\iota} \& \Hom_\Lambda(\rad P, \rad Q) \arrow{r}\arrow{uu}[above,rotate=90]{\sim}[swap]{i_{Q*}} \&\uHom_\Lambda(\rad P, \rad Q) \arrow[two heads]{uu}[swap]{\phi} \arrow{r} \&0.
\end{tikzcd}
\end{equation}
We get the first row by applying $\Hom_\Lambda(-, Q)$ to $0 \to \rad P \to P \to \tp P \to 0$. Indeed, since $Q$ is a free $R$-module, it follows that $\soc_\Lambda Q=\soc_R Q=0$ and hence, $\Hom_\Lambda(\tp P, Q)=0$. And the second row comes from the definition of $\uHom_\Lambda(-,-)$. 
Moreover, $i_{Q*}$ and $F$ are isomorphisms by Lemma \ref{G-lem:Hiso}.
By Lemma \ref{G-lem:Hiso} and the commutative diagram \eqref{G-basic}, we have $( \iota'  i^*_P )(g)= i_P g =f i_Q = (F  \iota i_{Q*})(g)$. The existence of $\phi$ is easy to deduce and the diagram \eqref{G-eq:main} is commutative.

By an easy diagram chasing, we get the short exact sequence \eqref{G-eq:seqforlength}
\end{proof}

Considering the field $k_P:=\End_\Lambda(P)/\rad_{\CM\Lambda}(P,P)$ for each indecomposable projective module $P$, the following property holds.
\begin{lemma}
\label{G-lem:residuefield}
For any $P \in \ind(\proj\Lambda)$, we have $\End_\Lambda(\tp P) \cong k_P \subset k_{\rad P}.$
\end{lemma}
\begin{proof}
By Lemma \ref{G-lem:commdiag} (1), there exists a morphism $\End_\Lambda (P) \to \End_\Lambda (\tp P)$ which is surjective since $\tp P$ is simple. The kernel of this morphism is $\rad_{\CM\Lambda}(P,P)$ according to Lemma \ref{G-lem:commdiag} (4). Therefore, we have $\End_\Lambda(\tp P) \cong k_P$.

We have the following commutative diagram
\begin{center}
\begin{tikzcd}
\End_\Lambda (P) \arrow{r} & \End_\Lambda (\rad P) \\
\rad_{\CM\Lambda}(P,P) \arrow{r} \arrow[hook]{u} & \rad_{\CM\Lambda}(\rad P,\rad P)\arrow[hook]{u},
\end{tikzcd}
\end{center}
which induces a morphism $k_P \to k_{\rad P}$ taking cokernels of vertical inclusions.
As the morphism of these two fields sends $\Id_P$ to $\Id_{\rad P}$, it is injective and we have $k_P \subset k_{\rad P}$.
\end{proof}

Now we are ready to give the proof of Theorem \ref{G-pro:c2}.
\begin{proof}[Proof of Theorem \ref{G-pro:c2}]
(i) Assume $P \ncong \nu^{-1} Q$, or equivalently, $Q\ncong \nu P$. For any $f :\rad P \to \rad Q$, by Lemma \ref{G-lem:commdiag} (2), we have the commutative diagram \eqref{G-basic}.
\begin{enumerate}[({i}-i)]
\item Assume $P \ncong Q$, or equivalently, $\tp P \ncong \tp Q$. Then $h=0$ holds, and $f$ factors through $\proj\Lambda$ according to Lemma \ref{G-lem:commdiag} (4). Thus, $ \uHom_\Lambda (\rad P, \rad Q)=0.$
\item Assume $P \cong Q$.
We show that there is an isomorphism $$\uHom_\Lambda (\rad P, \rad P) \to \End_\Lambda (\rad P)/ \rad_{\CM\Lambda} (\rad P, \rad P).$$
Indeed, as $\rad P$ is indecomposable non-projective by Assumption \ref{G-assumption}, we have $[\Lambda](\rad P, \rad P) \subset  \rad_{\CM\Lambda} (\rad P, \rad P)$.
Conversely, if $f \in \End_\Lambda (\rad P)$ is a non-isomorphism, then so is $g$. Hence, since $P$ and $Q$ are projective, $h$ is not an isomorphism. In particular, $h=0$ and by Lemma \ref{G-lem:commdiag} (4), we have $f \in [\Lambda](\rad P, \rad P)$. 

Therefore, we have $\len_{\End_\Lambda(\rad P)} \uHom_\Lambda (\rad P, \rad Q)=1.$
\end{enumerate}

(ii) Assume $P \cong   \nu^{-1}Q$. By Lemma \ref{G-lem:seqforlength}, we have the short exact sequence \eqref{G-eq:seqforlength}.
  \begin{enumerate}[({ii}-i)]
      \item Assume $P \ncong Q $. Then $\Hom_\Lambda(P,Q)/\rad_{\CM\Lambda}(P,Q)=0$.
Hence, we have an isomorphism $$  \Ext^1_\Lambda (\tp P,  Q) \cong \uHom_\Lambda(\rad P, \rad Q).$$
This is a one-dimensional $\End_\Lambda (\tp P)$-vector space by Lemma \ref{G-lem:onedim} (2). So by Lemma \ref{G-lem:residuefield}, it is a one-dimensional $k_{\rad P}$-vector space and we have $k_P \cong k_{\rad P}$. 

Therefore, $\len_{\End_\Lambda(\rad P)}\uHom_\Lambda(\rad P, \rad Q)=1.$
      \item Assume $P \cong  Q$.           
      By Lemma \ref{G-lem:seqforlength}, we have
              \begin{align*}
&\len_{\End_\Lambda(\rad P)}\uHom_\Lambda(\rad P, \rad Q)\\
=\, & \len_{\End_\Lambda(\rad P)}\Ext^1_\Lambda(\tp P,  P) + \len_{\End_\Lambda(\rad P)}  \End_\Lambda(P)/ \rad_{\CM\Lambda}(P,P)= 2. \qedhere
              \end{align*}
    \end{enumerate}
\end{proof}

\section{Reading hom-sets from Auslander-Reiten quivers}
\label{G-s:homspaces}
In this section, we give a simple method to read the hom-set between two objects in $\CM\Lambda$ by looking the related positions of these two objects in the Auslander-Reiten quiver. 
First, we introduce some combinatorial notions which will be used to calculate the length of the hom-sets.

For $x,y\in (\ZZ \Delta)_0$, we define 
$$\delta(y,x)
=\begin{cases}
1, \text{ if } y=x;\\
0, \text{ otherwise.}
\end{cases}$$
\begin{definition}
\label{G-def:hmor}
Let $Q$ be a stable valued translation quiver and $x$ be a fixed vertex in $Q_0$. 

(1) For each $n \in \ZZ$, we define a map $$h_{Q,n}(-, x):Q_0 \to \NN_{\ge 0}$$ as follows.
\begin{itemize}
\item  For $n <0$, $h_{Q,n}(-,x):=0$;
\item $h_{Q,0}(-,x):=\delta(-,x)$;          
\item  For $n>0$, we let $$h'_{Q,n}(y,x):= \sum_{v\in Q_0} d_{yv} h_{Q,n-1}(v,x) -  h_{Q,n-2}(\tau^{-1}y, x)$$ 
$$\and  h_{Q,n}(y,x):= \max\{h'_{Q,n}(y,x), 0\}.$$                          
\end{itemize}

(2) By using $h_{Q,n}(-,x)$, we define 
$$h_Q(-,x): = \sum_{n\ge 0}h_{Q,n}(-,x)$$
$$\and H_{Q}(x)=\{y\in Q_0 \mid h_{Q} (y,x) >0\}.$$

If there is no danger of confusion, we denote $h_{Q,n}$, $h_Q$ and $H_Q$ by $h_n$, $h$ and $H$ respectively.
\end{definition}

Let us see an example of type $A$.
\begin{example}
Let $x$ be the marked vertex in $(\ZZ E_6)_0$. The set $H(x)$ is the set consisting of all the vertices in the following diagram with positive value and the values on the vertices are the results of $h(-,x)$ acting on the vertices.
\begin{center}
 \textbf{\dots}
   \begin{tikzcd}[column sep=small, row sep=small]
 0\drar & &0\drar&& 1\drar & &0\drar && 1\drar & &1\drar && 0\drar&& 0 \\
 &0\urar\drar&&1\urar\drar& &1\urar\drar& &1\urar\drar& &2\urar\drar& &1\urar\drar& &0\urar\drar\\
0\urar\drar\rar&0\rar&0\urar\drar\rar&0\rar&1\urar\drar\rar&1\rar&2\urar\drar\rar&1\rar&2\urar\drar\rar&1\rar&2\urar\drar\rar& 1\rar&1\urar\drar\rar& 0\rar&0\\ 
 &0\urar\drar &&0\urar\drar& &1\urar\drar& &2\urar\drar& &1\urar\drar& &1\urar\drar& &\node[draw,circle,inner sep=1pt]{1}; \urar\drar\\
0 \urar&&0\urar &&0 \urar&&1 \urar&&1 \urar&&0\urar&&1 \urar&&0
 \end{tikzcd}  \textbf{\dots}
 \end{center}
\end{example}

The following theorem is about how to read hom-sets in $\uCM\Lambda$ from the stable Auslander-Reiten quiver $\underline{\mathfrak{A}}(\Lambda)$ by using the map $h$ defined in Definition \ref{G-def:hmor}.
\begin{theorem}
\label{G-thm:handhom}
Let $\Lambda$ be a representation-finite $R$-order (not necessarily Gorenstein) and $\underline{\mathfrak{A}}(\Lambda)$ the stable Auslander-Reiten quiver of $\Lambda$. For any $X,Y \in \ind(\uCM \Lambda)$, we have 
$$\len_{\End_\Lambda(Y)} \uHom_\Lambda(Y,X)=h_{\underline{\mathfrak{A}}(\Lambda)}(Y,X).$$
In particular, for any $X \in \ind(\uCM \Lambda)$, we have 
$$H_{\underline{\mathfrak{A}}(\Lambda)}(X)=\{ Y  \in \ind(\uCM\Lambda) \mid  \uHom_\Lambda(Y,X)\neq 0\}.$$ 
\end{theorem}

The remaining part of this section is devoted to prove Theorem \ref{G-thm:handhom}.

Let $\ZZ Q_0$ be the free Abelian group generated by $Q_0$. For any $P=\sum_{x\in Q_0} p_x x \in \ZZ Q_0$, we denote $P_+:=\sum_{x\in Q_0} \max\{p_x,0\}x$ and denote by $\supp P$ the subset $\{x\in Q_0 \mid p_x \neq 0\}$ of $Q_0$.

\begin{definition}
\label{G-def:HQbox}
Let $Q$ be a stable valued translation quiver and $\NN_{\ge 0} Q_0$ the free Abelian monoid generated by $Q_0$.
\begin{itemize}
      \item[(1)]  We define a map $\theta: \NN_{\ge 0} Q_0 \to \NN_{\ge 0} Q_0$ by $\theta(x) := \sum_{y \to x \in Q_1} d_{yx} y$.
       \item[(2)]  For $n \in \NN_{\ge 0}$ and $x\in Q_0$, we define $\theta_n(x) \in \NN_{\ge 0} Q_0$ by $$\theta_n(x):=
\begin{cases}
x, &\text{ if } n=0;\\
\theta(x), &\text{ if } n=1;\\
(\theta(\theta_{n-1}(x))-\tau(\theta_{n-2}(x)))_+, & \text{ if } n\ge 2.
\end{cases}$$ 
\end{itemize}
\end{definition}

\begin{lemma}
\label{G-rem:handhq}
Let $Q$ be a stable valued translation quiver. For any $x\in Q_0$ and $n \in \NN_{\ge 0}$, we have $$ \theta_n(x)=\sum_{y\in Q_0} h_{n}(y,x)y \and H(x) = \bigcup_{i\ge 0} \supp \theta_i (x).$$
\end{lemma}
\begin{proof}
This can be shown inductively by using the definitions.
\end{proof}

Consider a Gorenstein $R$-order $\Lambda$. The stable Auslander-Reiten quiver $\underline{\mathfrak{A}}(\Lambda)$ of $\Lambda$ is a stable translation quiver. 
We consider the bijection that sends the set of isomorphism classes of objects in $\uCM\Lambda$ to $\NN_{\ge 0} \underline{\mathfrak{A}}(\Lambda)_0=\NN_{\ge 0} \ind(\uCM\Lambda)$ by
$$X\cong \bigoplus_{i=1}^l M_i^{t_i} \mapsto X=\sum_{i=1}^l t_i M_i,$$
where $M_i\in\underline{\mathfrak{A}}(\Lambda)_0$.
Hence, for any $X \in \uCM\Lambda$, we have $\theta_n(X) \in  \uCM\Lambda$ naturally.

\begin{theorem}[{\cite[Theorems 4.1, 7.1]{tauone}}]
\label{G-th:tauonerad}
Let $\Lambda$ be an $R$-order and $X\in \uCM\Lambda$. We have a surjective morphism 
$$\uHom_\Lambda(-, \theta_n(X)) \to \rad^n_{\uCM\Lambda}(-,X)$$
 of functors which induces an isomorphism of functors
$$\rad^n_{\uCM\Lambda}(-,X)/\rad^{n+1}_{\uCM\Lambda}(-,X) \cong \uHom_\Lambda(-,\theta_n(X))/ \rad_{\uCM\Lambda}(-,\theta_n(X)).$$
\end{theorem}

\begin{proof}[Proof of Theorem \ref{G-thm:handhom}]
As $\Lambda$ is a representation-finite $R$-order, there exists an integer $l$ such that $$\rad^l_{\uCM\Lambda}(-,-)=0.$$
For any $X,Y \in \underline{\mathfrak{A}}(\Lambda)_0$, we have 
\begin{align*}
\len_{\End_\Lambda(Y)} \uHom_\Lambda(Y,X) &= \sum_{n\ge 0} \len_{\End_\Lambda(Y)} (\rad^n_{\uCM\Lambda}(Y,X)/\rad^{n+1}_{\uCM\Lambda}(Y,X))\\
&\overset{\ref{G-th:tauonerad}}{=} \sum_{n\ge 0} \len_{\End_\Lambda(Y)} (\uHom_\Lambda(Y,\theta_n(X))/ \rad_{\uCM\Lambda}(Y,\theta_n(X)))\\
&=\sum_{n\ge 0}(\text{multiplicity of $Y$ in }\theta_n(X))\\
&\overset{\ref{G-rem:handhq}}{=}\sum_{n\ge 0} h_{\underline{\mathfrak{A}}(\Lambda),n}(Y,X)=h_{\underline{\mathfrak{A}}(\Lambda)}(Y,X).\qedhere
\end{align*}
\end{proof}

\section{Combinatorial configuration}
\label{G-s:combinatorial}
In this section, we will give a combinatorial description of configurations by using the map $h$ defined in Definition \ref{G-def:hmor} for valued stable translation quivers.
First, we have the following proposition for Dynkin quivers.
\begin{proposition}
\label{G-lem:h0}
Let $ \Delta$ be a Dynkin diagram and $x \in (\ZZ  \Delta)_0$. There exist a positive integer $m$ and a vertex $y\in (\ZZ \Delta)_0$ such that $$h_{m-1}(-,x)=\delta(-,y) \and h_{n}(-,x )=\begin{cases} h'_{n}(-,x), &\text{ if }n<m;\\ 0,&\text{ if }n\ge m. \end{cases}$$
\end{proposition}
\begin{proof}
For each Dynkin diagram, the assertion is easy to check by calculation. 
(Computations of Sections \ref{G-s:typeA}, \ref{G-s:typeBC} and \ref{G-s:typeD} will make it clear.
See Figure \ref{G-serrebox} for type $A$ in Section \ref{G-s:typeA}, Figure \ref{G-HboxB} for type $B$, $C$ in Sections \ref{G-s:typeBC}, Case 1 and Case 2 in Section \ref{G-s:typeD} for type $D$.)
\end{proof}
By using this proposition, we define some morphisms as follows.
\begin{definition}
Let $\Delta$ be a Dynkin diagram and $G\subset \Aut(\ZZ\Delta)$ a weakly admissible group.

(1) For each $x \in (\ZZ  \Delta)_0$, we define $\omega(x):=y$, where $y$ is the vertex satisfying $h_{m-1}(-,x)=\delta(-,y)$ in Proposition \ref{G-lem:h0}. 
Then $\omega$ gives an automorphism of $\ZZ\Delta$.

(2) Since $\omega (g x)= g(\omega x)$ holds for each $g \in G$ and $x \in (\ZZ\Delta)_0$,
then $\omega$ induces an automorphism of the quotient quiver $\ZZ\Delta/G$ defined in Subsection \ref{G-def:translation} given by $\omega(Gx)=G(\omega(x))$ for any $x\in(\ZZ\Delta)_0$.
\end{definition}
For any $x \in (\ZZ\Delta)_0$, we have
\begin{equation}
\label{G-eq:hdynkin}
h_{\ZZ\Delta}(\omega(x), x)=1.
\end{equation}
By using $\omega$ and the map $h$, we define combinatorial configurations as follows.
\begin{definition}
\label{G-def:configuration}
Let $Q$ be a stable valued translation quiver of the form $\ZZ\Delta/G$, where $\Delta$ is a Dynkin diagram and $G\subset \Aut(\ZZ\Delta)$ is a weakly admissible group.
A \emph{configuration} $C$ of $Q$ is a set of vertices satisfying the following two conditions: 

    (C1) For any vertex $x \in Q_0$, there exists a vertex $c \in C$ such that $h(x,c)>0$.  In other words, we have $$Q_0 =\bigcup_{c\in C}H(c).$$
    
     (C2) $\omega(C)=C$ holds. Moreover, for any vertices $c,d \in C$, we have $$h(d,  c)=
\begin{cases}
2, &\text{ if }   d =c = \omega( c);\\
1, &\text{ if } d=c \neq\omega( c);\\
1, &\text{ if }  d=\omega( c) \neq c;\\
0,  &\text{ otherwise}.
\end{cases}$$
\end{definition}

\begin{remark}
\label{G-rem:cyclic}
When $Q=\ZZ\Delta$, then (C2) is equivalent to $C \cap H(c)=\{c, \omega(c)\}$ by \eqref{G-eq:hdynkin}.
\end{remark}

Now we compare the maps $h_{\ZZ\Delta}$ and $h_{\ZZ\Delta/G}$ in the following lemma. 
\begin{lemma}
\label{G-lem:hequation}
Let $ \Delta$ be a Dynkin diagram and $G\subset \Aut(\ZZ  \Delta)$ a weakly admissible group.
Consider the canonical map $\pi: (\ZZ\Delta)_0 \to (\ZZ\Delta)_0/G$ associated to the translation quiver $\ZZ  \Delta$ and $G$. 
For any $x,y \in (\ZZ  \Delta)_0$, we have $$h_{\ZZ  \Delta/G}(\pi y, \pi x)=\sum_{ y' \in Gy} h_{\ZZ  \Delta}(y', x).$$
\end{lemma}
\begin{proof}
According to Definition \ref{G-def:hmor}, it is sufficient to prove that $$h_{\ZZ  \Delta/G,n}(\pi y, \pi x)=\sum_{y' \in Gy} h_{\ZZ  \Delta,n}(y', x)$$ holds for any $x,y\in (\ZZ  \Delta)_0$ and $n\ge 0$. We use induction on $n$.

This is clear for $n=0$.

By the definition of quotients of valued translation quivers by weakly admissible automorphism groups in Subsection \ref{G-def:translation}, we have 
$$d_{\pi y \pi v}=\sum_{  y' \in Gy} d_{y' v}.$$ 
Assume that $n\ge 1$. Suppose that $h_{\ZZ  \Delta/G,i}(\pi y, \pi x)=\sum_{y' \in Gy} h_{\ZZ  \Delta,i}(y', x)$ holds for any $i\le n-1$.
We have 
\begin{align*}
\sum_{\pi v\in (\ZZ\Delta/G)_0} d_{\pi y \pi v} h_{\ZZ\Delta/G,n-1}(\pi v, \pi x)
&= \sum_{\pi v\in (\ZZ\Delta/G)_0} d_{\pi y \pi v}  \sum_{v' \in Gv} h_{\ZZ  \Delta,n-1}(v', x)\\
&= \sum_{v\in (\ZZ\Delta)_0} d_{\pi y \pi v} h_{\ZZ  \Delta,n-1}(v, x)\\
&= \sum_{v \in (\ZZ\Delta)_0} \left(\sum_{y' \in Gy} d_{y' v}\right) h_{\ZZ  \Delta,n-1}(v, x)\\
 &=\sum_{y' \in Gy}\sum_{v \in (\ZZ\Delta)_0}d_{y' v} h_{\ZZ  \Delta,n-1}(v, x),
\end{align*}
\begin{align*}
\and h_{\ZZ  \Delta/G,n-2}(\pi (\tau^{-1}y), \pi x) &= \sum_{y' \in G(\tau^{-1}y)} h_{\ZZ  \Delta,n-2}(y', x)\\
&=\sum_{y' \in Gy} h_{\ZZ  \Delta,n-2}(\tau^{-1}y', x).
\end{align*}

By definition, we get 
\begin{align*}
h'_{\ZZ  \Delta/G,n}(\pi y, \pi x)&=\sum_{\pi v\in (\ZZ\Delta/G)_0} d_{\pi y \pi v} h_{\ZZ\Delta/G,n-1}(\pi v, \pi x) - h_{\ZZ  \Delta/G,n-2}(\pi (\tau^{-1}y), \pi x)\\
&= \sum_{y' \in Gy}\left(\sum_{v \in (\ZZ\Delta)_0}d_{y' v} h_{\ZZ  \Delta,n-1}(v, x)- h_{\ZZ  \Delta,n-2}(\tau^{-1}y', x)\right)\\
&=\sum_{y' \in Gy}h'_{\ZZ  \Delta,n}(y', x).
\end{align*}
By Proposition \ref{G-lem:h0}, there exists an integer $m$ such that $$\begin{cases}
h'_{\ZZ  \Delta,n}(-, x)\ge 0, &\text{ if } n<m;\\
h'_{\ZZ  \Delta,n}(-, x)\le 0,&\text{ if } n\ge m.
\end{cases}
$$
Therefore, we have
\begin{align*}
h_{\ZZ  \Delta/G,n}(\pi y, \pi x) &=\max\{h'_{\ZZ  \Delta/G,n}(\pi y, \pi x), 0\}\\
&=\max\left\{\sum_{y' \in Gy}h'_{\ZZ  \Delta,n}(y', x),0\right\}\\
&=\sum_{y' \in Gy}\max\{h'_{\ZZ  \Delta,n}(y', x), 0\}\\
&=\sum_{y' \in Gy} h_{\ZZ  \Delta,n}(y', x).\qedhere
\end{align*}
\end{proof}

By Lemmas \ref{G-rem:handhq}, \ref{G-lem:hequation} and Proposition \ref{G-lem:h0}, we get the following lemma by direct calculation.
\begin{lemma}
\label{G-lem:omegaforquotient}
Let $Q$ be a stable translation quiver of the form $\ZZ\Delta /G$, where $\Delta$ is a Dynkin diagram and $G \subset \Aut(\ZZ\Delta)$ is a weakly admissible group.
For any $X \in Q_0$, we have $$h_{\ZZ\Delta/G, m-1}(-,X)= \delta(-, \omega X) \and \omega X = \theta_{m-1} (X),$$ where $m$ is the same integer as in Proposition \ref{G-lem:h0} for $\ZZ\Delta$ and $x \in \pi^{-1}(X)$. Moreover, we get $$h_{\ZZ\Delta/G, k}(-,X)=0 \and \theta_k (X) =0 $$ for any $k \ge m$.
\end{lemma}

The definitions about configurations and $\omega$ are completely combinatorial, we show in the following proposition that the automorphism $\omega$ also has some categorical meaning.
Recall that in Theorem \ref{G-th:structure}, for any representation-finite Gorenstein $R$-order $\Lambda$, the stable Auslander-Reiten quiver $\underline{\mathfrak{A}}(\Lambda)$ is isomorphic to $ \ZZ \Delta/G$ for some Dynkin diagram $\Delta$ and some weakly admissible group $G \subset \Aut(\ZZ\Delta)$. 
\begin{proposition}
Let $\Lambda$ be a representation-finite Gorenstein $R$-order with $\underline{\mathfrak{A}}(\Lambda) \cong \ZZ \Delta/G$. For any $X \in \ind(\uCM \Lambda)$, we have $\omega X = \nu^{-1} X$.
\end{proposition}
\begin{proof}
Fix any $X \in \ind(\uCM\Lambda)$. 
We denote $\bigoplus_{Y\in \ind(\uCM\Lambda)}Y$ by $M$.
Then $\uHom_\Lambda (\nu^{-1} X,M)$ is an indecomposable projective $\uEnd_\Lambda(M)^{\operatorname{op}}$-module. We denote by $ \tp_{\uEnd_\Lambda(M)^{\operatorname{op}}} \uHom_\Lambda(\nu^{-1}X,M)$ the top of $ \uHom_\Lambda(\nu^{-1} X,M)$.
By Proposition \ref{G-pro:nakayama} (2), we have $\uHom_\Lambda(\nu^{-1}X,M)\cong \Dk\uHom_\Lambda(M, X)$ as $\uEnd_\Lambda(M)^{\operatorname{op}}$-modules by functoriallity.
So we have \begin{align*}
\soc_{\uEnd_\Lambda(M)} \uHom_\Lambda(M, X) & \cong \soc_{\uEnd_\Lambda(M)} \Dk \uHom_\Lambda(\nu^{-1}X,M) \\
& \cong \Dk \tp_{\uEnd_\Lambda(M)^{\operatorname{op}}} \uHom_\Lambda(\nu^{-1}X,M),
\end{align*} 
where $\soc_{\uEnd_\Lambda(M)} \uHom_\Lambda(M, X)$ is the socle of $ \uHom_\Lambda(M, X)$.
Moreover, we have $$\Dk \tp_{\uEnd_\Lambda(M)^{\operatorname{op}}} \uHom_\Lambda(\nu^{-1}X,M) \cong  \tp_{\uEnd_\Lambda(M)} \uHom_\Lambda(M,\nu^{-1}X).$$

Thanks to Theorem \ref{G-th:tauonerad}, we have the following surjective morphism
$$\uHom_\Lambda(-,\theta_n(X)) \to \rad^n_{\uCM\Lambda}(-,X)$$
which induces an isomorphism of functors 
\begin{equation*}
\rad^n_{\uCM\Lambda}(-,X)/\rad^{n+1}_{\uCM\Lambda}(-,X) \cong \uHom_\Lambda(-,\theta_n(X))/ \rad_{\uCM\Lambda}(-,\theta_n(X)).
\end{equation*}
By Lemma \ref{G-lem:omegaforquotient}, $$\theta_{m-1}(X) \neq 0 \and \theta_m (X)=0$$ hold.
Hence, we get $$\rad^{m-1}_{\uCM\Lambda}(M,X) \neq 0  \and \rad^m_{\uCM\Lambda}(M,X)=0.$$

Thus, we have 
\begin{align*}
 & \tp_{\uEnd_\Lambda(M)} \uHom_\Lambda(M,\omega X)  \cong \rad_{\uCM\Lambda}^{m-1} (M, X)\\
\subset  & \soc_{\uEnd_\Lambda(M)} \uHom_\Lambda(M, X)  \cong \tp_{\uEnd_\Lambda(M)} \uHom_\Lambda(M,\nu^{-1}X).
\end{align*}
Therefore, $\uHom_\Lambda(M,\omega X) \cong \uHom_\Lambda(M,\nu^{-1}X)$ as an $\uEnd_\Lambda(M)$-module. As $\add M$ is dense in $\uCM\Lambda$, we have $\uHom_\Lambda(-,\omega X) \cong \uHom_\Lambda(-,\nu^{-1}X)$. By Yoneda Lemma, we get $\omega X \cong \nu^{-1} X$.
\end{proof}

By using the above proposition and the categorial properties of the radicals we showed in Section \ref{G-s:categorical}, we can prove the following main theorem now.
\begin{theorem}
\label{G-thm:mainthm1}
Let $R$ be a complete discrete valuation ring and $\Lambda$ a ring-indecomposable represen\-tation-finite Gorenstein $R$-order. 
Assume the following conditions satisfied:
             \begin{itemize}
                    \item $\rad P$ is indecomposable and non-projective for any $P\in \ind(\proj \Lambda)$;
                    \item $\rad P\ncong \rad Q$ when $P \ncong Q \in \ind(\proj\Lambda)$. 
             \end{itemize}
Then the Auslander-Reiten quiver $\mathfrak{A}(\Lambda)$ of $\CM\Lambda$ is isomorphic to $(\ZZ\Delta/G)_C,$ where $\Delta$ is a Dynkin diagram, $G\subset \Aut(\ZZ\Delta)$ is a weakly admissible group and $ C$ is a configuration of $\ZZ\Delta/G$.
\end{theorem}
\begin{proof}
By Theorem \ref{G-th:structure}, the Auslander-Reiten quiver $\underline{\mathfrak{A}}(\Lambda)$ of $\uCM\Lambda$ is isomorphic to the stable translation quiver $\ZZ \Delta/G$ for some Dynkin diagram $\Delta$ and some weakly admissible group $G \subset \Aut(\ZZ\Delta)$. 

Consider the set $$ C_\Lambda:=\{ \rad P \mid P \text{ is indecomposable projective}\}.$$
We will show that $C_\Lambda$ is a configuration.
According to Theorem \ref{G-pro:c1}, for any $X\in \underline{\mathfrak{A}}(\Lambda)$, there exists an indecomposable projective $\Lambda$-module $P$ such that $\uHom_\Lambda(X, \rad P)\neq 0.$ It follows Theorem \ref{G-thm:handhom} that $X\in H_{\underline{\mathfrak{A}}(\Lambda)}(\rad P).$ Hence, $ C_\Lambda$ satisfies the condition (C1) in Definition \ref{G-def:configuration}.
By Theorem \ref{G-pro:c2}, we know that 
$$\len_{\End_\Lambda(\rad P)} \uHom_\Lambda(\rad P, \rad Q)=\begin{cases}
2, &\text{ if } P \cong Q \cong \nu^{-1} Q;\\
1, &\text{ if } P \cong Q \ncong \nu^{-1} Q;\\
1, &\text{ if } P \cong \nu^{-1} Q \ncong Q;\\
0, &\text{ otherwise}.
\end{cases}
$$
Since $\rad P$ and $\rad Q$ are non-projective and indecomposable, according to Theorem \ref{G-thm:handhom}, it means that 
$$h_{\underline{\mathfrak{A}}(\Lambda)}(\rad P,  \rad Q)=
\begin{cases}
2, &\text{ if }  \rad P \cong \rad Q \cong \omega(\rad Q);\\
1, &\text{ if } \rad P \cong \rad Q \ncong \omega(\rad Q);\\
1, &\text{ if }  \rad P \cong \omega(\rad Q) \ncong \rad Q;\\
0,  &\text{ otherwise}.
\end{cases}
$$
Since $ \omega( C_\Lambda) = \nu^{-1}(C_\Lambda)=C_\Lambda$, it follows that (C2) holds. Therefore, $ C_\Lambda$ is a configuration of $\ZZ \Delta/G$ by Definiton \ref{G-def:configuration}.

By Remark \ref{G-rem:almostsplit} (2), $\rad P \to P$ is the only valued arrow ending at $P$ and $P \to \corad P$ is the only valued arrow starting from $P$. Since $\rad P$ is indecomposable non-projective, we have $\tau(\corad P) = \rad P$.
Finally, because $\rad P \ncong \rad Q$ when $P \ncong Q \in \ind(\proj \Lambda)$, the Auslander-Reiten quiver $\mathfrak{A}(\Lambda)$ is isomorphic to $(\ZZ\Delta/G)_{C_\Lambda}$.
\end{proof}

The remaining part of the section is devoted to study connections between combinatorial configurations of $\ZZ\Delta$ and combinatorial configurations of $\ZZ\Delta/G$. 
By using Lemma \ref{G-lem:hequation}, we show in the following proposition that configurations of $\ZZ\Delta/G$ are essentially the same as configurations of $ \ZZ\Delta$.
\begin{proposition}
\label{G-pro:covering}
Let $\Delta$ be a Dynkin diagram and $G$ a weakly admissible automorphism group of $\ZZ\Delta$.
The canonical map $\pi: (\ZZ\Delta)_0 \to (\ZZ\Delta)_0/G$ gives the following one-to-one correspondence 
$$\{\text{configurations $C$ in }\ZZ\Delta \text{ satisfying } GC=C\}  \xlongleftrightarrow{1-1} \{\text{configurations in }\ZZ\Delta/G \}.$$
\end{proposition}
\begin{proof}
Let us show that a set $ C$ of vertices of $ \ZZ\Delta/G$ is a configuration of $\ZZ\Delta/G$ if and only if $\pi^{-1}( C)$ is a configuration of $ \ZZ\Delta$.
According to Lemma \ref{G-lem:hequation}, for any $x,y \in (\ZZ\Delta)_0$, we have $$h_{\ZZ\Delta/G}(\pi x, \pi y)=\sum_{x' \in Gx}h_{\ZZ\Delta}(x',y).$$
Hence, we have $$ \bigcup_{x' \in Gx} H_{\ZZ\Delta}(x')= \pi^{-1}(H_{\ZZ\Delta/G}(\pi x))$$ for any $x \in (\ZZ\Delta)_0$.
Thus, $C$ satisfies (C1) if and only $\pi^{-1}(C)$ satisfies (C1).

Let $c, d \in \pi^{-1}( C)$. 
If 
\begin{equation}
\label{G-eq:quotienteq}
h_{\ZZ\Delta/G}(\pi d, \pi c)=
\begin{cases}
2, &\text{ if } \pi d =\pi c=\omega(\pi c);\\
1, &\text{ if } \pi d =\pi c\neq\omega(\pi c);\\
1, &\text{ if } \pi d =\omega(\pi c)\neq \pi c;\\
0, &\text{ otherwise},
\end{cases}
\end{equation}
holds, then since $h_{\ZZ\Delta}(\omega c,c)\ge h_{\ZZ\Delta, m-1}(\omega c,c)=1$ by Proposition \ref{G-lem:h0}, $h_{\ZZ\Delta}(c,c)\ge h_{\ZZ\Delta,0}(c,c)=1$ and $\nu(c) \neq c$ hold for any $c\in (\ZZ\Delta)_0$, it follows that 
\begin{equation}
\label{G-eq:originaleq}
h_{\ZZ\Delta}(d, c)=
\begin{cases}
1, &\text{ if }  d = c\neq\omega(c);\\
1, &\text{ if }  d =\omega( c)\neq c;\\
0, &\text{ otherwise},
\end{cases}
\end{equation}
holds.
On the other hand, if \eqref{G-eq:originaleq} holds, then \eqref{G-eq:quotienteq} holds by Lemma \ref{G-lem:hequation} and the definition of weakly admissible automorphism groups in Subsection \ref{G-def:translation}.
We also have that $\omega (C) =C$ if and only if $\omega(\pi^{-1}(C))= \pi^{-1}(\omega(C))=\pi^{-1}(C)$.
Hence, $C$ satisfies (C2) if and only $\pi^{-1}( C)$ satisfies (C2).
Therefore, $ C$ is a configuration of $\ZZ\Delta/G$ if and only if $\pi^{-1}( C)$ is a configuration of $\ZZ\Delta$.
\end{proof}

\section{Type $A$}
\label{G-s:typeA}
In this section, we will describe all the configurations of Dynkin type $A$. 
In his paper \cite{Wiedemannorder}, Wiedemann described configurations in terms of Brauer relations with \emph{Stra{\ss}eneigenschaft} \cite[Satz, p. 47]{Wiedemannorder}.
In this section, we simplify his description by introducing a much simpler notion of \emph{$2$-Brauer relations} and give nice bijections with Wiedemann's configurations. 
We start by introducing $2$-Brauer relations.
\begin{definition}
\label{G-def:Brauer}
We consider a disk with $n$ marked points $1,2, \ldots, n$ at the boundary in the clockwise order.

(1) \cite[2.6 Definition]{riedtmannA} A \emph{Brauer relation $B$ of rank $n$} is an equivalence relation on the set $ \{1,2, \ldots, n \}$ such that the convex hulls of distinct equivalence classes are disjoint.

(2) A \emph{$2$-Brauer relation $B$ of rank $n$} is a Brauer relation of rank $n$ which only allows at most two elements in any equivalence class.
\end{definition}
For a $2$-Brauer relation $B$, we define a permutation $\sigma_B$ as follows 
\begin{align*}
     \sigma_B(i):= \begin{cases}
        i, \text{ if $i$ itself is an equivalence class;}\\
        j, \text{ if $\{i, j \}$ is an equivalence class.}
                     \end{cases}
\end{align*}
We denote by $\mathbf{B}_n^2$ the set of $2$-Brauer relations of rank $n$ and by $i\sim j$ if $i$ is equivalent to $j$.

Here are some examples of $2$-Brauer relations.
\begin{example}
\label{G-ex:brauerA5}
All the $2$-Brauer relations of rank $4$ are shown as follows.

\begin{center}   
 \begin{tabular}{cccccc}
   \begin{xy}
    		0;<1cm,0cm>:<0cm,-1cm>::,="D",{\xypolygon50{~>.}},
    		{\xypolygon4"A"{~*{\xypolynode}~>{}}}, 
    \end{xy} &\hspace{0.5cm} &
    \begin{xy}
    		0;<1cm,0cm>:<0cm,-1cm>::,="D",{\xypolygon50{~>.}},
    		{\xypolygon4"A"{~*{\xypolynode}~>{}}}, 
    		"A1";"A2"**\dir{-},
    \end{xy}  &
     
         \begin{xy}
    		0;<1cm,0cm>:<0cm,-1cm>::,="D",{\xypolygon50{~>.}},
    		{\xypolygon4"A"{~*{\xypolynode}~>{}}}, 
    		"A3";"A2"**\dir{-},
    \end{xy}  &
         
            \begin{xy}
    		0;<1cm,0cm>:<0cm,-1cm>::,="D",{\xypolygon50{~>.}},
    		{\xypolygon4"A"{~*{\xypolynode}~>{}}}, 
    		"A4";"A3"**\dir{-},
    \end{xy} &
      
            \begin{xy}
    		0;<1cm,0cm>:<0cm,-1cm>::,="D",{\xypolygon50{~>.}},
    		{\xypolygon4"A"{~*{\xypolynode}~>{}}}, 
    		"A4";"A1"**\dir{-},
    \end{xy} 
           \\$B_1$
           & &$ B_2$
            &$B_3$
            & $B_4$
             &$B_5$
                 \end{tabular}
      \end{center}
      
      \begin{center}
      \begin{tabular}{ccccc} 
      \begin{xy}
    		0;<1cm,0cm>:<0cm,-1cm>::,="D",{\xypolygon50{~>.}},
    		{\xypolygon4"A"{~*{\xypolynode}~>{}}}, 
    		"A3";"A1"**\dir{-},
    \end{xy}  &  
    
      \begin{xy}
    		0;<1cm,0cm>:<0cm,-1cm>::,="D",{\xypolygon50{~>.}},
    		{\xypolygon4"A"{~*{\xypolynode}~>{}}}, 
    		"A2";"A4"**\dir{-},
    \end{xy}  & \hspace{0.5cm} &
    
            \begin{xy}
    		0;<1cm,0cm>:<0cm,-1cm>::,="D",{\xypolygon50{~>.}},
    		{\xypolygon4"A"{~*{\xypolynode}~>{}}}, 
    		"A1";"A2"**\dir{-},
                "A3";"A4"**\dir{-},
    \end{xy} & 
    
      \begin{xy}
    		0;<1cm,0cm>:<0cm,-1cm>::,="D",{\xypolygon50{~>.}},
    		{\xypolygon4"A"{~*{\xypolynode}~>{}}}, 
    		"A2";"A3"**\dir{-},
                "A1";"A4"**\dir{-},
    \end{xy} 
         \\$ B_6$
          &$B_7$
          &&$B_8$
          &$B_9$
       \end{tabular}
       \end{center}
%
%
     
In this example, we have $\sigma_{B_6}(1)=3$, $\sigma_{B_6}(2)=2$, $\sigma_{B_6}(3)=1$ and $\sigma_{B_6}(4)=4$. 
\end{example}

\begin{proposition}
\label{pro:cardinalA}
The cardinal number $M(n)$ of $\mathbf{B}_n^2$ is given by
$$M(n)=M(n-1)+ \sum_{i=0}^{n-2}M(i) M(n-2-i)=\sum_{k=0}^{\lfloor n/2 \rfloor} \frac{n!}{(n-2k)!(k+1)! k!}.$$
\end{proposition}
\begin{proof}
For any $2$-Brauer relation $B$ of rank $n$, if $\sigma_B(1)=1$, then the number of this kind of $2$-Brauer relations is $M(n-1)$.
If $1 \sim i+2$ ($0\le i \le n-2$), we divide the corresponding disk with $n$ marked points into two parts by drawing a diagonal connecting $1$ and $i+2$, then one part has $i$ marked points and the other part has $n-2-i$ marked points. The number of such kind of $2$-Brauer relations is $M(i)M(n-2-i)$. We get the recursive formula.

Moreover, $M(n)$ are known as Motzkin numbers and given by the following formula (see \cite{Motzkin1, Motzkin}):
\begin{align*}M(n)=\sum_{k=0}^{\lfloor n/2 \rfloor} \frac{n!}{(n-2k)!(k+1)! k!}. &\qedhere
\end{align*}
\end{proof}

In the rest, we consider the translation quiver $\ZZ A_{n+1}$ and associate to each vertex the label in $\ZZ/n\ZZ \times \ZZ/n\ZZ$ in Figure \ref{G-fig:labelA}. 
\begin{figure}
\begin{center}
\tiny{
\begin{tikzcd}[column sep=tiny, row sep=small]
&&n\!-\!1\ n\!-\!1 \drar & &n\ n\drar&&1\ 1 \drar & &2\ 2\drar[dotted,-] &&** \drar[dotted,-] & &n\!-\!1\ n\!-\!1\arrow[dotted,-]{rr}&&\phantom{X}\\
\arrow[dotted,-]{r}&n\!-\!1\ *\urar[dotted,-]\drar[dotted,-] & &n\ n\!-\!1\urar\drar&&1\ n\urar\drar& &2\ 1\urar\drar[dotted,-]& &**\urar[dotted,-]\drar[dotted,-]& &n\!-\!1\ n\!-\!2\urar\drar[dotted,-]\\
&&n\ 3\urar[dotted,-]\drar&&1\ n\!-\!1\urar\drar&&2\ n\urar\drar[dotted,-]&&**\urar[dotted,-]\drar[dotted,-]&&n\!-\!1\ *\urar[dotted,-]\drar[dotted,-]&&n\ *\arrow[dotted,-]{rr}&&\phantom{X}\\ 
\arrow[dotted,-]{r}&n\ 2\urar\drar & &1\ 3\urar[dotted,-]\drar &&2\ n\!-\!1\urar\drar[dotted,-]& &**\urar[dotted,-]\drar[dotted,-]& &n\!-\!1\ 1\urar[dotted,-]\drar& &n\ *\urar[dotted,-]\drar[dotted,-]\\
&&1\ 2\urar\drar&&2\ 3\urar[dotted,-]\drar[dotted,-]&&**\urar[dotted,-]\drar[dotted,-]&&n\!-\!1\ n\urar\drar&&n\ 1\urar[dotted,-]\drar&&1\ *\arrow[dotted,-]{rr}&&\phantom{X}\\ 
\arrow[dotted,-]{r}&1\ 1\urar & &2\ 2\urar &&**\urar[dotted,-]& &n\!-\!1\ n\!-\!1\urar& &n\ n\urar& &1\ 1\urar[dotted,-]\\
\end{tikzcd}
}
\end{center}
\caption{Labels of vertices of $\ZZ A_{n+1}$}
\label{G-fig:labelA}
\end{figure}
Notice that we have $[i\ j] = \omega([j\ i])$ as shown in Figure \ref{G-serrebox}. 
By Definition \ref{G-def:configuration}, a vertex $c$ is in a configuration $C$ if and only if $\omega(c)$ is in $C$. Therefore, the configuration contains all the vertices with the same label as $c$.
As a consequence, we can describe a configuration by these labels.
By calculation, the set $H([j\ i])$ defined in \ref{G-def:hmor} consists of all the vertices covered by the rectangle with the four vertices $[j\ i],[j\ j],[i\ j]$ and $[i\ i]$ in Figure \ref{G-serrebox}.
\begin{figure}[!h]
\scalebox{1} 
{
\begin{pspicture}(0,-2.408047)(10.82,2.408047)
\psline[linewidth=0.02cm](0.0,1.7096485)(10.8,1.7096485)
\psline[linewidth=0.02cm](0.0,-1.8903515)(10.8,-1.8903515)
\psline[linewidth=0.02cm](5.2,1.7096485)(2.8,-0.69035155)
\psline[linewidth=0.02cm](2.8,-0.69035155)(4.0,-1.8903515)
\psline[linewidth=0.02cm](4.0,-1.8903515)(6.4,0.50964844)
\psline[linewidth=0.02cm](5.2,1.7096485)(6.4,0.50964844)
\psline[linewidth=0.02cm](1.6,-1.8903515)(2.8,-0.69035155)
\psline[linewidth=0.02cm](6.4,0.50964844)(7.6,1.7096485)
\usefont{T1}{ptm}{m}{n}
\rput(5.1614552,2.1146485){$i\ i$}
\usefont{T1}{ptm}{m}{n}
\rput(7.761455,2.1146485){$j\ j$}
\usefont{T1}{ptm}{m}{n}
\rput(1.4614551,-2.1853516){$i\ i$}
\usefont{T1}{ptm}{m}{n}
\rput(3.861455,-2.1853516){$j\ j$}
\usefont{T1}{ptm}{m}{n}
\rput(1.761455,-0.5853516){$i\ j=\mu(j\ i)$}
\usefont{T1}{ptm}{m}{n}
\rput(6.711455,0.41464844){$j\ i$}
\usefont{T1}{ptm}{m}{n}
\rput(4.85214551,0.1253516){$H([j\ i]) $}
\end{pspicture} 
}
   \caption{}
   \label{G-serrebox}
 \end{figure}
 
 By using the observations above, we prove the following theorem.
  \begin{theorem}
 \label{G-thm:correspondenceA}
We denote by $\mathbf{C}(A_{n+1})$ the set of configurations of $\ZZ A_{n+1}$. There is a bijection 
\begin{align*}
\Phi: \mathbf{C}(A_{n+1})  \to \mathbf{B}_n^2, 
 \end{align*}
  where $\Phi(C)$ for $C\in  \mathbf{C}(A_{n+1})$ is the equivalence relation on the set $\{1,2, \ldots,n\}$ generated by $i \sim j$ for each $[i\ j] \in C$, and $\Phi^{-1}(B)$ for $B \in \mathbf{B}_n^2$ is $\{ [1 \ \sigma_B(1)], [2 \ \sigma_B(2)] ,\ldots, [n \ \sigma_B(n)]\}$.
 \end{theorem}
 
The following is an example.
 \begin{example}
The correspondence $\mathbf{B}_4^2 \xlongleftrightarrow{1-1} \mathbf{C}(A_5)  $ is the following, where the $2$-Brauer relations $B_i$ for $1\le i \le 9$ are listed in Example \ref{G-ex:brauerA5}.
\begin{center}
$B_1$ \hspace{0.5cm} \begin{tikzcd}[column sep=tiny, row sep=tiny]
&&\node[draw,circle,inner sep=1pt]{33}; \drar & &\node[draw,circle,inner sep=1pt]{44};\drar&&\node[draw,circle,inner sep=1pt]{11}; \drar & &\node[draw,circle,inner sep=1pt]{22};\drar &&\node[draw,circle,inner sep=1pt]{33}; \drar & &\node[draw,circle,inner sep=1pt]{44};\arrow[dotted,-]{rr}&&\phantom{X}\\
\arrow[dotted,-]{r}&32\urar\drar & &43\urar\drar&&14\urar\drar& &21\urar\drar& &32\urar\drar& &43\urar\drar\\
&&42\urar\drar&&13\urar\drar&&24\urar\drar&&31\urar\drar&&42\urar\drar&&13\arrow[dotted,-]{rr}&& \phantom{X}\\ 
\arrow[dotted,-]{r}&41\urar\drar & &12\urar\drar &&23\urar\drar& &34\urar\drar& &41\urar\drar& &12\urar\drar\\
&&\node[draw,circle,inner sep=1pt]{11}; \urar&&\node[draw,circle,inner sep=1pt]{22};\urar &&\node[draw,circle,inner sep=1pt]{33}; \urar&&\node[draw,circle,inner sep=1pt]{44}; \urar&&\node[draw,circle,inner sep=1pt]{11}; \urar&&\node[draw,circle,inner sep=1pt]{22}; \arrow[dotted,-]{rr}&&\phantom{X}\\
\end{tikzcd}
\end{center}

\begin{center}
$B_2$  \hspace{0.5cm} \begin{tikzcd}[column sep=tiny, row sep=tiny]
&&\node[draw,circle,inner sep=1pt]{33}; \drar & &\node[draw,circle,inner sep=1pt]{44};\drar&&11 \drar & &22\drar &&\node[draw,circle,inner sep=1pt]{33}; \drar & &\node[draw,circle,inner sep=1pt]{44};\arrow[dotted,-]{rr}&&\phantom{X}\\
\arrow[dotted,-]{r}&32\urar\drar & &43\urar\drar&&14\urar\drar& &\node[draw,circle,inner sep=1pt]{21};\urar\drar& &32\urar\drar& &43\urar\drar\\
&&42\urar\drar&&13\urar\drar&&24\urar\drar&&31\urar\drar&&42\urar\drar&&13\arrow[dotted,-]{rr}&& \phantom{X}\\ 
\arrow[dotted,-]{r}&41\urar\drar & &\node[draw,circle,inner sep=1pt]{12};\urar\drar &&23\urar\drar& &34\urar\drar& &41\urar\drar& &\node[draw,circle,inner sep=1pt]{12};\urar\drar\\
&&11 \urar&&22\urar &&\node[draw,circle,inner sep=1pt]{33}; \urar&&\node[draw,circle,inner sep=1pt]{44}; \urar&&11 \urar&&22\arrow[dotted,-]{rr}&&\phantom{X}\\
\end{tikzcd}
\end{center}

\begin{center}
$B_3$  \hspace{0.5cm} \begin{tikzcd}[column sep=tiny, row sep=tiny]
&&33 \drar & &\node[draw,circle,inner sep=1pt]{44};\drar&&\node[draw,circle,inner sep=1pt]{11}; \drar & &22\drar &&33 \drar & &\node[draw,circle,inner sep=1pt]{44};\arrow[dotted,-]{rr}&&\phantom{X}\\
\arrow[dotted,-]{r}&\node[draw,circle,inner sep=1pt]{32};\urar\drar & &43\urar\drar&&14\urar\drar& &21\urar\drar& &\node[draw,circle,inner sep=1pt]{32};\urar\drar& &43\urar\drar\\
&&42\urar\drar&&13\urar\drar&&24\urar\drar&&31\urar\drar&&42\urar\drar&&13\arrow[dotted,-]{rr}&& \phantom{X}\\ 
\arrow[dotted,-]{r}&41\urar\drar & &12\urar\drar &&\node[draw,circle,inner sep=1pt]{23};\urar\drar& &34\urar\drar& &41\urar\drar& &12\urar\drar\\
&&\node[draw,circle,inner sep=1pt]{11}; \urar&&22\urar &&33 \urar&&\node[draw,circle,inner sep=1pt]{44}; \urar&&\node[draw,circle,inner sep=1pt]{11}; \urar&&22\arrow[dotted,-]{rr}&&\phantom{X}\\
\end{tikzcd}
\end{center}

\begin{center}
$B_4$  \hspace{0.5cm} \begin{tikzcd}[column sep=tiny, row sep=tiny]
&&33 \drar & &44\drar&&\node[draw,circle,inner sep=1pt]{11}; \drar & &\node[draw,circle,inner sep=1pt]{22};\drar &&33 \drar & &44\arrow[dotted,-]{rr}&&\phantom{X}\\
\arrow[dotted,-]{r}&32\urar\drar & &\node[draw,circle,inner sep=1pt]{43};\urar\drar&&14\urar\drar& &21\urar\drar& &32\urar\drar& &\node[draw,circle,inner sep=1pt]{43};\urar\drar\\
&&42\urar\drar&&13\urar\drar&&24\urar\drar&&31\urar\drar&&42\urar\drar&&13\arrow[dotted,-]{rr}&& \phantom{X}\\ 
\arrow[dotted,-]{r}&41\urar\drar & &12\urar\drar &&23\urar\drar& &\node[draw,circle,inner sep=1pt]{34};\urar\drar& &41\urar\drar& &12\urar\drar\\
&&\node[draw,circle,inner sep=1pt]{11}; \urar&&\node[draw,circle,inner sep=1pt]{22};\urar &&33 \urar&&44 \urar&&\node[draw,circle,inner sep=1pt]{11}; \urar&&\node[draw,circle,inner sep=1pt]{22};\arrow[dotted,-]{rr}&&\phantom{X}\\
\end{tikzcd}
\end{center}

\begin{center}
$B_5$ \hspace{0.5cm} \begin{tikzcd}[column sep=tiny, row sep=tiny]
&&\node[draw,circle,inner sep=1pt]{33}; \drar & &44\drar&&11 \drar & &\node[draw,circle,inner sep=1pt]{22};\drar &&\node[draw,circle,inner sep=1pt]{33}; \drar & &44\arrow[dotted,-]{rr}&&\phantom{X}\\
\arrow[dotted,-]{r}&32\urar\drar & &43\urar\drar&&\node[draw,circle,inner sep=1pt]{14};\urar\drar& &21\urar\drar& &32\urar\drar& &43\urar\drar\\
&&42\urar\drar&&13\urar\drar&&24\urar\drar&&31\urar\drar&&42\urar\drar&&13\arrow[dotted,-]{rr}&& \phantom{X}\\ 
\arrow[dotted,-]{r}&\node[draw,circle,inner sep=1pt]{41};\urar\drar & &12\urar\drar &&23\urar\drar& &34\urar\drar& &\node[draw,circle,inner sep=1pt]{41};\urar\drar& &12\urar\drar\\
&&11 \urar&&\node[draw,circle,inner sep=1pt]{22};\urar &&\node[draw,circle,inner sep=1pt]{33}; \urar&&44 \urar&&11 \urar&&\node[draw,circle,inner sep=1pt]{22};\arrow[dotted,-]{rr}&&\phantom{X}\\
\end{tikzcd}
\end{center}

\begin{center}
$B_6$ \hspace{0.5cm} \begin{tikzcd}[column sep=tiny, row sep=tiny]
&&33 \drar & &\node[draw,circle,inner sep=1pt]{44};\drar&&11 \drar & &\node[draw,circle,inner sep=1pt]{22};\drar &&33 \drar & &\node[draw,circle,inner sep=1pt]{44};\arrow[dotted,-]{rr}&&\phantom{X}\\
\arrow[dotted,-]{r}&32\urar\drar & &43\urar\drar&&14\urar\drar& &21\urar\drar& &32\urar\drar& &43\urar\drar\\
&&42\urar\drar&&\node[draw,circle,inner sep=1pt]{13};\urar\drar&&24\urar\drar&&\node[draw,circle,inner sep=1pt]{31};\urar\drar&&42\urar\drar&&\node[draw,circle,inner sep=1pt]{13};\arrow[dotted,-]{rr}&& \phantom{X}\\ 
\arrow[dotted,-]{r}&41\urar\drar & &12\urar\drar &&23\urar\drar& &34\urar\drar& &41\urar\drar& &12\urar\drar\\
&&11 \urar&&\node[draw,circle,inner sep=1pt]{22};\urar &&33 \urar&&\node[draw,circle,inner sep=1pt]{44}; \urar&&11 \urar&&\node[draw,circle,inner sep=1pt]{22};\arrow[dotted,-]{rr}&&\phantom{X}\\
\end{tikzcd}
\end{center}

\begin{center}
$B_7$  \hspace{0.5cm} \begin{tikzcd}[column sep=tiny, row sep=tiny]
&&\node[draw,circle,inner sep=1pt]{33}; \drar & &44\drar&&\node[draw,circle,inner sep=1pt]{11}; \drar & &22\drar &&\node[draw,circle,inner sep=1pt]{33}; \drar & &44\arrow[dotted,-]{rr}&&\phantom{X}\\
\arrow[dotted,-]{r}&32\urar\drar & &43\urar\drar&&14\urar\drar& &21\urar\drar& &32\urar\drar& &43\urar\drar\\
&&\node[draw,circle,inner sep=1pt]{42};\urar\drar&&13\urar\drar&&\node[draw,circle,inner sep=1pt]{24};\urar\drar&&31\urar\drar&&\node[draw,circle,inner sep=1pt]{42};\urar\drar&&13\arrow[dotted,-]{rr}&& \phantom{X}\\ 
\arrow[dotted,-]{r}&41\urar\drar & &12\urar\drar &&23\urar\drar& &34\urar\drar& &41\urar\drar& &12\urar\drar\\
&&\node[draw,circle,inner sep=1pt]{11}; \urar&&22\urar &&\node[draw,circle,inner sep=1pt]{33}; \urar&&44 \urar&&\node[draw,circle,inner sep=1pt]{11}; \urar&&22\arrow[dotted,-]{rr}&&\phantom{X}\\
\end{tikzcd}
\end{center}

\begin{center}
$B_8$ \hspace{0.5cm} \begin{tikzcd}[column sep=tiny, row sep=tiny]
&&33 \drar & &44\drar&&11 \drar & &22\drar &&33 \drar & &44\arrow[dotted,-]{rr}&&\phantom{X}\\
\arrow[dotted,-]{r}&32\urar\drar & &\node[draw,circle,inner sep=1pt]{43};\urar\drar&&14\urar\drar& &\node[draw,circle,inner sep=1pt]{21};\urar\drar& &32\urar\drar& &\node[draw,circle,inner sep=1pt]{43};\urar\drar\\
&&42\urar\drar&&13\urar\drar&&24\urar\drar&&31\urar\drar&&42\urar\drar&&13\arrow[dotted,-]{rr}&& \phantom{X}\\ 
\arrow[dotted,-]{r}&41\urar\drar & &\node[draw,circle,inner sep=1pt]{12};\urar\drar &&23\urar\drar& &\node[draw,circle,inner sep=1pt]{34};\urar\drar& &41\urar\drar& &\node[draw,circle,inner sep=1pt]{12};\urar\drar\\
&&11 \urar&&22\urar &&33 \urar&&44 \urar&&11 \urar&&22\arrow[dotted,-]{rr}&&\phantom{X}\\
\end{tikzcd}
\end{center}

\begin{center}
$B_9$  \hspace{0.5cm} \begin{tikzcd}[column sep=tiny, row sep=tiny]
&&33 \drar & &44\drar&&11 \drar & &22\drar &&33 \drar & &44\arrow[dotted,-]{rr}&&\phantom{X}\\
\arrow[dotted,-]{r}&\node[draw,circle,inner sep=1pt]{32};\urar\drar & &43\urar\drar&&\node[draw,circle,inner sep=1pt]{14};\urar\drar& &21\urar\drar& &\node[draw,circle,inner sep=1pt]{32};\urar\drar& &43\urar\drar\\
&&42\urar\drar&&13\urar\drar&&24\urar\drar&&31\urar\drar&&42\urar\drar&&13\arrow[dotted,-]{rr}&& \phantom{X}\\ 
\arrow[dotted,-]{r}&\node[draw,circle,inner sep=1pt]{41};\urar\drar & &12\urar\drar &&\node[draw,circle,inner sep=1pt]{23};\urar\drar& &34\urar\drar& &\node[draw,circle,inner sep=1pt]{41};\urar\drar& &12\urar\drar\\
&&11 \urar&&22\urar &&33 \urar&&44 \urar&&11 \urar&&22\arrow[dotted,-]{rr}&&\phantom{X}\\
\end{tikzcd}
\end{center}
\end{example}
 
 \begin{proof}[Proof of Theorem \ref{G-thm:correspondenceA}]
 First, we show that the map $\Phi: \mathbf{C}(A_{n+1})  \to \mathbf{B}_n^2$ is well-defined, i.e. $\Phi(C) \in \mathbf{B}_n^2$ holds for any $C \in \mathbf{C}(A_{n+1})$.
  
   Assume that there is an equivalence class of $\Phi( C)$ which contains at least three elements.
Then there exist $[i \ j] \in C$ and $[i \ k] \in C$ with pairwise distinct $i,j,k$. 
Thanks to Definition \ref{G-def:configuration} (C2), $$[i\ k] \in (H([i\ j]) \cup H([j\ i])) \cap C = \{[i \ j], [j \ i]\}$$ which is a contradiction. Hence, each equivalence class of $\Phi(C)$ only contains at most two elements.

Now, we prove that the convex hulls of distinct equivalence classes of $\Phi(C)$ are disjoint. 
Assume that the convex hull of $i \sim j$ is joint with the convex hull of $r \sim s$.
Without losing generality, we assume that $1\le i <r<j<s \le n$.
This means $[r\ s]\in H([j\ i])$ (see Figure \ref{G-serrebox}), which contradicts with (C2) again.
Hence, the convex hull of distinct equivalence classes of $\Phi(C)$ are disjoint.
Therefore, the equivalence relation $\Phi(C)$ is a $2$-Brauer relation of rank $n$.

We construct a map $\Psi: \mathbf{B}_n^2 \to  \mathbf{C}(A_{n+1}) $ by $B \mapsto \Psi(B) :=\{ [1 \ \sigma_B(1)], [2 \ \sigma_B(2)] ,\ldots, [n \ \sigma_B(n)]\}$.
Let us prove that the map is well-defined, i.e. the set $\Psi(B)$ is a configuration of $\ZZ A_{n+1}$ for each $B\in\mathbf{B}_n^2$. 
For any $[i\ j] \in \Psi(B),$ we have $[j\ i] \in \Psi(B)$ since $\sigma_B^2=\Id$. 
Assume that $[r \ s] \in \Psi(B)$ belongs to $H([j\ i])$. 
Without losing generality, assume that $1=i \le r \le j \le s \le n$.
This means the convex hull of $i \sim j$ intersects with the convex hull of $r \sim s$, we have $\{r,s\}=\{i,j\}$. Hence, we have $ H([j\ i]) \cap \Psi(B) =\{[i\ j], [j\ i]\}$, so (C2) holds.
For any $[r\ s]\in (\ZZ A_{n+1})_0$, we have $[r\ s]\in H([r\ \sigma_B(r)]) \cup H([\sigma_B(r)\ r]),$ so (C1) holds. We get $\Psi(B) \in \mathbf{C}(A_{n+1}) $.

 For any $C \in \mathbf{C}(A_{n+1})$, it is obvious that $\Phi  \Psi (C) \supset C$. Let us prove that $\Phi  \Psi (C) \subset C$.
Assume that $[i \ j] \in \Phi  \Psi (C)$. Thanks to (C1), there exists $[i \ k] \in C$ such that $[i\ i] \in H([i\ k])$. So we have $[i \ j], [i \ k] \in \Phi  \Psi (C)$, where $ \Phi  \Psi (C)$ is also a configuration. As before it implies that $k=j$. Hence, we have $\Phi  \Psi =\Id$. Clearly, $\Psi \Phi =\Id$ holds.
\end{proof}
 
 \section{Types $B$ and $C$}
 \label{G-s:typeBC}
 In this section, we will describe all the configurations of type $B$.  In order to realize the description easily, we need to introduce symmetric $2$-Brauer relations which is a special class of $2$-Brauer relations given in Definition \ref{G-def:Brauer} for type $A$.
 
 \begin{definition}
\label{G-def:symmetricBrauer}
A \emph{symmetric $2$-Brauer relation $B$ of rank $2n$} is a $2$-Brauer relation of rank $2n$ which is symmetric with respect to rotation by $\pi$.
We denote by $\mathbf{B}_{2n}^{2,s}$ the set of all symmetric $2$-Brauer relations of rank $2n$.
\end{definition}

Let us see the example of symmetric $2$-Brauer relations of rank $4$ as follows.
\begin{example}
\label{G-ex:brauerD4}
All the symmetric $2$-Brauer relations of rank $4$ are shown as follows, where $$\mathbf{B}_4^{2,s}=\{ B_1^s, B_2^s,B_3^s,B_4^s,B_5^s\}.$$
      
      \begin{center}
      \begin{tabular}{ccccccc} 
      \begin{xy}
    		0;<1cm,0cm>:<0cm,-1cm>::,="D",{\xypolygon50{~>.}},
    		{\xypolygon4"A"{~*{\xypolynode}~>{}}}, 
    \end{xy} &\hspace{0.1cm} &
    
      \begin{xy}
    		0;<1cm,0cm>:<0cm,-1cm>::,="D",{\xypolygon50{~>.}},
    		{\xypolygon4"A"{~*{\xypolynode}~>{}}}, 
    		"A3";"A1"**\dir{-},
    \end{xy}  &   
    
      \begin{xy}
    		0;<1cm,0cm>:<0cm,-1cm>::,="D",{\xypolygon50{~>.}},
    		{\xypolygon4"A"{~*{\xypolynode}~>{}}}, 
    		"A2";"A4"**\dir{-},
    \end{xy}  & \hspace{0.1cm}&
    
     \begin{xy}
    		0;<1cm,0cm>:<0cm,-1cm>::,="D",{\xypolygon50{~>.}},
    		{\xypolygon4"A"{~*{\xypolynode}~>{}}}, 
    		"A1";"A2"**\dir{-},
                "A3";"A4"**\dir{-},
    \end{xy} & 
    
      \begin{xy}
    		0;<1cm,0cm>:<0cm,-1cm>::,="D",{\xypolygon50{~>.}},
    		{\xypolygon4"A"{~*{\xypolynode}~>{}}}, 
    		"A2";"A3"**\dir{-},
                "A1";"A4"**\dir{-},
    \end{xy} 
         \\$B_1^s$
         & &$ B_2^s$
          &$B_3^s$
          & &$B_4^s$
          &$B_5^s$
       \end{tabular}
       \end{center}
\end{example}

\begin{proposition}
\label{pro:cardinalBC}
The cardinal number $M^s(n)$ of $\mathbf{B}_{2n}^{2,s}$ is given by the following recursive formula 
$$M^s(n)=M^s(n-1)+M(n-1)+ 2\sum_{i=0}^{n-2}M(i) M^s(n-2-i)=\sum_{k=0}^{\lfloor(n+1)/2\rfloor} \frac{n!(n+1-k)!}{k!(n-k)!k!(n+1-2k)!}.$$ 
\end{proposition}
\begin{proof}
For any symmetric $2$-Brauer relation $B$ of rank $2n$, if $\sigma_B(1)=1$, then $\sigma_B(n+1)=n+1$. The number of such kind of symmetric $2$-Brauer relations is $M^s(n-1)$.
If $1 \sim n+1$, then the number of these symmetric $2$-Brauer relations is $M(n-1)$.
If $1 \sim i+2$ ($0 \le i \le n-2$), then we divide the corresponding disk with $2n$ marked points into three parts by drawing diagonals $1 \sim i+2$ and $n+1 \sim n+i+2$. The two symmetric parts have $i$ marked points and the third part has $2(n-2-i)$ marked points. Hence, the number of this kind of symmetric $2$-Brauer relations is $M(i)M^s(n-2-i)$. It is similar when $1 \sim n+2+i$ ($0 \le i \le n-2$) holds. We get the recursive formula.

By induction, we can get the other formula. Note that $M^s(n)$ is known as number of directed animals. 
\end{proof}

 In the rest, we consider the labels of the vertices of $\ZZ B_{n+1}$ in $\ZZ/2n\ZZ \times \ZZ/2n\ZZ$ in Figure \ref{G-fig:labelB}.   
 \begin{figure}
 \begin{center}
\tiny{
\begin{tikzcd}[column sep=small]
&&2n\!-\!1\ n\!-\!1 \arrow{dr}[sloped, above]{(1,2)} & &2n\ n\arrow{dr}[sloped, above]{(1,2)}&&1\  n\!+\!1\arrow{dr}[sloped, above]{(1,2)} & &2\ n\!+\!2\arrow[dotted,-]{rr}&&2n\!-\!1\ n\!-\!1\arrow[dotted,-]{rr}&&\phantom{X}\\
\arrow[dotted,-]{r}&2n\!-\!1\ *\urar[dotted,-]\drar[dotted,-] & &2n\ n\!-\!1\arrow{ur}[sloped, above]{(2,1)}\drar&&1\ n\arrow{ur}[sloped, above]{(2,1)}\drar& &2\ n\!+\!1\arrow{ur}[sloped, above]{(2,1)}\arrow[dotted,-]{rr}&&2n\!-\!1\ 2n\!-\!2\arrow{ur}[sloped, above]{(2,1)}\drar[dotted,-]\\
&&2n\ 3\urar[dotted,-]\drar&&1\ n\!-\!1\urar\drar&&2\ n\urar\arrow[dotted,-]{rr}&&2n\!-\!1\ *\urar[dotted,-]\drar[dotted,-]&&2n\ *\arrow[dotted,-]{rr}&&\phantom{X}\\ 
\arrow[dotted,-]{r}&2n\ 2\urar\drar & &1\ 3\urar[dotted,-]\drar &&2\ n\!-\!1\urar\arrow[dotted,-]{rr}&&2n\!-\!1\ 1\urar[dotted,-]\drar& &2n\ *\urar[dotted,-]\drar[dotted,-]\\
&&1\ 2\urar\drar&&2\ 3\urar[dotted,-]\arrow[dotted,-]{rr}&&2n\!-\!1\ 2n\urar\drar&&2n\ 1\urar[dotted,-]\drar&&1\ *\arrow[dotted,-]{rr}&&\phantom{X}\\ 
\arrow[dotted,-]{r}&1\ 1\urar & &2\ 2\urar\arrow[dotted,-]{rr}&&2n\!-\!1\ 2n\!-\!1\urar& &2n\ 2n\urar& &1\ 1\urar[dotted,-]\\
\end{tikzcd}
}
\end{center}
\caption{Labels of vertices of $\ZZ B_{n+1}$}
\label{G-fig:labelB}
\end{figure}
Remark that we only use labels $[r \ s]$ such that $0 \le s-r (\mod 2n) \le n $.
By calculation, we have 
\begin{equation}
\label{G-eq:nuB}
\omega([j+n\ \ i+n])=[j\ i].
\end{equation}
Moreover, 
 \begin{align}
 \label{G-eq:typeB}
 H([j+n\ \ i+n])=\, &  \{[s\ t] \mid \forall s \in \{j, j+1, \ldots, i \} , \forall t \in \{i, i+1, \ldots, i+n\ \}, t-s(\mod 2n) \le n \}  \cup \notag \\
   &  \{[s\ t] \mid \forall s \in \{i+1, i+2,\ldots, j+n \} , \forall t \in \{j+n, j+n+1, \ldots, i+n\} \} 
 \end{align}
 is the set consisting of all the vertices in Figure \ref{G-HboxB}.
   \begin{figure}[!h]
\scalebox{1} 
{
\begin{pspicture}(0,-2.0080469)(13.042911,2.0080469)
\psline[linewidth=0.024cm](5.821016,0.14964844)(7.5010157,-1.5303515)
\psline[linewidth=0.024cm](1.2610157,-1.5303515)(11.581016,-1.5303515)
\psline[linewidth=0.024cm](7.2610154,1.5896485)(8.941015,-0.09035156)
\psline[linewidth=0.024cm](1.2610157,1.5896485)(11.581016,1.5896485)
\psline[linewidth=0.024cm](2.7010157,-0.09035156)(4.381016,1.5896485)
\psline[linewidth=0.024cm](4.1410155,-1.5303515)(5.821016,0.14964844)
\psline[linewidth=0.024cm](7.5010157,-1.5303515)(8.941015,-0.09035156)
\psline[linewidth=0.024cm](4.1410155,-1.5303515)(2.7010157,-0.09035156)
\usefont{T1}{ptm}{m}{n}
\rput(4.0824708,-1.8553515){$i\  i$}
\usefont{T1}{ptm}{m}{n}
\rput(7.5424706,-1.8553515){$j\!+\!n\ \  j\!+\!n$}
\usefont{T1}{ptm}{m}{n}
\rput(2.3824708,0.10535156){$j\ i$}
\usefont{T1}{ptm}{m}{n}
\rput(10.002471,0.10535156){$j\!+\!n\ \  i\!+\!n$}
\usefont{T1}{ptm}{m}{n}
\rput(7.342471,1.8546484){$i\ \ i\!+\!n$}
\usefont{T1}{ptm}{m}{n}
\rput(4.5824705,1.8546484){$j\ \  j\!+\!n$}
\usefont{T1}{ptm}{m}{n}
\rput(5.8824705,0.38146484){$i\ \  j\!+\!n$}
\psline[linewidth=0.034cm,linestyle=dotted,dotsep=0.16cm](4.381016,1.5896485)(5.821016,0.14964844)
\psline[linewidth=0.034cm,linestyle=dotted,dotsep=0.16cm](7.2610154,1.5896485)(5.821016,0.14964844)
\end{pspicture} 
}
   \caption{}
   \label{G-HboxB}
 \end{figure}
 \newline

By using the above analysis, we can get the following theorem.
\begin{theorem}
 \label{G-thm:correspondenceB}
We denote by $\mathbf{C}(B_{n+1})$ the set of all the configurations of $\ZZ B_{n+1}$.
There is a bijection  
 \begin{align*} 
 \Phi: \mathbf{C}(B_{n+1})  \to  \mathbf{B}_{2n}^{2,s},  
  \end{align*}
    where $\Phi(C)$ for $C \in \mathbf{C}(B_{n+1}) $ is the equivalence relation on the set $\{1,2, \ldots, 2n\}$ generated by $i \sim j$ for each $[j\ i] \in C$, and $\Phi^{-1}(B)$ for $B\in \mathbf{B}_{2n}^{2,s}$ is $\{[1 \ \sigma_{B}(1)],\ldots, [2n \ \sigma_{B}(2n)]\}\cap (\ZZ B_{n+1})_0$.
\end{theorem}

First, let us see the example of configurations of $\ZZ B_3$ and symmetric $2$-Brauer relations of rank $4$.
\begin{example}
 The symmetric $2$-Brauer relations of rank $4$ are listed in Example \ref{G-ex:brauerD4}.
 Then the one-to-one correspondence $\mathbf{C}(B_3) \longleftrightarrow \mathbf{B}_4^{2,s}$ is given by:
 \begin{center}
$B_1^s$  \begin{tikzcd}[column sep=small, row sep=tiny]
\arrow[dotted,-]{r}&42\arrow{dr}[sloped, above]{(1,2)} & &13\arrow{dr}[sloped, above]{(1,2)}&&24\arrow{dr}[sloped, above]{(1,2)}& &31\arrow{dr}[sloped, above]{(1,2)}& &42\arrow{dr}[sloped, above]{(1,2)}& &13\arrow{dr}[sloped, above]{(1,2)}&&24\arrow[dotted,-]{rr}&& \phantom{X}\\
&&12\arrow{ur}[sloped, above]{(2,1)}\drar&&23\arrow{ur}[sloped, above]{(2,1)}\drar&&34\arrow{ur}[sloped, above]{(2,1)}\drar&&41\arrow{ur}[sloped, above]{(2,1)}\drar&&12\arrow{ur}[sloped, above]{(2,1)}\drar&&23\arrow{ur}[sloped, above]{(2,1)}\drar\\
\arrow[dotted,-]{r}&\node[draw,circle,inner sep=1pt]{11};\urar& &\node[draw,circle,inner sep=1pt]{22};\urar &&\node[draw,circle,inner sep=1pt]{33};\urar& &\node[draw,circle,inner sep=1pt]{44};\urar& &\node[draw,circle,inner sep=1pt]{11};\urar& &\node[draw,circle,inner sep=1pt]{22};\urar&&\node[draw,circle,inner sep=1pt]{33};\arrow[dotted,-]{rr}&& \phantom{X}
\end{tikzcd}
\end{center}

 \begin{center}
$B_2^s$ \begin{tikzcd}[column sep=small, row sep=tiny]
\arrow[dotted,-]{r}&42\arrow{dr}[sloped, above]{(1,2)} & &\node[draw,circle,inner sep=1pt]{13};\arrow{dr}[sloped, above]{(1,2)}&&24\arrow{dr}[sloped, above]{(1,2)}& &\node[draw,circle,inner sep=1pt]{31};\arrow{dr}[sloped, above]{(1,2)}& &42\arrow{dr}[sloped, above]{(1,2)}& &\node[draw,circle,inner sep=1pt]{13};\arrow{dr}[sloped, above]{(1,2)}&&24\arrow[dotted,-]{rr}&& \phantom{X}\\
&&12\arrow{ur}[sloped, above]{(2,1)}\drar&&23\arrow{ur}[sloped, above]{(2,1)}\drar&&34\arrow{ur}[sloped, above]{(2,1)}\drar&&41\arrow{ur}[sloped, above]{(2,1)}\drar&&12\arrow{ur}[sloped, above]{(2,1)}\drar&&23\arrow{ur}[sloped, above]{(2,1)}\drar\\ 
\arrow[dotted,-]{r}&11\urar& &\node[draw,circle,inner sep=1pt]{22};\urar &&33\urar& &\node[draw,circle,inner sep=1pt]{44};\urar& &11\urar& &\node[draw,circle,inner sep=1pt]{22};\urar&&33\arrow[dotted,-]{rr}&& \phantom{X}
\end{tikzcd}
\end{center}

 \begin{center}
$B_3^s$ \begin{tikzcd}[column sep=small, row sep=tiny]
\arrow[dotted,-]{r}&\node[draw,circle,inner sep=1pt]{42};\arrow{dr}[sloped, above]{(1,2)} & &13\arrow{dr}[sloped, above]{(1,2)}&&\node[draw,circle,inner sep=1pt]{24};\arrow{dr}[sloped, above]{(1,2)}& &31\arrow{dr}[sloped, above]{(1,2)}& &\node[draw,circle,inner sep=1pt]{42};\arrow{dr}[sloped, above]{(1,2)}& &13\arrow{dr}[sloped, above]{(1,2)}&&\node[draw,circle,inner sep=1pt]{24};\arrow[dotted,-]{rr}&& \phantom{X}\\
&&12\arrow{ur}[sloped, above]{(2,1)}\drar&&23\arrow{ur}[sloped, above]{(2,1)}\drar&&34\arrow{ur}[sloped, above]{(2,1)}\drar&&41\arrow{ur}[sloped, above]{(2,1)}\drar&&12\arrow{ur}[sloped, above]{(2,1)}\drar&&23\arrow{ur}[sloped, above]{(2,1)}\drar\\
\arrow[dotted,-]{r}&\node[draw,circle,inner sep=1pt]{11};\urar& &22\urar &&\node[draw,circle,inner sep=1pt]{33};\urar& &44\urar& &\node[draw,circle,inner sep=1pt]{11};\urar& &22\urar&&\node[draw,circle,inner sep=1pt]{33};\arrow[dotted,-]{rr}&& \phantom{X}
\end{tikzcd}
\end{center}

 \begin{center}
$B_4^s$ \begin{tikzcd}[column sep=small, row sep=tiny]
\arrow[dotted,-]{r}&42\arrow{dr}[sloped, above]{(1,2)} & &13\arrow{dr}[sloped, above]{(1,2)}&&24\arrow{dr}[sloped, above]{(1,2)}& &31\arrow{dr}[sloped, above]{(1,2)}& &42\arrow{dr}[sloped, above]{(1,2)}& &13\arrow{dr}[sloped, above]{(1,2)}&&24\arrow[dotted,-]{rr}&& \phantom{X}\\
&&\node[draw,circle,inner sep=1pt]{12};\arrow{ur}[sloped, above]{(2,1)}\drar&&23\arrow{ur}[sloped, above]{(2,1)}\drar&&\node[draw,circle,inner sep=1pt]{34};\arrow{ur}[sloped, above]{(2,1)}\drar&&41\arrow{ur}[sloped, above]{(2,1)}\drar&&\node[draw,circle,inner sep=1pt]{12};\arrow{ur}[sloped, above]{(2,1)}\drar&&23\arrow{ur}[sloped, above]{(2,1)}\drar\\
\arrow[dotted,-]{r}&11\urar& &22\urar &&33\urar& &44\urar& &11\urar& &22\urar&&33\arrow[dotted,-]{rr}&& \phantom{X}
\end{tikzcd}
\end{center}

 \begin{center}
$B_5^s$ \begin{tikzcd}[column sep=small, row sep=tiny]
\arrow[dotted,-]{r}&42\arrow{dr}[sloped, above]{(1,2)} & &13\arrow{dr}[sloped, above]{(1,2)}&&24\arrow{dr}[sloped, above]{(1,2)}& &31\arrow{dr}[sloped, above]{(1,2)}& &42\arrow{dr}[sloped, above]{(1,2)}& &13\arrow{dr}[sloped, above]{(1,2)}&&24\arrow[dotted,-]{rr}&& \phantom{X}\\
&&12\arrow{ur}[sloped, above]{(2,1)}\drar&&\node[draw,circle,inner sep=1pt]{23};\arrow{ur}[sloped, above]{(2,1)}\drar&&34\arrow{ur}[sloped, above]{(2,1)}\drar&&\node[draw,circle,inner sep=1pt]{41};\arrow{ur}[sloped, above]{(2,1)}\drar&&12\arrow{ur}[sloped, above]{(2,1)}\drar&&\node[draw,circle,inner sep=1pt]{23};\arrow{ur}[sloped, above]{(2,1)}\drar\\
\arrow[dotted,-]{r}&11\urar& &22\urar &&33\urar& &44\urar& &11\urar& &22\urar&&33\arrow[dotted,-]{rr}&& \phantom{X}
\end{tikzcd}
\end{center}
 \end{example}
 
 \begin{proof}[Proof of Theorem \ref{G-thm:correspondenceB}]
 First of all, we will prove that the map $\Phi: \mathbf{C}(B_{n+1}) \to \mathbf{B}_{2n}^{2,s}$ is well-defined, i.e. $\Phi(C)$ is a symmetric $2$-Brauer relation for each $C \in \mathbf{C}(B_{n+1})$.
 For each $[j\ i]\in  C$, according to \eqref{G-eq:nuB}, we know that $\omega ([j\ i])= [j-n\ \ i-n] =[j+n\ \ i+n] \in C.$
So the equivalence relation $\Phi (C)$ is symmetric with respect to the rotation by $\pi$.
   
Assume that there is an equivalence class of $\Phi(C)$ which contains at least three elements.
Then there exist $[j\ i]\in  C$ and $[t\ i] \in C$ (or $[j \ t] \in C$) with pairwise distinct $i,j,t$.
Thanks to (C2) and \eqref{G-eq:typeB}, $$[t\ i] \in (H([j\ i]) \cup H([j+n\ \ i+n])) \cap C=\{[j \ i],[j+n \ \ i+n]\}$$
which is a contradiction.
 Hence, each equivalence class in $\Phi(C)$ contains at most two elements.
 
 Assume that that the convex hull of $j \sim i$ is joint with the convex hull of $r \sim s$. Then by \eqref{G-eq:typeB}, we get $[r \ s] \in (H([j \ i]) \cup H([j+n \ \ i+n]))\cap C$, which contradicts (C2). Hence, the convex hulls of distinct equivalence classes are disjoint. 
Therefore, $\Phi(C)\in \mathbf{B}_{2n}^{2,s}.$

We construct a map $\Psi:  \mathbf{B}_{2n}^{2,s} \to \mathbf{C}(B_{n+1})  \text{ by }  \Psi(B):=\{[1 \ \sigma_{B}(1)],\ldots, [2n \ \sigma_{B}(2n)]\}\cap (\ZZ B_{n+1})_0 .$
In fact, for any $s,t \in \{1,2, \ldots, 2n\}$ with $s \neq t\pm n$, there is only one of $[s\ t]$ and $ [t\ s]$ representing a vertex in $(\ZZ B_{n+1})_0$ for our labels. And when $s= t \pm n$, both $[s\ t]$ and $[t\ s]$ are in $(\ZZ B_{n+1})_0$.
Since $B$ is symmetric with respect to the rotation by $\pi$, for each $[j\ i]\in C$, it follows that $\omega([j\ i])=[j+n\ \ i+n] \in \Psi(B).$

Assume that $[r \ s] \in \Psi(B)$ belongs to $H([j\ i]) \cup H([j+n \ \ i+n])$. Then by \eqref{G-eq:typeB}, it means that the convex hull of $r \sim s$ intersects with the convex hull of $i \sim j$ or the convex hull of $i+n \sim j+n$, we have $\{r,s\}=\{i,j\}$ or $\{r,s\}=\{i+n, j+n\}$.
Hence, by \eqref{G-eq:typeB} and Figure \ref{G-HboxB}, we have $$H([j\ i]) \cap \Psi(B) =\{[j\ i], [j+n\ \ i+n]\}\and H([j+n\ \ i+n]) \cap \Psi(B) =\{[j\ i], [j+n\ \ i+n]\}.$$ Thus, (C2) holds.
For any $[r\ s] \in (\ZZ B_{n+1})_0$, we have $[r\ s] \in H([r\ \sigma_B(r)]) \cup H([r+n\ \ \sigma_B(r)+n]),$ so (C1) holds. 
Therefore, $\Psi(B) \in \mathbf{C}(B_{n+1})$.

For any $C \in \mathbf{C}(B_{n+1})$, it is obvious that $\Phi\Psi(C) \supset C$. Let us prove that $\Phi\Psi(C) \subset C$. Assume that $[j \ i] \in \Phi\Psi(C)$. Thanks to (C1), there exists $[k \ i]\in C$ such that $[i\ i] \in H([k+n\ \ i+n])$. So we have $[j\ i], [k \ i] \in \Phi\Psi(C)$, where $\Phi\Psi(C)$ is also a configuration. As before it implies $k=j$. Hence, we have $\Phi  \Psi=\Id$. Clearly, $\Psi  \Phi=\Id$ holds.
\end{proof}


Now we will describe all the configurations of type $C$.
We consider the labels of the vertices of $\ZZ C_{n+1}$ in $\ZZ/2n\ZZ \times \ZZ/2n\ZZ$ in Figure \ref{G-fig:labelC}. 
\begin{figure}
 \begin{center}
\tiny{
\begin{tikzcd}[column sep=small]
&&2n\!-\!1\ n\!-\!1 \arrow{dr}[sloped, above]{(2,1)} & &2n\ n\arrow{dr}[sloped, above]{(2,1)}&&1\  n\!+\!1\arrow{dr}[sloped, above]{(2,1)} & &2\ n\!+\!2\arrow[dotted,-]{rr}&&2n\!-\!1\ n\!-\!1\arrow[dotted,-]{rr}&&\phantom{X}\\
\arrow[dotted,-]{r}&2n\!-\!1\ *\urar[dotted,-]\drar[dotted,-] & &2n\ n\!-\!1\arrow{ur}[sloped, above]{(1,2)}\drar&&1\ n\arrow{ur}[sloped, above]{(1,2)}\drar& &2\ n\!+\!1\arrow{ur}[sloped, above]{(1,2)}\arrow[dotted,-]{rr}&&2n\!-\!1\ 2n\!-\!2\arrow{ur}[sloped, above]{(1,2)}\drar[dotted,-]\\
&&2n\ 3\urar[dotted,-]\drar&&1\ n\!-\!1\urar\drar&&2\ n\urar\arrow[dotted,-]{rr}&&2n\!-\!1\ *\urar[dotted,-]\drar[dotted,-]&&2n\ *\arrow[dotted,-]{rr}&&\phantom{X}\\ 
\arrow[dotted,-]{r}&2n\ 2\urar\drar & &1\ 3\urar[dotted,-]\drar &&2\ n\!-\!1\urar\arrow[dotted,-]{rr}&&2n\!-\!1\ 1\urar[dotted,-]\drar& &2n\ *\urar[dotted,-]\drar[dotted,-]\\
&&1\ 2\urar\drar&&2\ 3\urar[dotted,-]\arrow[dotted,-]{rr}&&2n\!-\!1\ 2n\urar\drar&&2n\ 1\urar[dotted,-]\drar&&1\ *\arrow[dotted,-]{rr}&&\phantom{X}\\ 
\arrow[dotted,-]{r}&1\ 1\urar & &2\ 2\urar\arrow[dotted,-]{rr}&&2n\!-\!1\ 2n\!-\!1\urar& &2n\ 2n\urar& &1\ 1\urar[dotted,-]\\
\end{tikzcd}
}
\end{center}
\caption{Labels of vertices of $\ZZ C_{n+1}$}
\label{G-fig:labelC}
\end{figure}

By similar analysis with type $B$, we have the following theorem.
 \begin{theorem}
 \label{G-thm:correspondenceC}
We denote by $\mathbf{C}(C_{n+1})$ the set of all configurations of $\ZZ C_{n+1}$.
 There is a bijection  
 \begin{align*} 
 \Phi: \mathbf{C}(C_{n+1})  \to \mathbf{B}_{2n}^{2,s},
  \end{align*}
    where $\Phi(C)$ for $C \in \mathbf{C}(C_{n+1})$ is the equivalence relation on the set $\{1,2,\ldots,2n\}$ generated by $i \sim j$ for each $[j\ i] \in C$ and $\Phi^{-1}(B)$ for $B \in \mathbf{B}_{2n}^{2,s}$ is $\{[1 \ \sigma_{B}(1)],\ldots, [2n \ \sigma_{B}(2n)]\}\cap (\ZZ C_{n+1})_0$.
\end{theorem}
We omit the proof because it is identical with the one for Type $B$.

\section{Type $D$}
\label{G-s:typeD}
In this section, we describe configurations of type $D$ by using crossing $2$-Brauer relations. First, let us introduce the following notion.
\begin{definition}
\label{G-def:crossingBrauer}
A \emph{crossing $2$-Brauer relation $B$ of rank $2n$} is an equivalence relation which only allows at most two elements in any equivalence class on the set $ \{1,2, \ldots, 2n \} $ satisfying the following two conditions:
\begin{itemize}
\item[(1)] exactly two of the convex hulls of distinct equivalence classes are joint;
\item[(2)] the equivalence relation is symmetric with respect to rotation by $\pi$.
\end{itemize}
We define a permutation $\sigma_{B}$ as follows 
\begin{align*}
     \sigma_{B}(i):= \begin{cases}
        i, \text{ if $\{i\}$ is an equivalence class;}\\
        j, \text{ if $\{i, j \}$ is an equivalence class.}
                     \end{cases}
\end{align*}
Denote by $\mathbf{B}_{2n}^{2,c}$ the set of crossing $2$-Brauer relations of rank $2n$ and $i\sim j$ if $i$ is equivalent with $j$.
\end{definition}

Let us see an example of crossing $2$-Brauer relations.
\begin{example}
\label{G-ex:crossing}
The set $\mathbf{B}_4^{2,c}$ consists of only one crossing $2$-Brauer relation $B_1^c$ of rank $4$ as follows.
$$ \begin{xy}
    		0;<1cm,0cm>:<0cm,-1cm>::,="D",{\xypolygon50{~>.}},
    		{\xypolygon4"A"{~*{\xypolynode}~>{}}}, 
    		"A1";"A3"**\dir{-},
                "A2";"A4"**\dir{-},
    \end{xy}$$
    In this case, we have $\sigma_{B^c_1}(1)=3$, $\sigma_{B^c_1}(2)=4$, $\sigma_{B^c_1}(3)=1$ and $\sigma_{B^c_1}(4)=2$.
\end{example}

\begin{proposition}
\label{pro:cardinalD}
Tthe cardinal number $M^c(n)$ of $\mathbf{B}_{2n}^{2,c}$ is given by the following residue formula
$$M^c(n)= \sum_{i=1}^{n-1} \sum_{j=i+1}^{n}M(j-i-1)M(n+i-j-1)= \sum_{k=0}^{\lfloor (n-2)/2 \rfloor} \frac{n!}{k! (k+2)!(n-2-2k)!}.
$$
\end{proposition}
\begin{proof}
For any crossing $2$-Brauer relation $B$ of rank $2n$, if $\sigma_B(i)=n+i$ ($1 \le i \le n-1$) and $\sigma_B(j)=n+j$ ($i+1\le j \le n$), then we divide the corresponding disk with $2n$ marked points into four parts by drawing diagonals $i \sim n+i$ and $j \sim n+j$. One of the two symmetric parts has $j-i-1$ marked points and the other one has $n+i-j-1$ marked points. Hence, the number of this kind of crossing $2$-Brauer relation is $M(j-i-1)M(n+i-j-1)$. We get the recursive formula.

By induction, we get the other formula.
\end{proof}

We consider the labels of the vertices of $\ZZ D_{n+2}$ in Figure \ref{G-parametrizeD} in $\ZZ/2n\ZZ \times \ZZ/2n\ZZ$. 
For each $[r \ s]$, $0 \le s-r (\mod 2n) \le n$ holds. If $s-r=n(\mod 2n)$ holds, we add a sign $+$ or $-$ such that $\tau ([r\ s]_+) = [r-1 \ \ s-1]_-$ and $\tau ([r\ s]_-) = [r-1 \ \ s-1]_+$.
\begin{center}
 \begin{figure}[!h]
 \tiny{
    \input{Gorenstein/coordinateD}
   \caption{Labels of vertices of $\ZZ D_{n+2}$}
   \label{G-parametrizeD}
   }
 \end{figure}
 \end{center}
 We denote $$[i\ \ i+n]_{\pm}:=[i\ \ i+n]_++[i\ \ i+n]_- \in \NN_0(\ZZ D_{n+2})_0$$ for any $i$. 
 We give a description of the map $\omega$ and the set $H$ for the vertices in $\ZZ D_{n+2}$. We need to divide the description into two cases.
 
Case 1: Let $[j+n\ \ i+n] \in (\ZZ D_{n+2})_0$ with $i \neq j \pm n$. By calculation, we have 
\begin{equation}
\label{G-eq:case2}
\omega([j+n\ \ i+n])=[j\ i].
\end{equation}
\begin{align}
 \label{G-eq:case2h}
 H([j+n\ \ i+n])= \, & \{[s\ \ s+n]_\pm \mid \forall s \in \{j, j+1, \ldots, i \}  \}  \cup \notag \\
  &  \{[s\ t] \mid \forall s \in \{j, j+1, \ldots, i \} , \forall t \in \{i, i+1, \ldots, i+n-1\ \}, t\neq s+n \}  \cup \notag \\
   &  \{[s\ t] \mid \forall s \in \{i+1, i+2,\ldots, j+n \} , \forall t \in \{j+n, j+n-1, \ldots, i+n\} \} 
 \end{align}
 is the set consisting of all the vertices in Figure \ref{G-HQboxD}.
  \begin{figure}[!h]
\scalebox{1} 
{
\begin{pspicture}(0,-2.0080469)(13.042911,2.0080469)
\psline[linewidth=0.024cm](5.821016,0.14964844)(7.5010157,-1.5303515)
\psline[linewidth=0.024cm](1.2610157,-1.5303515)(11.581016,-1.5303515)
\psline[linewidth=0.024cm](7.2610154,1.5896485)(8.941015,-0.09035156)
\psline[linewidth=0.024cm](1.2610157,1.5896485)(11.581016,1.5896485)
\psline[linewidth=0.024cm](2.7010157,-0.09035156)(4.381016,1.5896485)
\psline[linewidth=0.024cm](4.1410155,-1.5303515)(5.821016,0.14964844)
\psline[linewidth=0.024cm](7.5010157,-1.5303515)(8.941015,-0.09035156)
\psline[linewidth=0.024cm](4.1410155,-1.5303515)(2.7010157,-0.09035156)
\usefont{T1}{ptm}{m}{n}
\rput(4.0824708,-1.8553515){$i\ i$}
\usefont{T1}{ptm}{m}{n}
\rput(7.7424706,-1.8553515){$j\!+\!n\ \  j\!+\!n$}
\usefont{T1}{ptm}{m}{n}
\rput(2.3824708,0.10535156){$j\ i$}
\usefont{T1}{ptm}{m}{n}
\rput(9.902471,0.10535156){$j\!+\!n\ \ i\!+\!n$}
\usefont{T1}{ptm}{m}{n}
\rput(7.542471,1.8546484){$i\ \ i\!+\!n_{\pm}$}
\usefont{T1}{ptm}{m}{n}
\rput(4.5824705,1.8546484){$j\ \  j\!+\!n_{\pm}$}
\usefont{T1}{ptm}{m}{n}
\rput(5.8824705,0.38146484){$i\ \  j\!+\!n$}
\psline[linewidth=0.034cm,linestyle=dotted,dotsep=0.16cm](4.381016,1.5896485)(5.821016,0.14964844)
\psline[linewidth=0.034cm,linestyle=dotted,dotsep=0.16cm](7.2610154,1.5896485)(5.821016,0.14964844)
\end{pspicture} 
}
   \caption{}
   \label{G-HQboxD}
 \end{figure}

Case 2: Let $[i\ \ i+n]_\epsilon \in (\ZZ D_{n+2})_0$ for $\epsilon \in \{+, -\}$. By Definition \ref{G-def:HQbox} and direct calculation, we have 
\begin{equation}
\label{G-eq:case3}
  \omega([i\ \ i+n]_\epsilon)=[i-n\ \ i]_\epsilon = [i+n \ \ i]_\epsilon
\end{equation}
  and 
 \begin{align}
 \label{G-eq:case3h}
 H([i\ \ i+n]_\epsilon)=\, & \{ [i-k\ \ i+n-k]_\epsilon \mid \forall k \in \{0, 1, \ldots, n\} \}  \cup \notag\\
   & \{ [i-n+s\ \ i+t] \mid \forall s \in \{1, \ldots, n\}, \forall t \in \{0, 1, \ldots, n-1\} \}.
  \end{align} 

We divide the configurations into two disjoint sets $\mathbf{C}^1(D_{n+2})$ and $\mathbf{C}^2(D_{n+2})$.
  We denote by $\mathbf{C}^1(D_{n+2})$ the set of all the configurations of $\ZZ D_{n+2}$ which either contain one subset of vertices of the form $\{[i\ \ i+n]_+, [i\ \ i+n]_-, [i+n\ \ i]_+, [i+n\ \ i]_-\} $ or contain no vertices of the form $[i\ \ i+n]_\epsilon$. 
And we denote the complement of $\mathbf{C}^1(D_{n+2})$ by $\mathbf{C}^2(D_{n+2})$.

We know the structure of $ \mathbf{C}^1(D_{n+2})$. Now we study $\mathbf{C}^2(D_{n+2})$.
   It is obvious by (C2) and \eqref{G-eq:case3h} that for any configuration $C$ of $\ZZ D_{n+2}$, there exist at most two numbers $i,j \in \{1,2,\ldots, 2n\}$ such that  
   \begin{equation}
   \label{G-eq:crossingvertices}
   [i\ \ i+n]_+,[i+n\ \ i]_+,[j\ \ j+n]_-,[j+n\ \ j]_- \in C .
   \end{equation}  
By using this fact, we give a description of $ \mathbf{C}^2(D_{n+2})$ in the following lemma.
  \begin{lemma}
 $\mathbf{C}^2(D_{n+2}) =\{C \mid C \text{ is a configuration of $\ZZ D_{n+2}$ satisfying \eqref{G-eq:crossingvertices} with $i\neq j$} \}$.
 \end{lemma}
 \begin{proof}
 Let $C$ be a configuration of $\ZZ D_{n+2}$.
 If $C$ does not contain any vertex of the form $[i\ \ i+n]_\epsilon$ with $\epsilon \in \{+,-\}$, then $C \in \mathbf{C}^1(D_{n+2})$.
 
 Since $\omega([i\ \ i+n]_\epsilon)=[i+n\ \ i]_\epsilon$, it follows that $[i+n\ \ i]_\epsilon \in C$ if and only if $[i\ \ i+n]_\epsilon \in C$.
Without losing generality, we assume that $[i\ \ i+n]_+, [i+n\ \ i]_+ \in C$. 
 By (C1), there exists $c \in  C$ such that $[i\ \ i+n]_- \in H(c)$. 
If there is no other vertex of the form $[t\ \ t+n]_\epsilon$ in $ C$, then $c$ could only be a vertex in Case 1. However, in Case 1, $[i\ \ i+n]_- \in H(c)$ if and only if $[i\ \ i+n]_+ \in H(c)$. This contradicts with the facts that $[i\ \ i+n]_+ \in C$ and $\omega(c) \neq [i\ \ i+n]_+$. Hence, there exists $[j\ \ j+n]_- \in C$ such that $[i\ \ i+n]_- \in H([j\ \ j+n]_- )$. Moreover, $\omega([j\ \ j+n]_-)=[j+n\ \ j]_- \in C$.

When $i=j$, we have $C \in  \mathbf{C}^1(D_{n+2})$. 
Since $C$ can not contain more than $4$ vertices of the form $[t\ \ t+n]_\epsilon$, it follows that $\mathbf{C}^2(D_{n+2})$ consists of all the configurations containing $[i\ \ i+n]_+$, $[i+n\ \ i]_+$, $[j\ \ j+n]_-$ and $[j+n\ \ j]_-$ with $i \neq j$.
\end{proof}

 Now we define an involution $(-)^*$ on the set $\mathbf{C}^2(D_{n+2})$ as follows.
For each $C\in \mathbf{C}^2(D_{n+2})$, we define 
 \begin{align}
 \label{G-def:involution}
C^*:=\, &(C \smallsetminus \{[i\ \ i+n]_+,[i+n\ \ i]_+,[j\ \ j+n]_-,[j+n\ \ j]_- \} )\cup \notag \\
&\{[i\ \ i+n]_-,[i+n\ \ i]_-,[j\ \ j+n]_+,[j+n\ \ j]_+ \}.
\end{align}
It is clear that $(C^*)^*=C$ and $C^*$ is a different configuration of $\ZZ D_{n+2}$ than $C$.

There are two different kinds of configurations of type $D$ and there are symmetric $2$-Brauer relations and crossing $2$-Brauer relations. We have the following theorem of correspondence between the configurations of $\ZZ D_{n+2}$ and those two types of Brauer relations we defined.
 \begin{theorem}
 \label{G-thm:correspondenceD}
There are a bijection  
 \begin{align*} 
\Phi_1:  \mathbf{C}^1(D_{n+2})  \to  \mathbf{B}_{2n}^{2,s}, 
 \end{align*}
 and a two-to-one correspondence 
  \begin{align*} 
 \Phi_2: \mathbf{C}^2(D_{n+2})  \to \mathbf{B}_{2n}^{2,c}, 
   \end{align*} 
   where $\Phi_1(C)$ and $\Phi_2(C)=\Phi_2(C^*)$ are equivalence relations on the set $\{1,2, \ldots, 2n\}$ generated by $i \sim j$ for each $[j\ i]_\epsilon \in C$ with $\epsilon \in \{\varnothing, +,-\}$.
 \end{theorem}

We use $\ZZ D_4$ as an example.
 \begin{example}
 The symmetric $2$-Brauer relations of rank $4$ are listed in Example \ref{G-ex:brauerD4} and the crossing $2$-Brauer relation of rank $4$ is listed in Example \ref{G-ex:crossing}.
 Then the one-to-one correspondence $\mathbf{C}^1(D_4) \longleftrightarrow \mathbf{B}_4^{2,s}$ is the following.
\begin{center}
$B_1^s$ \hspace{0.5cm} \begin{tikzcd}[column sep=tiny, row sep=tiny]
\arrow[dotted,-]{rr}&&31_+\drar&&42_- \drar & &13_+\drar&&24_- \drar & &31_+ \drar &&42_- \drar & &13_+\arrow[dotted,-]{rr}&&\phantom{X}\\
\arrow[dotted,-]{rr}&&31_-\arrow{r}&41\urar\drar\arrow{r} &42_+\arrow{r} &12\urar\drar\arrow{r}&13_-\arrow{r}&23\urar\drar\arrow{r}&24_+\arrow{r} &34\urar\drar\arrow{r}&31_-\arrow{r} &41\urar\drar\arrow{r}&42_+\arrow{r} &12\urar\drar\arrow{r}&13_-\arrow[dotted,-]{rr} &&\phantom{X}\\
\arrow[dotted,-]{rr}&&\node[draw,circle,inner sep=1pt]{44};\urar&&\node[draw,circle,inner sep=1pt]{11};\urar&&\node[draw,circle,inner sep=1pt]{22};\urar&&\node[draw,circle,inner sep=1pt]{33};\urar&&\node[draw,circle,inner sep=1pt]{44};\urar&&\node[draw,circle,inner sep=1pt]{11};\urar&&\node[draw,circle,inner sep=1pt]{22};\arrow[dotted,-]{rr}&& \phantom{X}\\ 
\end{tikzcd}
\end{center}

\begin{center}
$B_2^s$ \hspace{0.5cm} \begin{tikzcd}[column sep=tiny, row sep=tiny]
\arrow[dotted,-]{rr}&&\node[draw,circle,inner sep=1pt]{31_+};\drar&&42_- \drar & &\node[draw,circle,inner sep=1pt]{13_+};\drar&&24_- \drar & &\node[draw,circle,inner sep=1pt]{31_+}; \drar &&42_- \drar & &\node[draw,circle,inner sep=1pt]{13_+};\arrow[dotted,-]{rr}&&\phantom{X}\\
\arrow[dotted,-]{rr}&&\node[draw,circle,inner sep=1pt]{31_-};\arrow{r}&41\urar\drar\arrow{r} &42_+\arrow{r} &12\urar\drar\arrow{r}&\node[draw,circle,inner sep=1pt]{13_-};\arrow{r}&23\urar\drar\arrow{r}&24_+\arrow{r} &34\urar\drar\arrow{r}&\node[draw,circle,inner sep=1pt]{31_-};\arrow{r} &41\urar\drar\arrow{r}&42_+\arrow{r} &12\urar\drar\arrow{r}&\node[draw,circle,inner sep=1pt]{13_-};\arrow[dotted,-]{rr} &&\phantom{X}\\
\arrow[dotted,-]{rr}&&\node[draw,circle,inner sep=1pt]{44};\urar&&11\urar&&\node[draw,circle,inner sep=1pt]{22};\urar&&33\urar&&\node[draw,circle,inner sep=1pt]{44};\urar&&11\urar&&\node[draw,circle,inner sep=1pt]{22};\arrow[dotted,-]{rr}&& \phantom{X}\\ 
\end{tikzcd}
\end{center}

\begin{center}
$B_3^s$ \hspace{0.5cm} \begin{tikzcd}[column sep=tiny, row sep=tiny]
\arrow[dotted,-]{rr}&&31_+\drar&&\node[draw,circle,inner sep=1pt]{42_-}; \drar & &13_+\drar&&\node[draw,circle,inner sep=1pt]{24_-}; \drar & &31_+ \drar &&\node[draw,circle,inner sep=1pt]{42_-}; \drar & &13_+\arrow[dotted,-]{rr}&&\phantom{X}\\
\arrow[dotted,-]{rr}&&31_-\arrow{r}&41\urar\drar\arrow{r} &\node[draw,circle,inner sep=1pt]{42_+};\arrow{r} &12\urar\drar\arrow{r}&13_-\arrow{r}&23\urar\drar\arrow{r}&\node[draw,circle,inner sep=1pt]{24_+};\arrow{r} &34\urar\drar\arrow{r}&31_-\arrow{r} &41\urar\drar\arrow{r}&\node[draw,circle,inner sep=1pt]{42_+};\arrow{r} &12\urar\drar\arrow{r}&13_-\arrow[dotted,-]{rr} &&\phantom{X}\\
\arrow[dotted,-]{rr}&&44\urar&&\node[draw,circle,inner sep=1pt]{11};\urar&&22\urar&&\node[draw,circle,inner sep=1pt]{33};\urar&&44\urar&&\node[draw,circle,inner sep=1pt]{11};\urar&&22\arrow[dotted,-]{rr}&& \phantom{X}\\ 
\end{tikzcd}
\end{center}

\begin{center}
$B_4^s$ \hspace{0.5cm} \begin{tikzcd}[column sep=tiny, row sep=tiny]
\arrow[dotted,-]{rr}&&31_+\drar&&42_- \drar & &13_+\drar&&24_- \drar & &31_+ \drar &&42_- \drar & &13_+\arrow[dotted,-]{rr}&&\phantom{X}\\
\arrow[dotted,-]{rr}&&31_-\arrow{r}&41\urar\drar\arrow{r} &42_+\arrow{r} &\node[draw,circle,inner sep=1pt]{12};\urar\drar\arrow{r}&13_-\arrow{r}&23\urar\drar\arrow{r}&24_+\arrow{r} &\node[draw,circle,inner sep=1pt]{34};\urar\drar\arrow{r}&31_-\arrow{r} &41\urar\drar\arrow{r}&42_+\arrow{r} &\node[draw,circle,inner sep=1pt]{12};\urar\drar\arrow{r}&13_-\arrow[dotted,-]{rr} &&\phantom{X}\\
\arrow[dotted,-]{rr}&&44\urar&&11\urar&&22\urar&&33\urar&&44\urar&&11\urar&&22\arrow[dotted,-]{rr}&& \phantom{X}\\ 
\end{tikzcd}
\end{center}

\begin{center}
$B_5^s$ \hspace{0.5cm} \begin{tikzcd}[column sep=tiny, row sep=tiny]
\arrow[dotted,-]{rr}&&31_+\drar&&42_- \drar & &13_+\drar&&24_- \drar & &31_+ \drar &&42_- \drar & &13_+\arrow[dotted,-]{rr}&&\phantom{X}\\
\arrow[dotted,-]{rr}&&31_-\arrow{r}&\node[draw,circle,inner sep=1pt]{41};\urar\drar\arrow{r} &42_+\arrow{r} &12\urar\drar\arrow{r}&13_-\arrow{r}&\node[draw,circle,inner sep=1pt]{23};\urar\drar\arrow{r}&24_+\arrow{r} &34\urar\drar\arrow{r}&31_-\arrow{r} &\node[draw,circle,inner sep=1pt]{41};\urar\drar\arrow{r}&42_+\arrow{r} &12\urar\drar\arrow{r}&13_-\arrow[dotted,-]{rr} &&\phantom{X}\\
\arrow[dotted,-]{rr}&&44\urar&&11\urar&&22\urar&&33\urar&&44\urar&&11\urar&&22\arrow[dotted,-]{rr}&& \phantom{X}\\ 
\end{tikzcd}
\end{center}

The two-to-one correspondence between $\mathbf{C}^2(D_4)$ and $\mathbf{B}_4^{2,c}$ is given as follows.
\begin{center}
$B_1^c$ \hspace{0.5cm} \begin{tikzcd}[column sep=tiny, row sep=tiny]
\arrow[dotted,-]{rr}&&\node[draw,circle,inner sep=1pt]{31_+};\drar&&\node[draw,circle,inner sep=1pt]{42_-}; \drar & &\node[draw,circle,inner sep=1pt]{13_+};\drar&&\node[draw,circle,inner sep=1pt]{24_-}; \drar & &\node[draw,circle,inner sep=1pt]{31_+}; \drar &&\node[draw,circle,inner sep=1pt]{42_-}; \drar & &\node[draw,circle,inner sep=1pt]{13_+};\arrow[dotted,-]{rr}&&\phantom{X}\\
\arrow[dotted,-]{rr}&&31_-\arrow{r}&41\urar\drar\arrow{r} &42_+\arrow{r} &12\urar\drar\arrow{r}&13_-\arrow{r}&23\urar\drar\arrow{r}&24_+\arrow{r} &34\urar\drar\arrow{r}&31_-\arrow{r} &41\urar\drar\arrow{r}&42_+\arrow{r} &12\urar\drar\arrow{r}&13_-\arrow[dotted,-]{rr} &&\phantom{X}\\
\arrow[dotted,-]{rr}&&44\urar&&11\urar&&22\urar&&33\urar&&44\urar&&11\urar&&22\arrow[dotted,-]{rr}&& \phantom{X}\\ 
\end{tikzcd}
\end{center}
\begin{center}
$B_1^c$ \hspace{0.5cm} \begin{tikzcd}[column sep=tiny, row sep=tiny]
\arrow[dotted,-]{rr}&&31_+\drar&&42_- \drar & &13_+\drar&&24_- \drar & &31_+ \drar &&42_- \drar & &13_+\arrow[dotted,-]{rr}&&\phantom{X}\\
\arrow[dotted,-]{rr}&&\node[draw,circle,inner sep=1pt]{31_-};\arrow{r}&41\urar\drar\arrow{r} &\node[draw,circle,inner sep=1pt]{42_+};\arrow{r} &12\urar\drar\arrow{r}&\node[draw,circle,inner sep=1pt]{13_-};\arrow{r}&23\urar\drar\arrow{r}&\node[draw,circle,inner sep=1pt]{24_+};\arrow{r} &34\urar\drar\arrow{r}&\node[draw,circle,inner sep=1pt]{31_-};\arrow{r} &41\urar\drar\arrow{r}&\node[draw,circle,inner sep=1pt]{42_+};\arrow{r} &12\urar\drar\arrow{r}&\node[draw,circle,inner sep=1pt]{13_-};\arrow[dotted,-]{rr} &&\phantom{X}\\
\arrow[dotted,-]{rr}&&44\urar&&11\urar&&22\urar&&33\urar&&44\urar&&11\urar&&22\arrow[dotted,-]{rr}&& \phantom{X}\\ 
\end{tikzcd}
\end{center}
\end{example}
 
\begin{proof}[Proof of Theorem \ref{G-thm:correspondenceD}]
First, we show that the two maps $$\Phi_1: \mathbf{C}^1(D_{n+2}) \to \mathbf{B}_{2n}^{2,s} \and \Phi_2:  \mathbf{C}^2(D_{n+2}) \to \mathbf{B}_{2n}^{2,c}$$ are well-defined.
Let $C$ be a configuration of $\ZZ D_{n+2}$.
  By a similar proof as Theorem \ref{G-thm:correspondenceB}, we know that $\Phi_1(C)$ (or $\Phi_2(C)$) is an equivalence relation on the set $\{1,2, \ldots, 2n\}$, symmetric with respect to the rotation by $\pi$ and each equivalence class contains at most two elements.

Assume that $ C\in \mathbf{C}^1(D_{n+2}).$ 
Identify $\{[i\ \ i+n]_+, [i\ \ i+n]_- \}$ with $[i\ \ i+n]_\pm$.
Then $C$ becomes a configuration of type $B$ with $[i\ \ i+n]_\pm$.
Thus, by Theorem \ref{G-thm:correspondenceB}, we have $\Phi_1(C)\in \mathbf{B}_{2n}^{2,s}.$

Assume that $C \in \mathbf{C}^2(D_{n+2})$ and $\{[i\ \ i+n]_+,[i+n\ \ i]_+,[j\ \ j+n]_-,[j+n\ \ j]_- \} \subset  C$ with $i\neq j$.
Since $i\neq j$, we know that the convex hull of $i \sim i+n$ and the convex hull of $j \sim j+n$ are joint. 
By \eqref{G-eq:case3}, \eqref{G-eq:case3h} and (C2), no convex hulls of the equivalence classes of vertices in Case 1 cross the convex hulls of $i \sim i+n$ and $j \sim j+n$. According to \eqref{G-def:involution}, the configuration $C^*$ corresponds to the same equivalence relation. Therefore, $\Phi_2(C)=\Phi_2(C^*)\in \mathbf{B}_{2n}^{2,c}.$

Notice that for any $s,t \in \{1,2, \ldots, 2n\}$ with $s \neq t\pm n$, there is only one of $\{[s\ t], [t\ s]\}$ representing a vertex in $(\ZZ D_{n+2})_0$ for our labels. And when $s= t + n= t-n$, both $[t+n \ \ t]_\pm$ and $[t\ \ t+n]_\pm$ are in $(\ZZ D_{n+2})_0$.
We construct a map
$$\Psi_1:   \mathbf{B}_{2n}^{2,s} \to \mathbf{C}^1(D_{n+2}) \quad \text{by} \quad B\mapsto \Psi_1(B):=\{[1 \ \sigma_{B}(1)]_\epsilon,\ldots, [2n \ \sigma_{B}(2n)]_\epsilon \}\cap (\ZZ D_{n+2})_0 , $$ 
where $\epsilon = \varnothing$ when $\sigma_{B}(i) \neq i \pm n$ and $\epsilon = \pm$ when $\sigma_{B}(i) = i \pm n$ for $1\le i \le 2n$.
Similarly as we did in the proof of Theorem \ref{G-thm:correspondenceB}, we have $\Psi_1(B) \in \mathbf{C}^1(D_{n+2}).$ Hence, the map $\Psi_1$ is well-defined.

We construct a map $$\Psi_2: \mathbf{B}_{2n}^{2,c} \to \{ \{C, C^*\} \mid C \in \mathbf{C}^2(D_{n+2})\}  \quad \text{by} \quad B\mapsto \Psi_2(B):=\{C_B, C^*_B\},$$
where $C_B$ and $ C^*_B$ are defined as follows.
For each equivalence class $i \sim j$ of $B$ with $j \neq i\pm n $, we put the vertex $[j \ i]$ in $C_B$.
Since there exist $s<t$ such that the equivalence classes $s\sim s+n$ and $t \sim t+n$ are in $B$, we put $$[s\ \ s+n]_+, [s+n\ \ s]_+, [t\ \ t+n]_- , [t+n\ \ t]_- \in C_B.$$
By exchanging $+$ and $-$, we get $C^*_B$.
Now, we prove that $\Psi_2$ is well-defined.
For any vertex $[j\ i] \in C_B$, we have $[j+n\ \ i+n] \in C_B$ since $B$ is symmetric.
As the convex hull of each equivalence class $i \sim j$ with $j \neq i\pm n$ is disjoint with any other convex hulls, by \eqref{G-eq:case2h}, we have $H([j+n\ \ i+n]) \cap C_B =\{[j\ i], [j+n\ \ i+n]\}.$
Similarly, by \eqref{G-eq:case3h}, we have 
$$H([s\ \ s+n]_{+}) \cap C_B =\{[s\ \ s+n]_{+}, [s+n\ \ s]_{+}\} \and H([t\ \ t+n]_{-}) \cap C_B =\{[t\ \ t+n]_{-}, [t+n\ \ t]_{-}\}.$$
 Therefore, (C2) holds.
For the vertices of the forms $[r\ \ r+n]_+$ and $[r\ \ r+n]_-$, by \eqref{G-eq:case3h}, we have 
$$[r\ \ r+n]_{+} \in H([s\ \ s+n]_+) \cup H([s+n\ \ s]_+) \and [r\ \ r+n]_- \in H([t\ \ t+n]_-) \cup H([t+n\ \ t]_-).$$
For other vertices of the form $[r\ s] \in (\ZZ D_{n+2})_0$ with $s \neq r \pm n$, there exists an equivalence class $r \sim t$ in $B$ and $[r\ s] \in H([r\ t]) \cup H([r+n\ \ t+n])$. So (C1) holds. 
Thus, $C_B$ and $C^*_B$ are configurations of $\ZZ D_{n+2}$. The map $\Psi_2$ is well-defined.

By a similar reasoning in Theorem \ref{G-thm:correspondenceB}, we have $\Phi_1  \Psi_1=\Id$ and $\Psi_1  \Phi_1=\Id$.
It is immediate that $\Phi_2  \Psi_2= \Id$. It is clear that $C \subset D$ or $C \subset D^*$ for $\Psi_2  \Phi_2(C)= \{D, D^*\}$. Using (C1), we get $C= D$ or $C = D^*$ as in the proof of Theorem \ref{G-thm:correspondenceB}.
 \end{proof}
 
 \section{Types $E$, $F$ and $G$}
 \label{G-s:appendix}
In this section, for sake of completeness, we discuss the exceptional cases and give a list of configurations of types $E_6$, $E_7$, $E_8$, $F_4$ and $G_2$. The cases $E_6$, $E_7$ and $E_8$ were already studied in \cite{Hummel} (unpublished).

 \begin{theorem}
 \label{G-th:exceptional}
 For exceptional cases, the number of configurations is given by the following table
   \begin{center}
\begin{tabular}{|c|c|c|c|c|}
\hline
$E_6$ & $E_7$ &$E_8$ & $F_4$ &$G_2$\\
\hline
$77$&$346$ &$1892$ &$25$ &$4$\\
\hline
$11$&$44$ &$138$ &$5$ &$2$\\
\hline
\end{tabular}
\end{center}
where the last row is the number of configurations modulo the action of $\tau$.
\end{theorem}
 
 We give the list of configurations of types $E_6$, $E_7$, $E_8$, $F_4$ and $G_2$ modulo the action of $\tau$ as follows.
 
By calculation, configurations of $\ZZ E_6$ are stable under $\tau^5 \rho$, where $\tau$ is the Auslander-Reiten translation and $\rho$ is the only non-trivial automorphism of $E_6$. By Proposition \ref{G-pro:covering}, we only need to find the list of configurations of $\ZZ E_6 /\langle\tau^5 \rho\rangle$. We give the list of configurations of $\ZZ E_6 /\langle\tau^5\rho\rangle$ modulo the action of $\tau$.
 \begin{center}
    \includegraphics{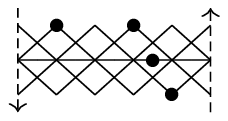}
\includegraphics{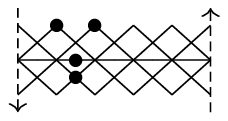}
\includegraphics{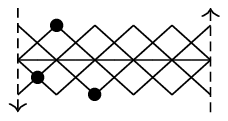}
\includegraphics{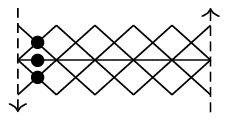}
\includegraphics{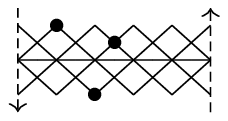}
\includegraphics{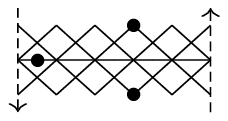}

\includegraphics{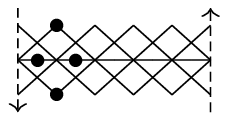}
\includegraphics{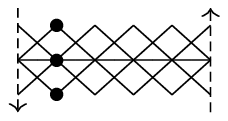}
\includegraphics{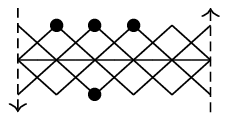}
\includegraphics{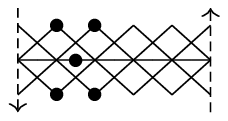}
\includegraphics{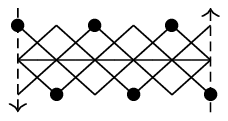}

 \end{center}

 By calculation, configurations of $\ZZ E_7$ are stable under $\tau^8 $. By Proposition \ref{G-pro:covering}, we only need to find out configurations of $\ZZ E_7 /\langle\tau^8\rangle$. We give the list of configurations of $\ZZ E_7 /\langle\tau^8\rangle$ modulo the action of $\tau$.
  \begin{center}
      \includegraphics{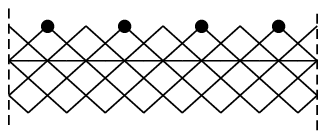}
\includegraphics{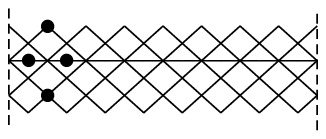}
\includegraphics{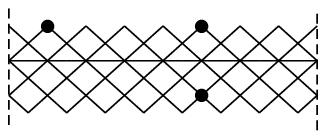}
\includegraphics{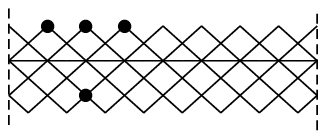}

\includegraphics{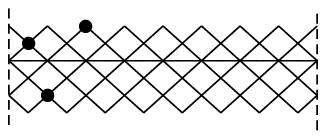}
\includegraphics{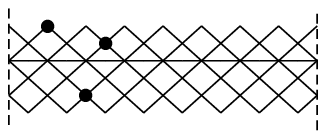}
\includegraphics{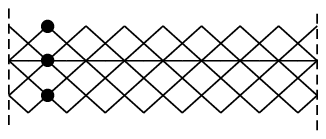}
\includegraphics{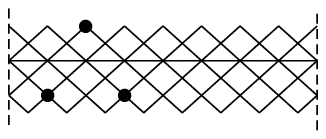}

\includegraphics{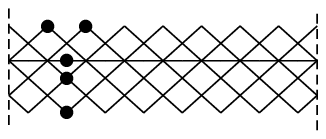}
\includegraphics{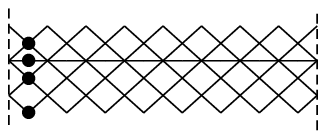}
\includegraphics{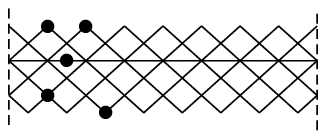}
\includegraphics{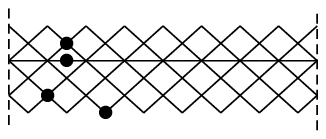}

\includegraphics{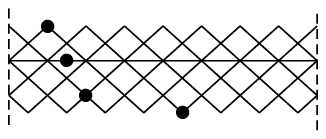}
\includegraphics{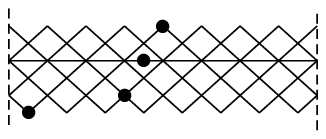}
\includegraphics{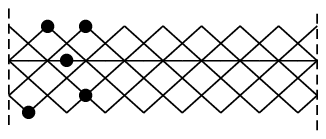}
\includegraphics{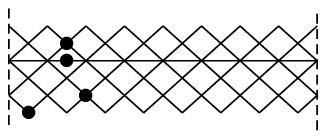}

\includegraphics{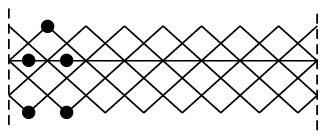}
\includegraphics{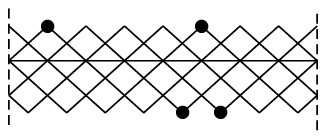}
\includegraphics{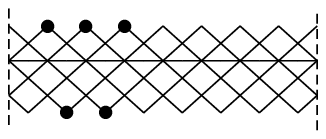}
\includegraphics{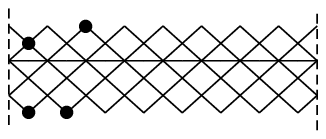}

\includegraphics{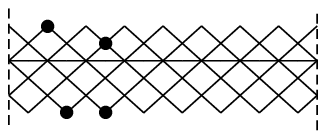}
\includegraphics{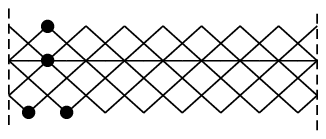}
\includegraphics{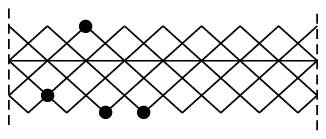}
\includegraphics{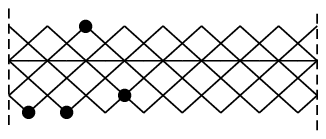}

\includegraphics{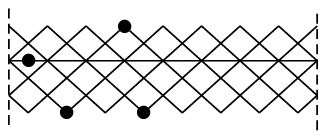}
\includegraphics{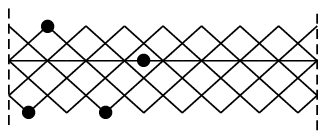}
\includegraphics{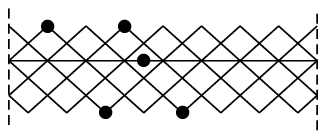}
\includegraphics{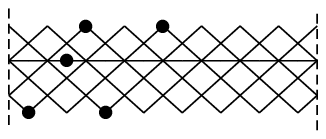}

\includegraphics{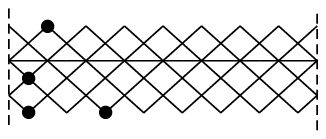}
\includegraphics{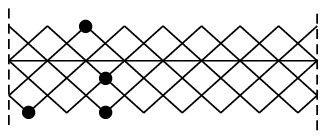}
\includegraphics{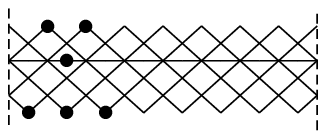}
\includegraphics{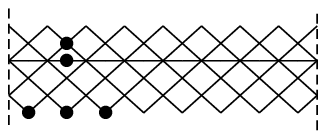}

\includegraphics{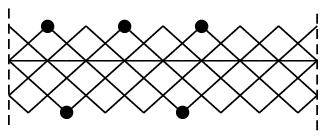}
\includegraphics{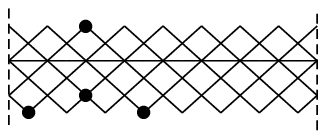}
\includegraphics{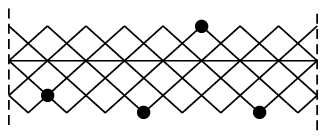}
\includegraphics{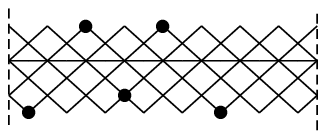}

\includegraphics{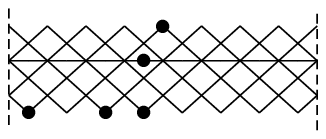}
\includegraphics{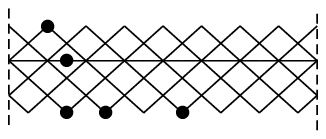}
\includegraphics{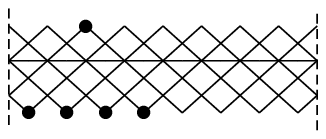}
\includegraphics{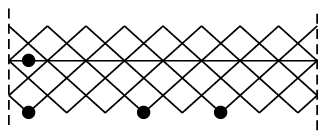}

\includegraphics{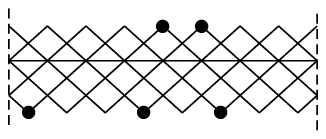}
\includegraphics{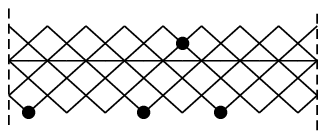}
\includegraphics{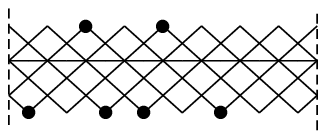}
\includegraphics{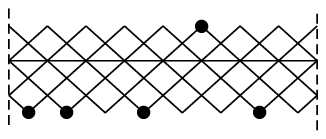}

 \end{center}
 
Similarly, configurations of $\ZZ E_8$ are preserved by $\tau^{14}$. We give the list of configurations of $\ZZ E_8 /\langle\tau^{14}\rangle$ modulo the action of $\tau$.
 \begin{center}
        \includegraphics{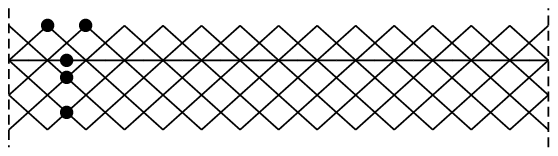}
\includegraphics{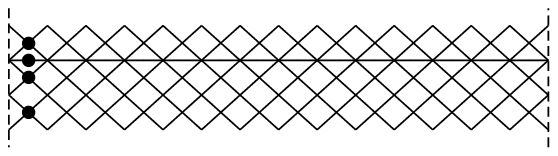}

\includegraphics{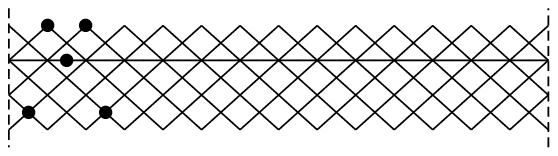}
\includegraphics{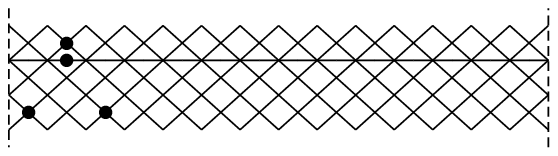}

\includegraphics{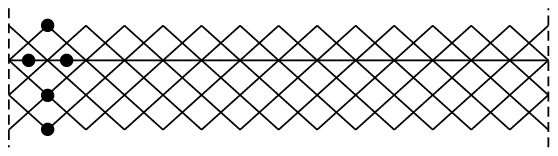}
\includegraphics{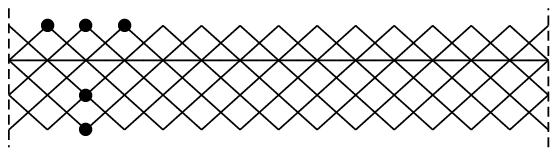}

\includegraphics{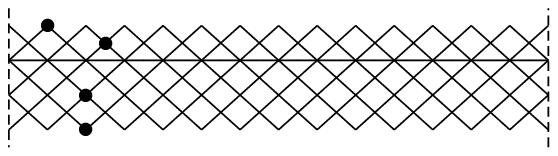}
\includegraphics{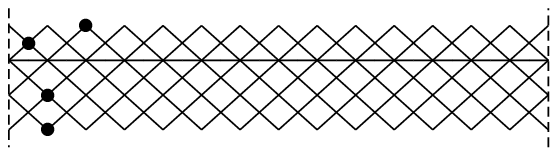}

\includegraphics{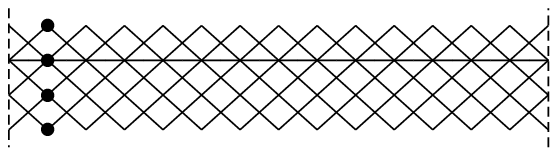}
\includegraphics{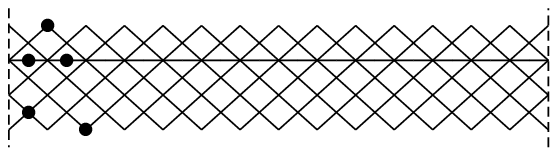}

\includegraphics{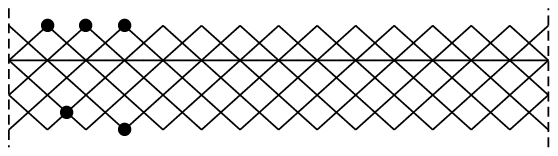}
\includegraphics{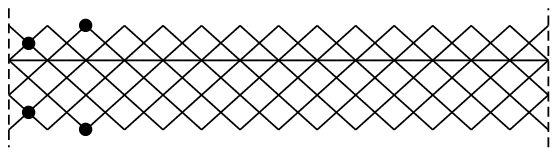}

\includegraphics{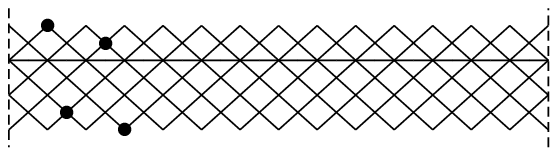}
\includegraphics{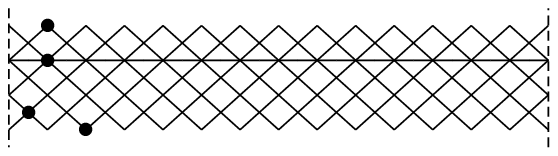}

\includegraphics{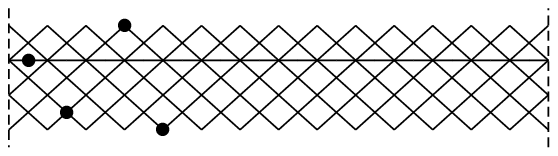}
\includegraphics{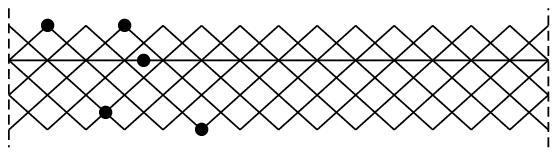}

\includegraphics{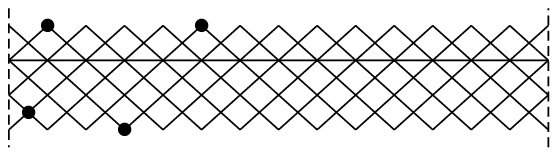}
\includegraphics{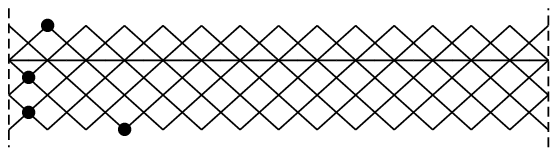}

\includegraphics{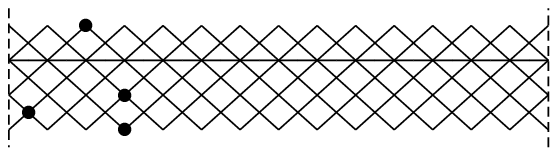}
\includegraphics{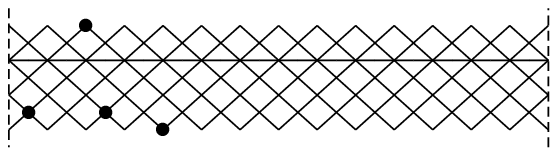}

\includegraphics{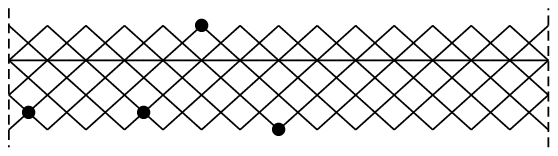}
\includegraphics{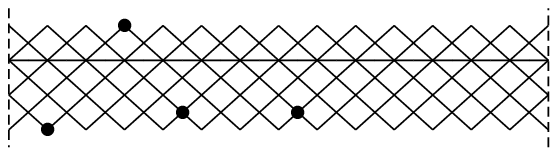}

\includegraphics{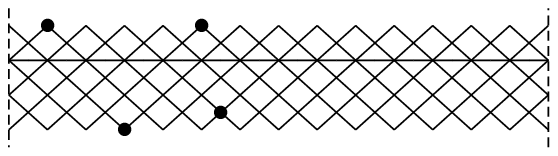}
\includegraphics{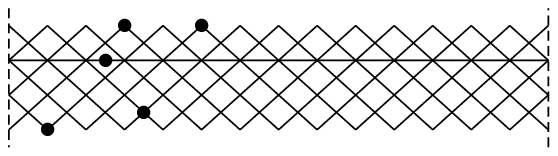}

\includegraphics{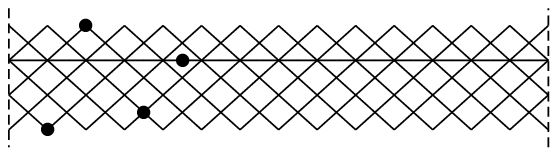}
\includegraphics{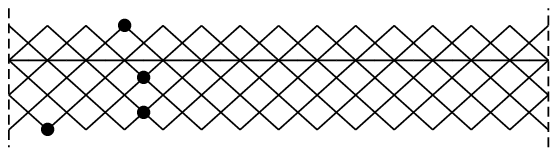}

\includegraphics{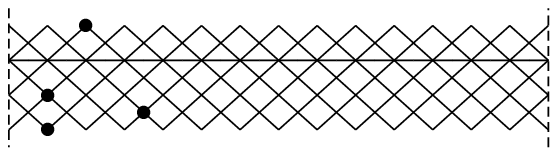}
\includegraphics{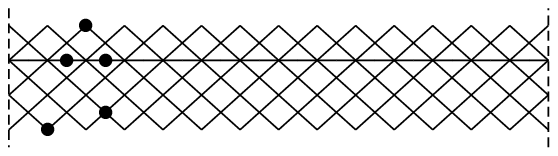}

\includegraphics{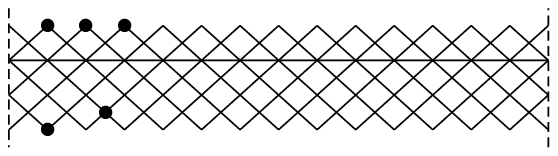}
\includegraphics{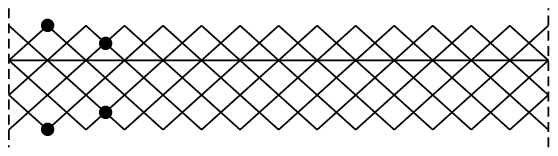}

\includegraphics{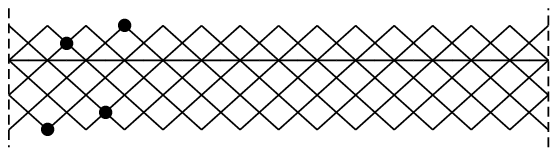}
\includegraphics{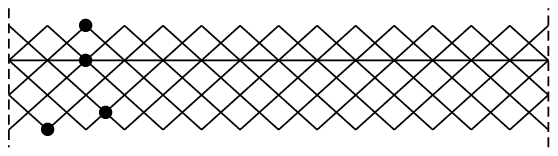}

\includegraphics{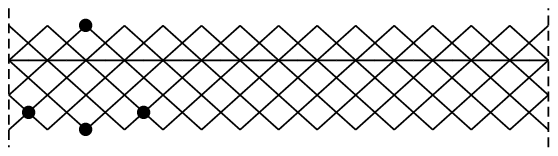}
\includegraphics{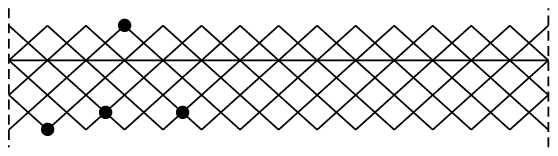}

\includegraphics{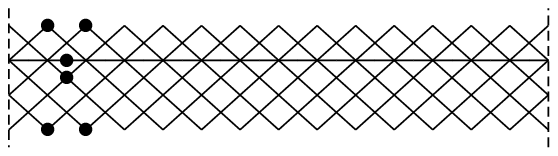}
\includegraphics{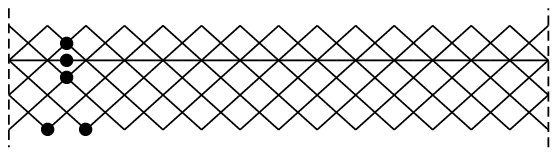}

\includegraphics{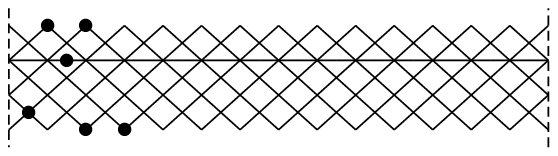}
\includegraphics{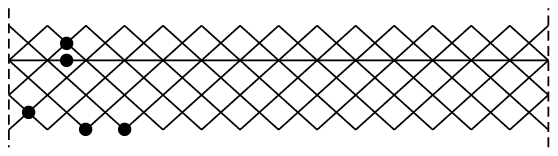}

\includegraphics{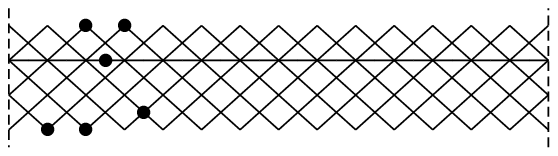}
\includegraphics{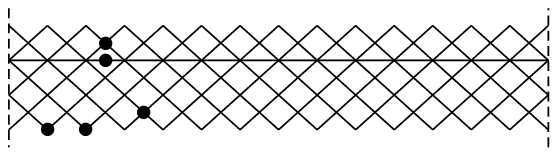}

\includegraphics{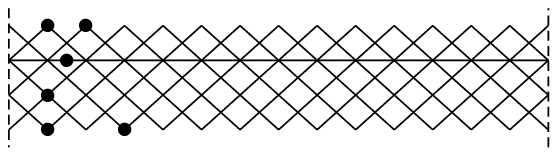}
\includegraphics{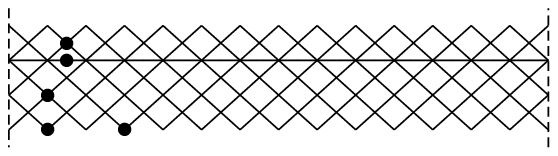}

\includegraphics{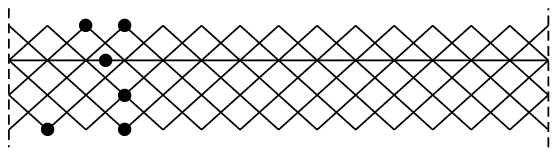}
\includegraphics{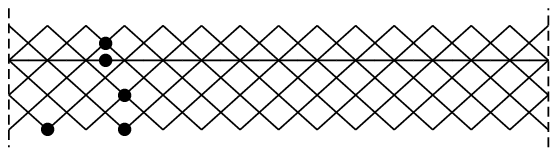}

\includegraphics{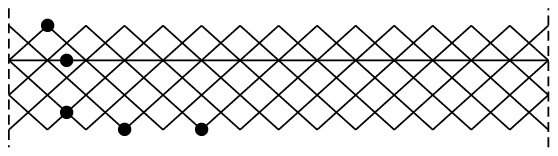}
\includegraphics{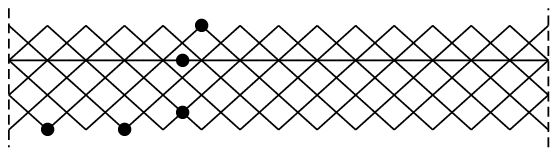}

\includegraphics{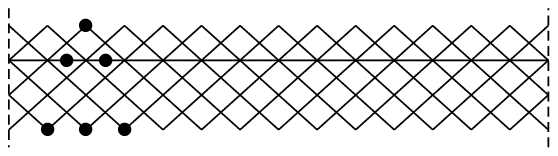}
\includegraphics{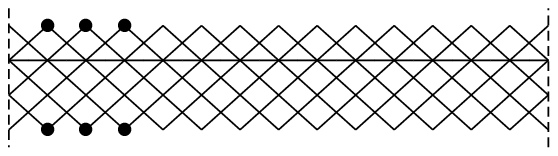}

\includegraphics{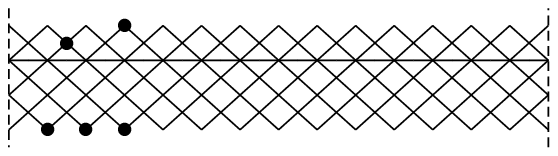}
\includegraphics{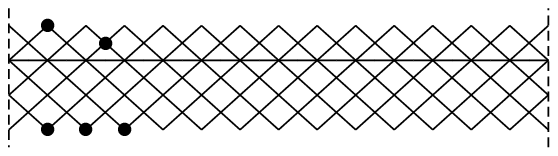}

\includegraphics{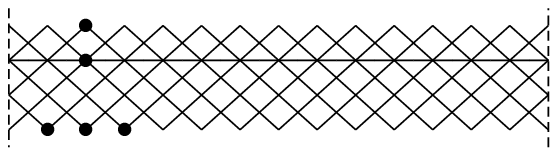}
\includegraphics{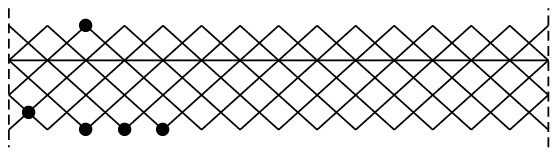}

\includegraphics{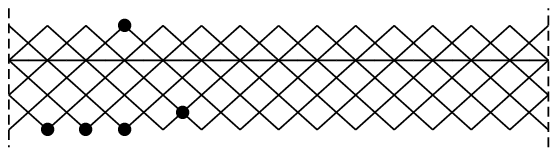}
\includegraphics{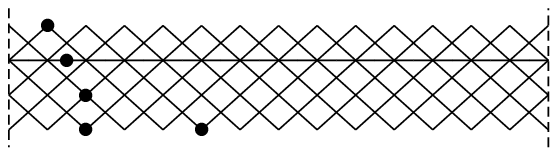}

\includegraphics{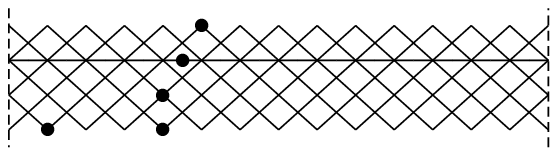}
\includegraphics{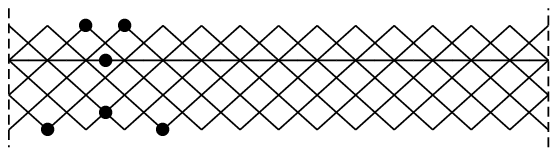}

\includegraphics{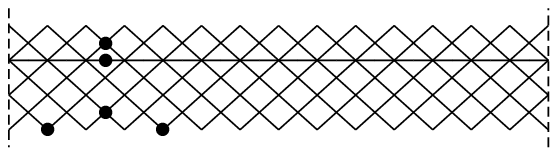}
\includegraphics{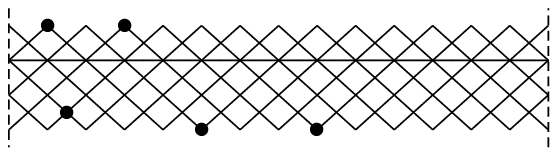}

\includegraphics{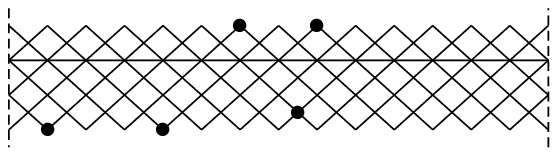}
\includegraphics{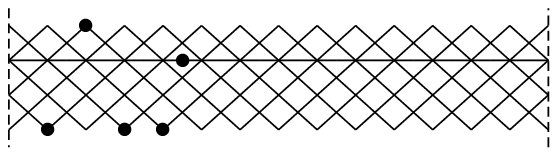}

\includegraphics{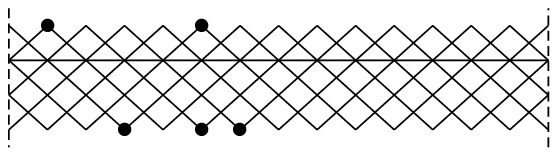}
\includegraphics{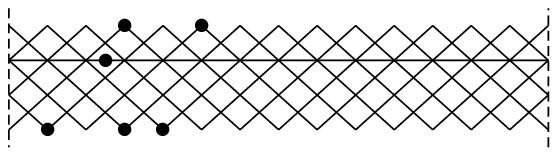}

\includegraphics{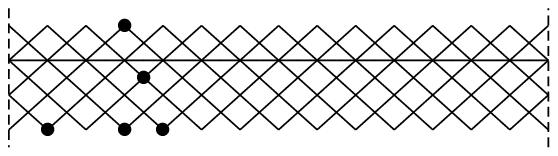}
\includegraphics{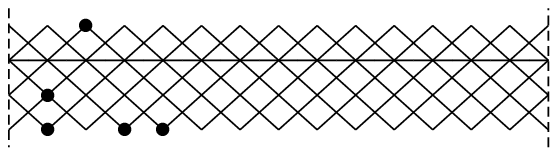}

\includegraphics{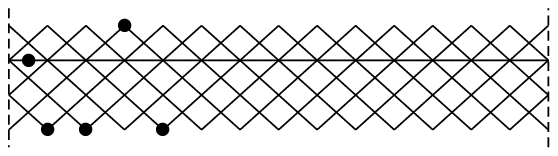}
\includegraphics{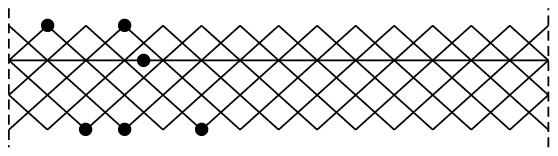}

\includegraphics{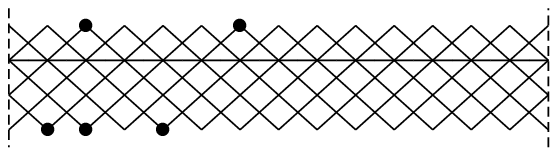}
\includegraphics{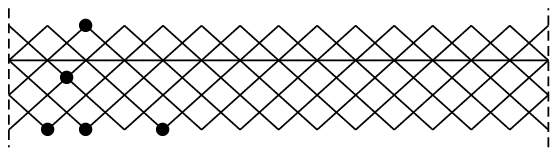}

\includegraphics{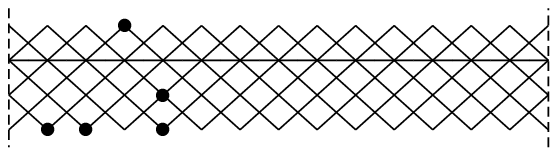}
\includegraphics{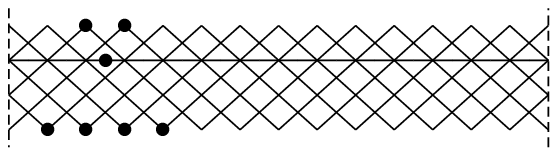}

\includegraphics{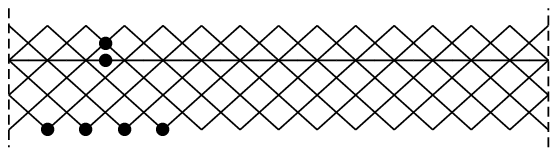}
\includegraphics{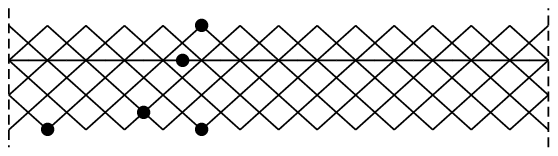}

\includegraphics{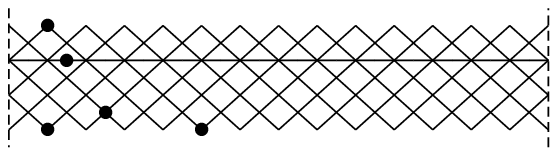}
\includegraphics{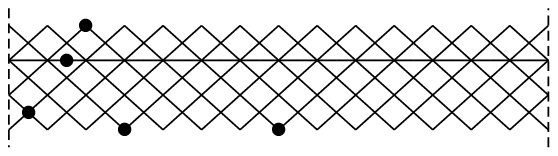}

\includegraphics{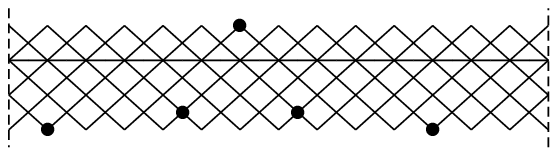}
\includegraphics{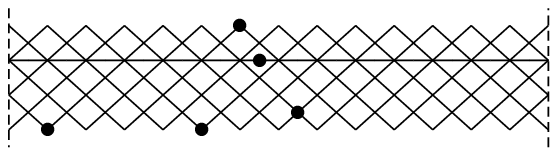}

\includegraphics{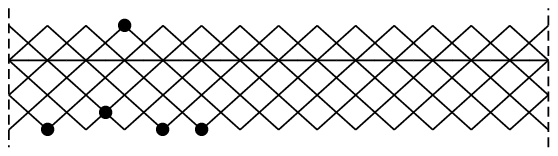}
\includegraphics{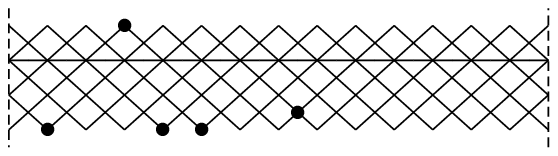}

\includegraphics{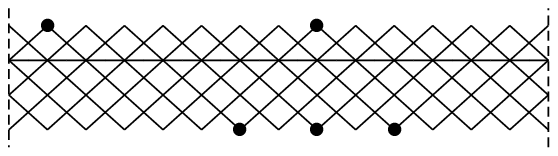}
\includegraphics{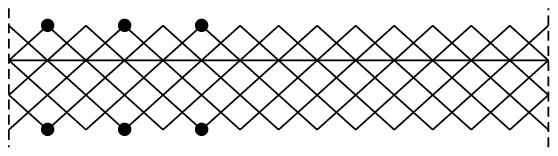}

\includegraphics{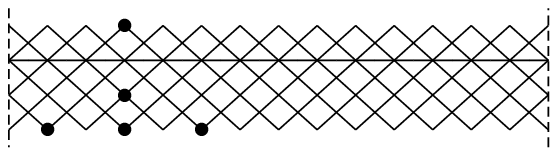}
\includegraphics{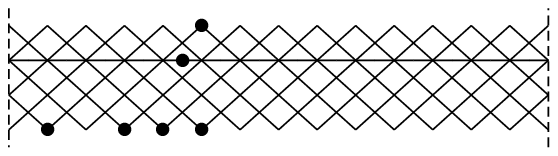}

\includegraphics{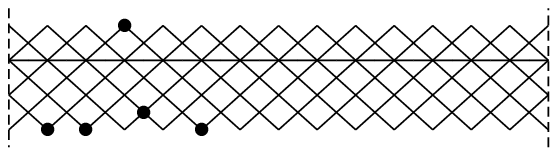}
\includegraphics{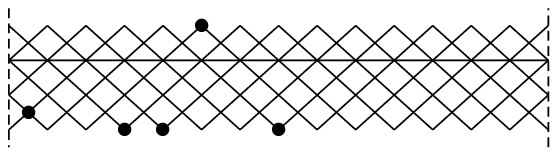}

\includegraphics{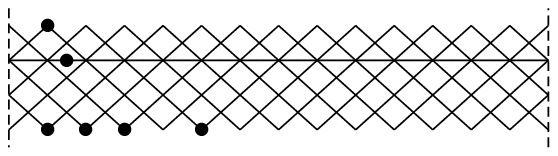}
\includegraphics{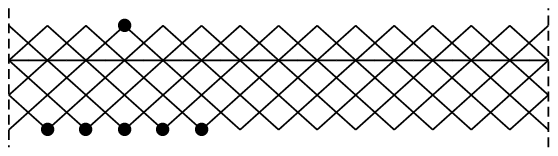}

\includegraphics{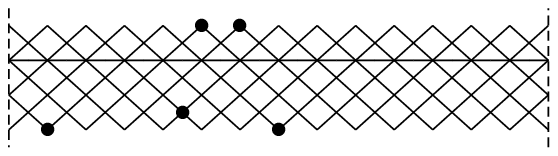}
\includegraphics{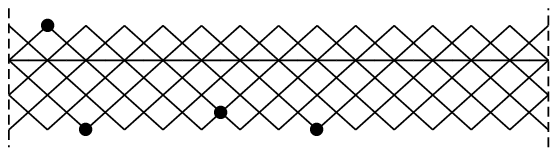}

\includegraphics{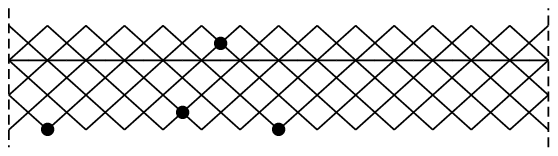}
\includegraphics{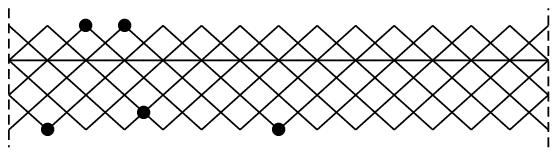}

\includegraphics{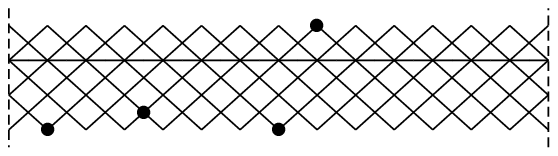}
\includegraphics{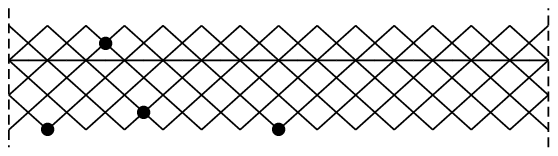}

\includegraphics{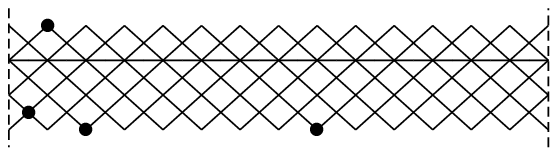}
\includegraphics{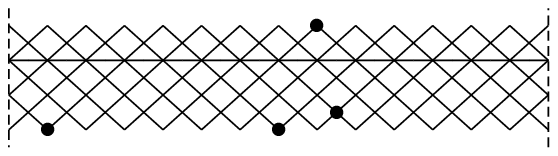}

\includegraphics{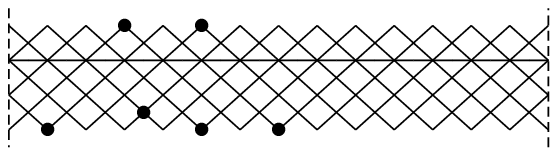}
\includegraphics{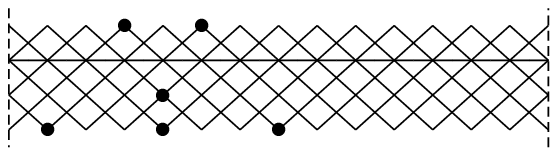}

\includegraphics{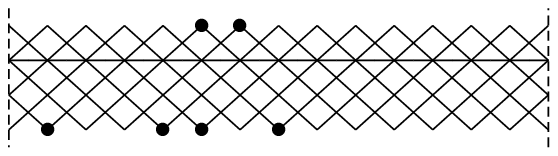}
\includegraphics{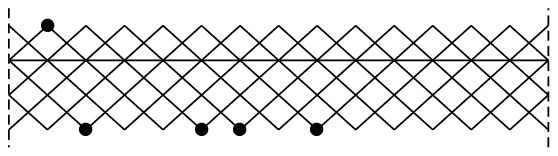}

\includegraphics{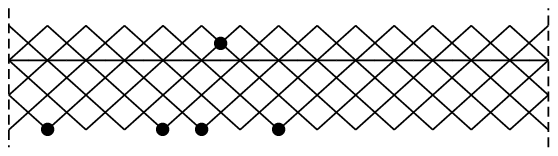}
\includegraphics{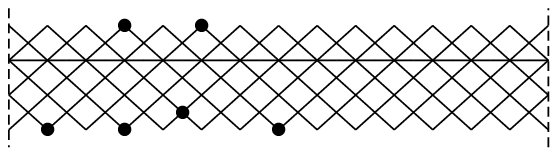}

\includegraphics{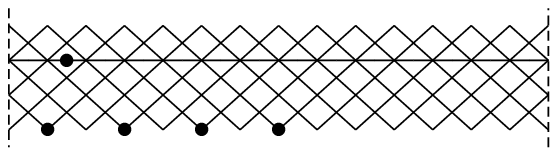}
\includegraphics{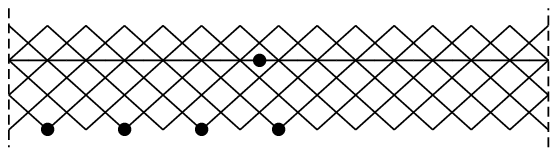}

\includegraphics{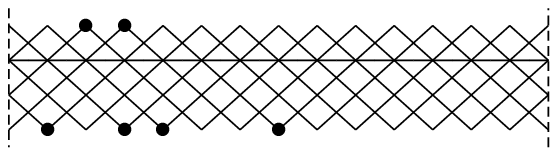}
\includegraphics{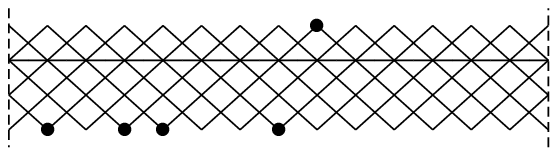}

\includegraphics{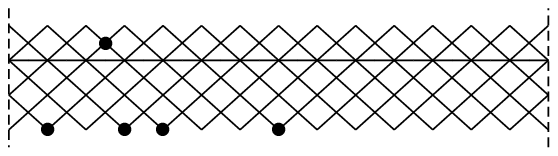}
\includegraphics{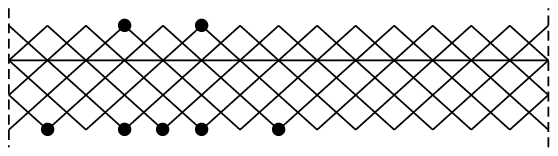}

\includegraphics{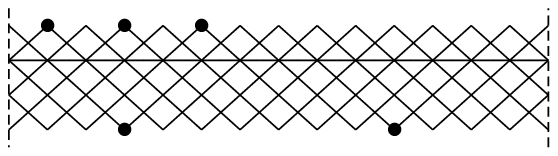}
\includegraphics{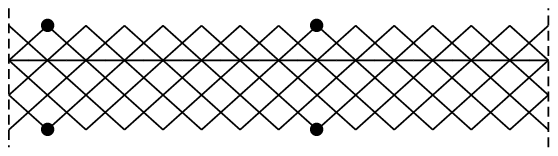}

\includegraphics{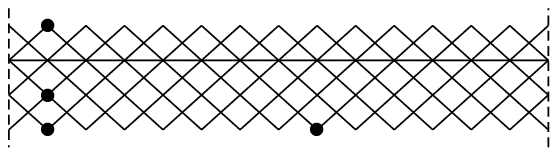}
\includegraphics{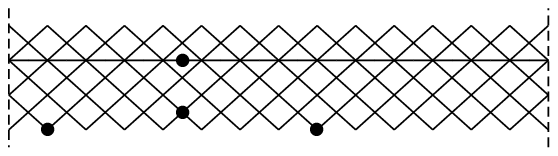}

\includegraphics{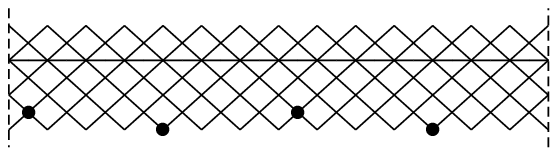}
\includegraphics{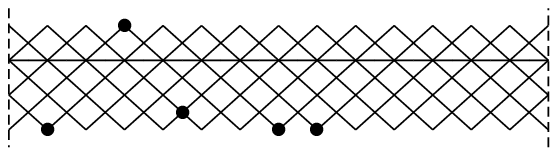}

\includegraphics{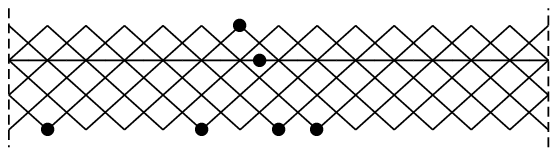}
\includegraphics{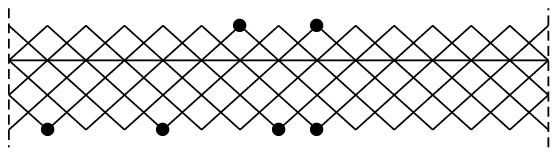}

\includegraphics{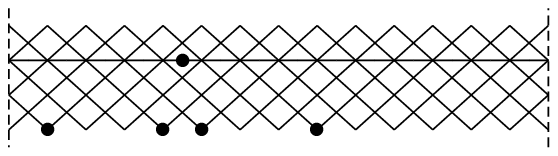}
\includegraphics{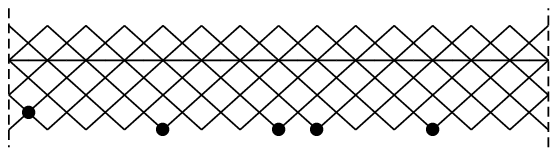}

\includegraphics{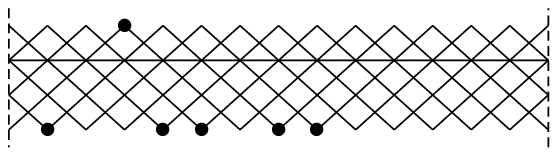}
\includegraphics{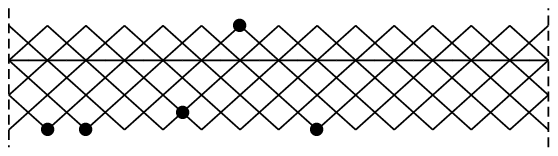}

\includegraphics{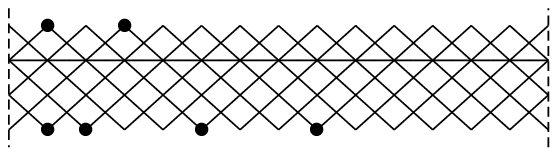}
\includegraphics{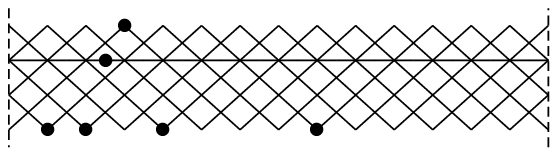}

\includegraphics{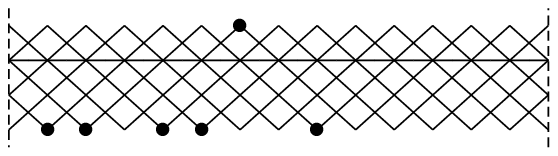}
\includegraphics{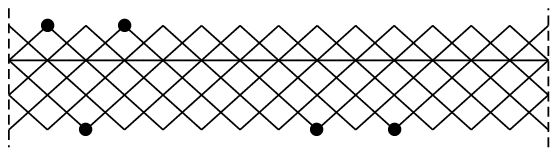}

\includegraphics{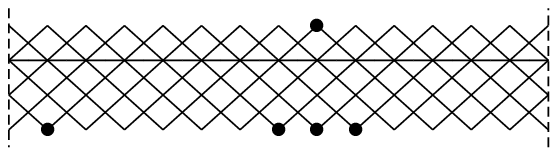}
\includegraphics{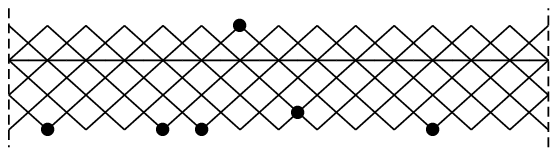}

\includegraphics{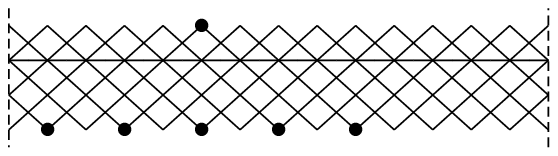}
\includegraphics{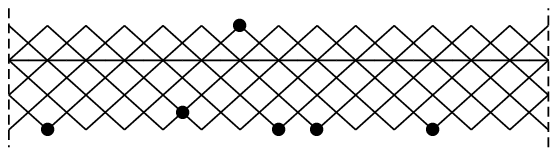}

\includegraphics{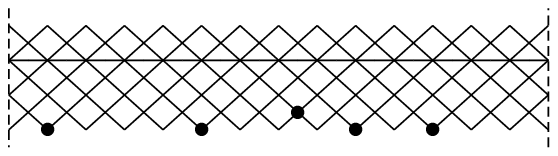}
\includegraphics{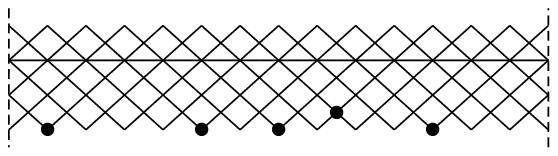}

\includegraphics{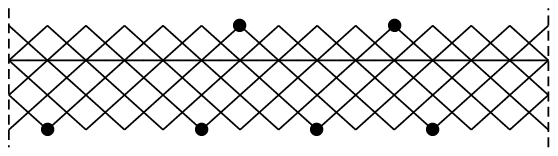}
\includegraphics{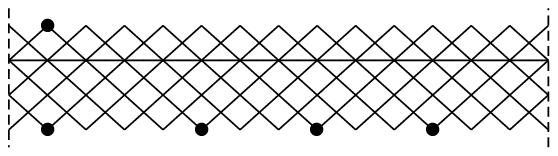}

\includegraphics{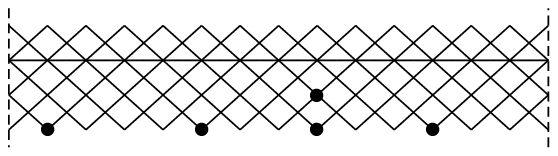}
\includegraphics{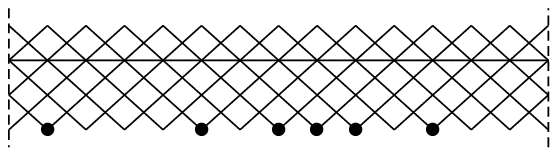}

\includegraphics{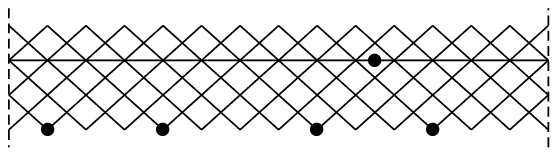}
\includegraphics{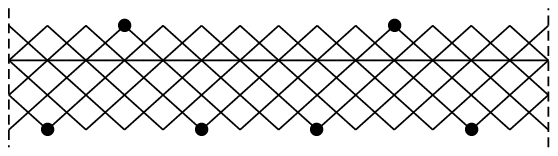}

\includegraphics{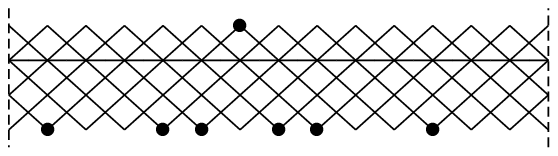}
\includegraphics{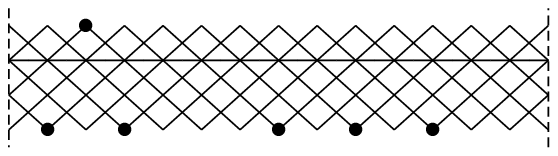}

\includegraphics{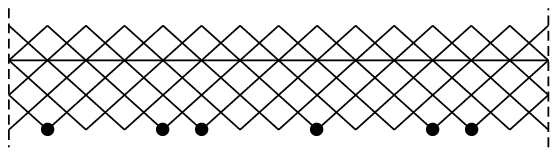}
\includegraphics{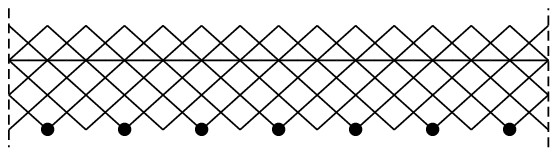}

 \end{center}
 
Similarly, configurations of $\ZZ F_4$ are preserved by $\tau^{5}$. We give the list of configurations of $\ZZ F_4 /\langle\tau^5 \rangle$ modulo the action of $\tau$.
 \begin{center}
        \includegraphics{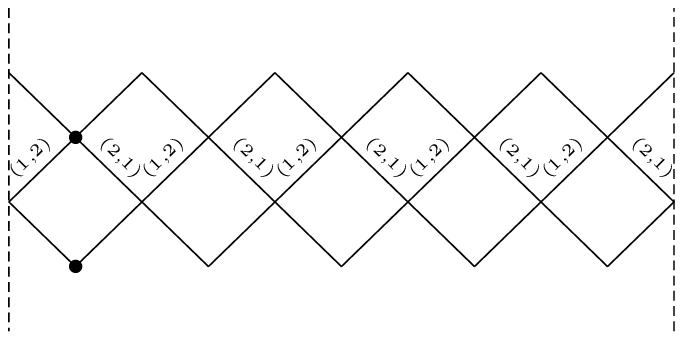}
\includegraphics{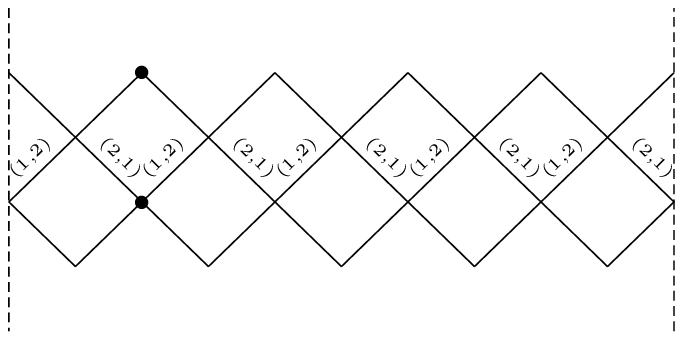}

\includegraphics{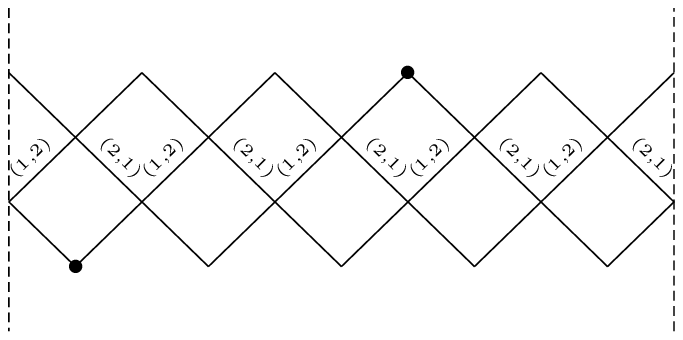}
\includegraphics{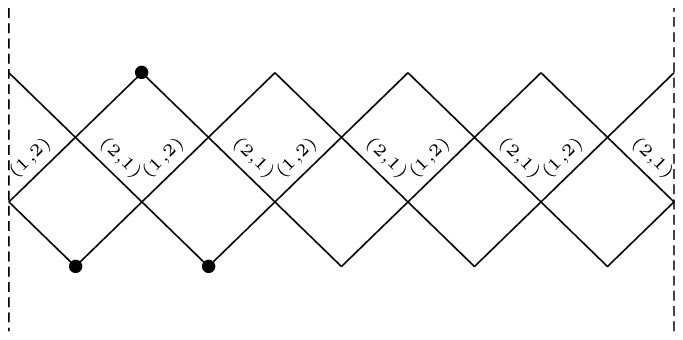}

\includegraphics{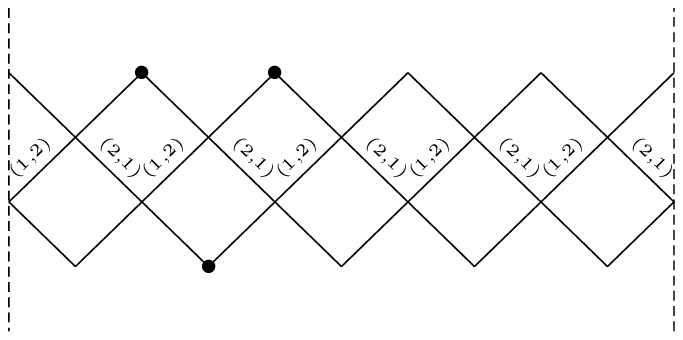}

 \end{center}
 
 Similarly, configurations of $\ZZ G_2$ are preserved by $\tau^2$. We give the list of configurations of $\ZZ G_2 /\langle\tau^2 \rangle$ modulo the action of $\tau$.
 \begin{center}
        \includegraphics{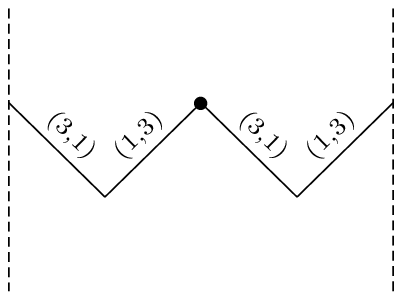}
\includegraphics{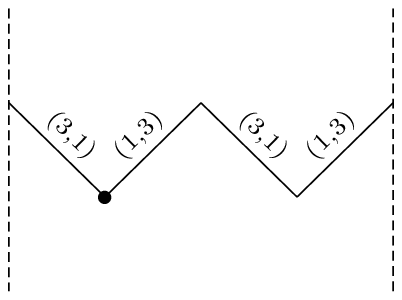}

 \end{center}

\bibliographystyle{plain}
\bibliography{../biblio/biblio}
\end{document}